\RequirePackage{fix-cm}
\documentclass[11pt,a4paper]{amsart}

\usepackage[margin=2.7cm]{geometry}
\usepackage[utf8]{inputenc}
\usepackage{graphicx}
\usepackage{amsmath,amssymb,amsfonts}
\usepackage{xcolor}
\usepackage{booktabs}
\usepackage{enumitem}
\usepackage{microtype}
\usepackage[numbers,square,sort&compress]{natbib}
\usepackage[colorlinks=true,linkcolor=blue!60!black,citecolor=blue!60!black,urlcolor=blue!60!black]{hyperref}
\usepackage{bookmark}
\newcommand{\Z}{{\mathbb Z}}
\newif\ifmarkchanges
\markchangestrue

\allowdisplaybreaks

\let\pa\partial
\let\na\nabla
\let\eps\varepsilon
\newcommand{\N}{{\mathbb N}}
\newcommand{\R}{{\mathbb R}}
\newcommand{\diver}{\operatorname{div}}
\newcommand{\diverm}{\operatorname{div}_{M}}

\renewcommand{\d}{\,\textnormal{d}}
\newcommand{\dt}{\,\textnormal{d}t}
\newcommand{\ds}{\,\textnormal{d}s}
\newcommand{\dr}{\,\textnormal{d}r}
\newcommand{\dx}{\,\textnormal{d}x}

\newcommand{\dW}{\,\textnormal{d}W}

\newcommand{\dd}{\mathrm{d}}
\newcommand{\dom}{\mathcal{O}}
\newcommand{\E}{{\mathbb E}}
\newcommand{\F}{{\mathbb F}}
\newcommand{\Prob}{{\mathbb P}}
\newcommand{\DD}{\mathrm{D}}
\renewcommand{\L}{{\mathcal L}}

\newcommand{\correction}{\frac{1}{N}}
\newcommand{\sqrtcorrection}{\sqrt{\frac{1}{N}}}
\newcommand{\Ared}{\widetilde{A}}

\theoremstyle{plain}%
\newtheorem{theorem}{Theorem}[section]%
\newtheorem{lemma}[theorem]{Lemma}%
\newtheorem{proposition}[theorem]{Proposition}%
\newtheorem{corollary}[theorem]{Corollary}%
\newtheorem{assumption}[theorem]{Assumption}%
\theoremstyle{remark}%
\newtheorem{remark}[theorem]{Remark}%
\theoremstyle{definition}%
\newtheorem{definition}[theorem]{Definition}%

\begin{document}

\title[Stochastic SKT system]{Fluctuation correction and global solutions for the stochastic Shigesada--Kawasaki--Teramoto system via entropy-based regularization}

\author{Florian Huber}
\email{dr.huber.florian@gmail.com}

\subjclass[2020]{60H15, 35R60, 35K51, 35Q92}
\keywords{Shigesada--Kawasaki--Teramoto system, cross-diffusion, stochastic partial differential equations, martingale solutions, entropy method, It\^o--Stratonovich correction}

\begin{abstract}
We study a stochastic extension of the $n$-species Shigesada--Kawasaki--Teramoto cross-diffusion system, in which a multiplicative noise of intensity $N^{-1/2}$ accounts for the fluctuations of a finite population of size $N$. The noise is motivated by the moderately interacting particle system of Chen, Daus, Holzinger and J\"ungel, and it only partially respects the entropy structure of the deterministic system. Under a detailed-balance condition and for $N$ sufficiently large, we prove the existence of global nonnegative martingale solutions. The equation contains a $\lambda$-modified It\^o--Stratonovich correction, and our result requires $\lambda>\frac12$. We compute the noise-induced entropy production exactly as a function of $\lambda$. For $\lambda<\frac12$ it cannot in general be bounded by the entropy and its dissipation. At the Stratonovich value $\lambda=\frac12$ the leading singular term cancels, but a term generated by the dependence of each species' noise on the other species remains, and it again cannot be bounded. The existence result applies with $\lambda=\frac12+O(1/N)$, so that the equation differs from its Stratonovich form only by a correction with a coefficient of order $N^{-2}$. Numerical experiments with the particle system indicate that $\lambda$ acts as a renormalization constant, determined by the ratio of the correlation length of the noise to the interaction range in the particle system.
\end{abstract}

\maketitle

\section{Introduction}

In their seminal work \citep{SKT79}, Shigesada, Kawasaki, and Teramoto (SKT) introduced a deterministic cross-diffusion system for two competing species, designed to model segregation dynamics in population interactions. The deterministic SKT system generalizes, for an arbitrary number $n$ of species, to
\begin{equation}\label{eq:SKT_intro_deterministic}
  \pa_t u_i - \diver\bigg(\sum_{j=1}^n A_{ij}(u)\na u_j\bigg) = 0
	\quad\mbox{in }\dom,\ t>0,\ i=1,\ldots,n,
\end{equation}
with no-flux boundary conditions, where $u_i$ denotes the population density of the $i$th species, $\dom\subset\R^d$ ($d\ge 1$) is a bounded domain, and
\begin{equation}\label{eq:SKT_intro_matrix_A}
  A_{ij}(u) = \delta_{ij}\bigg(a_{i0} + \sum_{k=1}^n a_{ik}u_k\bigg) + a_{ij}u_i,
	\quad i,j=1,\ldots,n.
\end{equation}
This model can be rigorously derived from nonlocal population models \citep{DiMo21, Mou20}, stochastic interacting particle systems \citep{chen_holzinger_2021_rigorous_SKT}, and finite-state jump Markov models \citep{BMM21, DDD19}. Unlike the original SKT system \citep{SKT79}, we omit deterministic environmental potentials and Lotka--Volterra terms for simplicity. We refer to $a_{i0}$ as the diffusion coefficients, $a_{ii}$ as the self-diffusion coefficients, and $a_{ij}$ for $i\neq j$ as the cross-diffusion coefficients.

A natural stochastic extension of \eqref{eq:SKT_intro_deterministic}, motivated by fluctuating hydrodynamics, accounts for the fact that \eqref{eq:SKT_intro_deterministic} only describes the mean-field, infinite-population limit of an underlying system of finitely many, discretely interacting individuals. Fluctuations around this mean-field limit are expected to be of order $1/\sqrt{N}$ in the total population size $N$ \citep{chen_holzinger_2021_rigorous_SKT}, and a formal computation based on the particle system (we comment on this heuristic below, and prove a linear fluctuation limit for the particle system in forthcoming work) suggests the noise should enter \eqref{eq:SKT_intro_deterministic} multiplicatively, componentwise, with intensity governed by the same coefficients $a_{ij}$ that drive the deterministic diffusion. This leads us to study the stochastic system
\begin{equation}\label{eq:SKT_intro}
  \dd u_i - \diver\bigg(\sum_{j=1}^n A_{ij}(u)\na u_j\bigg) \dt
	= \sqrtcorrection\diver\left( \sigma_{i}(u)\circ_{\lambda}\dW_{i}(t)\right)\quad\mbox{in }\dom,\ t>0,\ i=1,\ldots,n,
\end{equation}
with initial and no-flux boundary conditions
\begin{equation}\label{eq:SKT_intro_bic}
  u_i(0)=u_i^0\quad\mbox{in }\dom, \quad
	\sum_{j=1}^n A_{ij}(u)\na u_j\cdot\nu = 0\quad\mbox{on }\pa\dom,\ t>0,\
	i=1,\ldots,n,
\end{equation}
and noise coefficient
\begin{align}\label{eq:SKT_intro_diffusion_sigma}
   \sigma(u)_{ij}=\delta_{ij}\sqrt{u_{i}\left(a_{i0}+\sum_{k=1}^{n}a_{ik}u_{k}\right)},\quad i,j=1,\dots,n.
\end{align}
Here $\nu$ denotes the exterior unit normal vector to $\pa\dom$, $a_{ij}\ge 0$ for $i=1,\ldots,n$, $j=0,\ldots,n$ are parameters, and $(W_1,\ldots,W_n)$ is an $n$-dimensional, spatially colored Wiener noise. The precise functional setting, including the family $(e^{il}_k)_{ilk}\colon\dom\to\R$ that determines the spatial color of $W$, is given in Section~\ref{sec.main}. The parameter $\frac{1}{N}\in(0,1]$ in \eqref{eq:SKT_intro} is the inverse of the (common) population size of the $n$ species: it quantifies the strength of the fluctuations, and as $N\to\infty$ the noise vanishes and \eqref{eq:SKT_intro} reduces to the deterministic system \eqref{eq:SKT_intro_deterministic}.

The symbol $\circ_\lambda$ denotes a \emph{$\lambda$-modified It{\^o}--Stratonovich correction}: we set
\begin{align*}
    \sqrtcorrection\diver\left( \sigma_{i}(u)\circ_{\lambda}\dW_{i}(t)\right):=\sqrtcorrection\diver\left( \sigma_{i}(u)\dW_{i}(t)\right)+ \correction\lambda  \mathcal{T}(u)_{i}\dt,
\end{align*}
where, for a parameter $\lambda\ge 0$ to be fixed below,
\begin{align}\label{eq:T_intro}
    \mathcal{T}(u)_{i}=\sum_{k,l=1}^{\infty,d}\partial_{x_{l}}\left(\left(\partial_{u_{i}}\sigma\right)(u)_{ii}\,e^{il}_{k}e^{il}_{k}\partial_{x_{l}}\sigma(u)_{ii}+\left(\partial_{u_{i}}\sigma\right)(u)_{ii}\,e^{il}_{k}\partial_{x_{l}}e^{il}_{k}\,\sigma(u)_{ii}\right).
\end{align}
Here and below, the divergences in the noise term and in $\mathcal T$ are understood in the weak sense used in Definition~\ref{def:martingale_solution_SKT}: if $\mathcal T(u)_i=\sum_l\partial_{x_l}G_l$, then $\langle\mathcal T(u)_i,\phi\rangle:=-\sum_l\int_\dom G_l\,\partial_{x_l}\phi\dx$. By the no-flux condition \eqref{eq:noise_bc} below, the weak and the strong form coincide for regular positive $u$. The classical It\^o--Stratonovich correction, obtained from \eqref{eq:T_intro} at $\lambda=\frac12$, converts the Stratonovich noise $\sigma_i(u)\circ\dW_i$ into the equivalent It\^o form $\sigma_i(u)\dW_i + \frac12\mathcal T(u)_i\dt$. The value $\lambda=0$ gives the It\^o equation. Our main theorem requires $\lambda>\frac12$, so it covers neither of the two classical interpretations. Section~\ref{sec.lambda} explains why, and how close to $\lambda=\frac12$ the result reaches:
\begin{itemize}
\item The noise-induced entropy production can be computed exactly as a function of $\lambda$ (Lemma~\ref{lem:P_identity}). It contains a positive term of order $|\nabla\log u_i|^2$ from the It\^o correction of the noise and a negative one from $\lambda\mathcal T$. For $\lambda>\frac12$ the negative term dominates and absorbs all other singular terms, at the cost of non-singular terms in $|\nabla u_j|^2$, which are controlled by Assumption~\ref{Assumption:A4_correction_factor}. At $\lambda=\frac12$ the two terms cancel.
\item For $\lambda<\frac12$, and for $\lambda=\frac12$ whenever $\sigma_i(u)$ depends on another species, the entropy production cannot be bounded by the entropy and its dissipation (Proposition~\ref{prop:lambda_obstruction}, stated for noise with $F^{11}_2\equiv0$ on a slab of the cube). At $\lambda=\frac12$ the obstruction comes only from this cross-dependence. It has no counterpart for the Dean--Kawasaki equation, where the Stratonovich equation with stationary correlated noise on the torus is well posed \citep{fehrman2024well}, see Remark~\ref{rem:DK_comparison}. Thus the restriction $\lambda>\frac12$ is a limitation of the entropy estimate used here. Whether it reflects a property of the equation is open.
\item The parameter may be chosen as $\lambda_N=\frac12+c_\star/N$ (Corollary~\ref{cor:lambda_N}). The resulting equation is the Stratonovich equation plus the correction term with the coefficient $c_\star/N^2$, while the It\^o and the Stratonovich equations differ by a term with the coefficient $1/(2N)$. Our estimates do not show that the additional term is small in a norm, since the available bound on $\nabla\log u$ deteriorates as $N\to\infty$.
\end{itemize}
Concerning modeling, \eqref{eq:SKT_intro} formally reproduces the same Gaussian fluctuations at order $N^{-1/2}$ for every $\lambda$, which formally match the fluctuations of the particle system to be obtained in forthcoming work on $\R^d$ with white noise (Remark~\ref{rem:lambda_modelling}). The particle system therefore does not select a value of $\lambda$ at this order. At order $N^{-1}$, the numerical experiments of Section~\ref{sec.numerics} suggest that $\lambda$ acts as a renormalization constant, fixed by the ratio of the noise scale to the interaction range.

\begin{remark}[Heuristic motivation from the particle system]\label{rem:particle_motivation}\rm
The deterministic system \eqref{eq:SKT_intro_deterministic}--\eqref{eq:SKT_intro_matrix_A} was rigorously derived in \citet{chen_holzinger_2021_rigorous_SKT} as the mean-field limit, as $N\to\infty$, of a system of $N$ moderately interacting diffusions per species. The central limit fluctuations of the corresponding empirical measures around this mean-field limit are governed by a martingale that may be read off from the particle system as follows. Each particle $X^{N}_{k,i}$ of species $i$ carries the diffusion coefficient $\big(2\sigma_i+2\sum_{j=1}^n f_\eta(\tfrac1N\sum_{\ell}B_{ij}^\eta(X_{k,i}-X_{\ell,j}))\big)^{1/2}$ of \citet{chen_holzinger_2021_rigorous_SKT}. Applying It\^o's formula to the rescaled empirical field $\tfrac1{\sqrt N}\sum_{k=1}^N\varphi(X^{N}_{k,i})$ and retaining only its stochastic part yields the fluctuation martingale
\begin{align*}
    M_{i,N}(t,\varphi)=\int_0^t\frac1{\sqrt N}\sum_{k=1}^N\nabla\varphi(X^{N}_{k,i})\Big(2\sigma_i+2\sum_{j=1}^n f_\eta\big(\tfrac1N{\textstyle\sum_{\ell}}B_{ij}^\eta(X_{k,i}-X_{\ell,j})\big)\Big)^{1/2}\dd W_i^k(s).
\end{align*}
Passing formally to the mean-field limit $\tfrac1N\sum_m\delta_{X^{N}_{m,i}}\to u_i$, the quadratic variation of $M_{i,N}$ acquires the structure
\begin{align*}
    \E\left[M_{i}(t,\varphi)M_{k}(s,\psi)\right]=\delta_{ik}\int_{0}^{s\wedge t}\left\langle u_{i},\nabla\varphi\, \nabla\psi\left(2 \sigma_i+2 \sum_{j=1}^n f\left(a_{ij}u_{j}\right)\right)\right\rangle\dr,
\end{align*}
for suitable test functions $\varphi,\psi$, where $u=(u_1,\ldots,u_n)$ solves the mean-field equation \eqref{eq:SKT_intro_deterministic}, see \citet[Lemmas 9--10]{chen_holzinger_2021_rigorous_SKT}. This is exactly the covariance structure produced by the multiplicative noise $\sigma_i(u)\dW_i$ in \eqref{eq:SKT_intro}, and is what motivates our choice of noise coefficient \eqref{eq:SKT_intro_diffusion_sigma}. A rigorous derivation of this fluctuation limit, identifying $M_i$ as a stochastic integral against a colored Wiener process and justifying the passage from the particle system to the SPDE \eqref{eq:SKT_intro}, requires substantial additional work (tightness of the fluctuation martingales, identification of the limit on a suitable weighted Sobolev scale, and control of the mean-field convergence rate). A linear fluctuation limit of this type, on $\R^d$ and with space-time white noise, will be proved in forthcoming work. In the present paper we take \eqref{eq:SKT_intro}--\eqref{eq:SKT_intro_diffusion_sigma} as given and study the resulting SPDE on its own terms.
\end{remark}

Recent developments in fluctuating hydrodynamics (e.g., \citealp{fehrman2023non,fehrman2024well}) provide a complementary perspective on noise terms of this type. As in \citet{fehrman2023non,fehrman2024well}, the noise in \eqref{eq:SKT_intro} must be sufficiently smooth (spatially colored, rather than space-time white) for the SPDE to be well-posed. Otherwise, as for the Dean--Kawasaki equation, the system becomes supercritical in the framework of singular SPDEs \citep{hairer2014theory}.

Previous studies of \eqref{eq:SKT_intro}--\eqref{eq:SKT_intro_matrix_A} have considered multiplicative It{\^o} noise that preserves certain structural properties of the system. For example, the existence of a local pathwise mild solution with $n=2$ was established in \citet[Theorem 4.3]{KuNe20} under the assumption that the diffusion matrix is positive definite and the noise coefficient satisfies standard Lipschitz conditions. A related system, incorporating quadratic rather than linear coefficients, was analyzed under detailed balance and small cross-diffusion coefficients in \citet{DJZ19}. The structural preservation assumed in these works does not hold in our present setting.

The system \eqref{eq:SKT_intro} presents two primary challenges. First, the diffusion matrix $A(u)$ is generally nonsymmetric and lacks positive semidefiniteness, which renders standard semigroup theory inapplicable. In the deterministic case, methods relying on the gradient-flow (or entropy) structure, implicit Euler time discretization, and the Leray-Schauder fixed-point theorem have been successfully employed (see \citealp{ChJu04, ChJu06, Jue15}). However, in the stochastic setting, an explicit Euler scheme is necessary to accommodate the stochastic It\^o integral, precluding the use of entropy estimates. Alternatively, a Galerkin scheme, as detailed in \citet[Theorem 4.2.4]{LiRo15}, reduces the system to a finite-dimensional setting but relies on energy-type ($L^{2}$) estimates. In contrast, our system admits only entropy estimates involving the test function $\log(u_{i})$, which lies outside the Galerkin space. Global martingale solutions for an SKT system with general coefficients satisfying detailed balance were established in \citet{braukhoff2024global}, even in the absence of self-diffusion, by leveraging a novel regularization technique.

The second challenge is that the noise \eqref{eq:SKT_intro_diffusion_sigma}, motivated in Remark~\ref{rem:particle_motivation}, only partially aligns with the entropy structure of the system. Unlike scalar equations, this noise does not conform to the form dictated by the full gradient-flow structure, $\diver(B^{1/2}(u)\dW_t)$, where $B$ is introduced below in \eqref{1.w}--\eqref{1.gradient_flow}. A noise based on this formal gradient-flow structure has instead been used as a fluctuation correction in numerical simulations of multi-component fluid mixtures \citep{donev2015low}. The particle-system heuristic of Remark~\ref{rem:particle_motivation} indicates that, for the SKT system, the physically appropriate fluctuation correction depends instead on a component-wise mobility, which is precisely the source of the second challenge: the entropy dissipation generated by \eqref{eq:SKT_intro_diffusion_sigma} together with its It\^o--Stratonovich correction $\mathcal T$ does not have a definite sign, and controlling it forces the restriction $\lambda>1/2$ in our main theorem, together with a smallness condition on $1/N$ (Assumption \ref{Assumption:A4_correction_factor} below).

Our strategy for proving existence of solutions relies on the system's entropy, or formal gradient-flow, structure, the technical difficulties caused by the noise being addressed in Sections~\ref{sec.approx}--\ref{sec.uniform} below.
We say \eqref{eq:SKT_intro_deterministic} has an entropy structure if there exists a 
function $h:[0,\infty)^n\to[0,\infty)$, called an entropy density, 
such that the deterministic analog of \eqref{eq:SKT_intro} 
can be written in terms of the entropy variables
(or chemical potentials) $w_i=\pa h/\pa u_i$ as
\begin{equation}\label{1.w}
  \pa_t u_i(w) - \diver\bigg(\sum_{j=1}^n B_{ij}(w)\na w_j\bigg) = 0, \quad
	i=1\ldots,n,
\end{equation}
where $w=(w_1,\ldots,w_n)$, $u_i$ is interpreted as a function of $w$, 
and $B(w)=A(u(w))h''(u(w))^{-1}$ 
with $B=(B_{ij})$ is positive semidefinite. Formally, one could read this representation as
\begin{equation}\label{1.gradient_flow}
  \pa_t u_i - \diver\bigg(\sum_{j=1}^n B_{ij}(w)\na D\mathcal{H}(u)_j\bigg) = 0, \quad
	i=1\ldots,n,
\end{equation}
for a suitable entropy functional $\mathcal{H}=\int_{\dom}h(u)\dx$.
For the deterministic analog of \eqref{eq:SKT_intro}, it was 
shown in \citet{CDJ18} that a useful entropy density is given by
\begin{equation}\label{1.h}
  h(u) = \sum_{i=1}^n\pi_i \big(u_i(\log u_i-1)+1\big), \quad u\in[0,\infty)^n,
\end{equation}
where the numbers $\pi_i>0$ are assumed to satisfy the so-called detailed-balance condition:
$\pi_ia_{ij}=\pi_ja_{ji}$ for all $i,j=1,\ldots,n$.
For the Markov chain associated
to $(a_{ij})$, and $(\pi_1,\ldots,\pi_n)$, this condition corresponds to reversibility (\citealp{CDJ18}). 
A formal computation yields that, for the deterministic SKT system under the detailed-balance condition,
\begin{equation}\label{1.ei}
  \frac{\dd}{\dt}\int_\dom h(u)\dx 
	+ 2\int_\dom\sum_{i=1}^n\pi_i\bigg(2a_{i0}|\na\sqrt{u_i}|^2
	+ 2a_{ii}|\na u_i|^2 + \sum_{j\neq i}a_{ij}|\na\sqrt{u_iu_j}|^2\bigg)\dx  = 0.
\end{equation}
A similar expression can be derived in the stochastic setting; see \eqref{eqn:e_sup_ito_estimate_final}. Not only does the entropy structure provide us with bounds on $u_{i}$, it gives $L^2$ estimates for $\na\sqrt{u_i}$ if $a_{i0}>0$ and for
$\na u_i$ if $a_{ii}>0$. Solving the system in terms of the entropy variables $w$ leads to the positivity
of $u_i(w)=\exp(w_i/\pi_i)$.

We will use this entropy structure combined with a regularization scheme, that preserves the entropy estimates and nonnegativity after passing to the 
de-regularization limit. The main idea of this scheme is to ``regularize'' the entropy
variable $w$. To be specific, we perturb the mapping $w\mapsto u(w)$,
and define $Q_\eps(w)=u(w) + \eps L^*Lw$, where 
$L:D(L)\to H$ with domain $D(L)\subset H$ is a 
suitable operator and $L^*$ its dual.
The operator $L$ is chosen in such a way that all elements of $D(L)$ are 
bounded functions, implying that $u(w)$ is well-defined.
It can be shown that the mapping $Q_\eps:D(L)\to D(L)'$ is invertible. We will use its inverse, denoted by $R_\eps:D(L)'\to D(L)$ as a regularization operator for our approximation scheme to \eqref{eq:SKT_intro}. The approximated equation can be written as
\begin{equation}\label{1.approx}
  \dd v(t) = \diver\big(B(R_\eps(v))\na R_\eps(v)\big)\dt
	+ \sqrtcorrection\diver\left(\sigma\big(u(R_\eps(v))\big)\cdot \circ_{\lambda} \dW(t)\right), \quad t>0.
\end{equation}

The existence of a local solution $v^\eps$ to \eqref{1.approx} with suitable
initial and boundary conditions can be shown by applying standard results, e.g.
\citet[Theorem 4.2.4]{LiRo15}. The entropy
inequality for $w^\eps:=R_\eps(v^\eps)$ and $u^\eps:=u(w^\eps)$,
\begin{align*}
  \E&\sup_{0<t<T\wedge\tau_R}\int_\dom h(u^\eps(t))\dx 
	+ \frac{\eps}{2}\E\sup_{0<t<T\wedge\tau_R}\|Lw^\eps(t))\|_{L^2(\dom)}^2 \\
	&{}+ \E\sup_{0<t<T\wedge\tau_R}\int_0^t\int_\dom
	\na w^\eps(s):B(w^\eps(s))\na w^\eps(s)\dx \ds\leq C(u^0,T),
\end{align*}
up to some stopping time $\tau_R>0$ allows us to extend the local solution to a 
global one. 
The entropy inequality provides suitable, uniform in $\eps$, bounds for $u_i^\eps$, which can be further refined by the Gagliardo--Nirenberg inequality to
prove uniform bounds for $u_i^\eps$ in $L^q(0,T;L^q(\dom))$ with $q\ge 2$. 
Such an estimate is crucial to define, for instance, the product $u_i^\eps u_j^\eps$. 
Uniform estimate for $u_i^\eps$ in the 
Sobolev--Slobodeckij space $W^{\alpha,p}(0,T;D(L)')$ for some $\alpha\in(0,1)$ and
$p>1$ with $\alpha p>1$ then allows us to prove the tightness of the laws of
$(u^\eps)$ in some sub-Polish space and to conclude strong convergence in $L^2$
thanks to the Skorokhod--Jakubowski theorem.

This regularization scheme, together with the entropy structure, is the basis for our main result: under the detailed-balance condition and a smallness condition on $1/N$ relative to the coefficients of $A$, the system \eqref{eq:SKT_intro}--\eqref{eq:SKT_intro_diffusion_sigma} admits a global, nonnegative martingale solution; see Theorem~\ref{thm:existence_SKT} below for the precise statement. We emphasise at the outset that this is an existence theory in the perturbative, large-population regime: the smallness of $1/N$ (made precise in Assumption~\ref{Assumption:A4_correction_factor}) is the price paid for a noise that only partially respects the entropy structure, and, for fixed coefficients, it holds for all $N$ sufficiently large (Corollary~\ref{cor:A4_A5_example}).

\emph{Structural and technical hypotheses.} Two hypotheses are structural. The detailed-balance condition \ref{Assumption:A3_SKT_matrix} provides the entropy \eqref{1.h}. For general coefficients without it, no entropy structure is known \citep{CDJ18}. The spatial color of the noise in \ref{Assumption:A5_noise_onb} is needed because white noise makes the equation supercritical. For the entropy estimate used here, the restriction $\lambda>\frac12$ is structural as well, in general (Proposition~\ref{prop:lambda_obstruction}). The smallness of $1/N$ expresses that we work in the large-population regime. Only its scaling in $N$ matters, and its constants are technical. The no-flux condition \eqref{eq:noise_bc} on the noise is needed for the It\^o correction in the entropy estimate. It is the reason why the good term of the correction yields only a weighted bound on $\nabla\log u_i$, which suffices for tightness and for identifying $\mathcal T$ in the limit. All other hypotheses, and all constants, are technical. Remark~\ref{rem:sharpness} gives the details.

\begin{remark}\label{rem:correction_motivation}\rm
The restriction $\lambda>1/2$ in Theorem~\ref{thm:existence_SKT} originates in the entropy estimate of Proposition~\ref{prop.ent}: the It\^o correction of the noise contributes a positive term of order $|\nabla\log u_i|^2$ to the entropy inequality, and the correction $\lambda\mathcal T$ contributes a negative term of the same order (Lemma~\ref{lem:P_identity}). For $\lambda>\frac12$ the negative term dominates, and the remaining terms are absorbed at the cost of a smallness condition on $1/N$. This good term also yields the bound on $\nabla\log u_i$ that gives a meaning to $\mathcal T(u)$ for the limit $u$ (Corollary~\ref{cor:nabla_log_u_estimate}). By Proposition~\ref{prop:lambda_obstruction}, the restriction cannot in general be removed from this estimate. Whether martingale solutions exist for $\lambda\le\frac12$, for instance via a kinetic formulation as in \citet{fehrman2024well}, is open, see Sections~\ref{sec.lambda} and~\ref{sec.outlook}.
\end{remark}

This paper is organized as follows. We present our notation, assumptions, and main
result in Section \ref{sec.main}. The operators needed to define the
approximative scheme are introduced in Section \ref{sec.op}. In Section \ref{sec.approx}, the existence of local solutions to a general approximate equation is proven. Section \ref{sec.uniform} derives estimates uniform in the regularization parameters, in particular the entropy inequality controlling the It\^o--Stratonovich correction term, extends the local solutions to global ones, and passes to the limit in the regularization via tightness arguments, concluding the proof of Theorem~\ref{thm:existence_SKT}. Section~\ref{sec.lambda} analyzes the role of the restriction $\lambda>\frac12$, and Section~\ref{sec.numerics} illustrates it numerically. We close with a discussion of open questions in Section~\ref{sec.outlook}. The proofs of several technical lemmas are collected in Appendix \ref{app:proofs_of_lemmata}.


\section{Notation and main result}\label{sec.main}

\subsection{Notation and stochastic framework}

Let $\dom\subset\R^d$ ($d\ge 1$) be a bounded domain. The Lebesgue and Sobolev
spaces are denoted by $L^p(\dom)$ and $W^{k,p}(\dom)$, respectively, where
$p\in[1,\infty]$, $k\in\N$, and $H^k(\dom)=W^{k,2}(\dom)$. 
For notational simplicity, we generally do not distinguish between $W^{k,p}(\dom)$
and $W^{k,p}(\dom;\R^n)$.
We set $H_N^m(\dom) = \{v\in H^m(\dom):\na v\cdot\nu=0$ on $\pa\dom\}$ for $m\ge 2$.
If $u=(u_1,\ldots,u_n)\in X$
is some vector-valued function in the normed space $X$, we write
$\|u\|_X^2=\sum_{i=1}^n\|u_i\|_X^2$. The inner product of a Hilbert space $H$
is denoted by $(\cdot,\cdot)_H$, and $\langle\cdot,\cdot\rangle_{V',V}$ is the dual
product between the Banach space $V$ and its dual $V'$. If $F:U\to V$
is a Fr\'echet differentiable function between Banach spaces $U$ and $V$, 
we write $\DD F[v]:U\to V$ for its Fr\'echet derivative, for any $v\in U$.

Given two quadratic matrices $A=(A_{ij})$, $B=(B_{ij})\in\R^{n\times n}$, 
$A:B=\sum_{i,j=1}^n A_{ij}B_{ij}$ is the Frobenius matrix product,
$\|A\|_F=(A:A)^{1/2}$ the Frobenius norm of $A$, and $\operatorname{tr}A
=\sum_{i=1}^n A_{ii}$ the trace of $A$. The constants $C>0$
in this paper are generic and their values change from line to line.
Throughout, iterated sums are abbreviated by collecting the summation indices in a single subscript and the corresponding upper limits, in the same order, in the superscript; that is,
\[
  \sum_{i,k=1}^{n,\infty}:=\sum_{i=1}^{n}\sum_{k=1}^{\infty},
  \qquad
  \sum_{i,k,l=1}^{n,\infty,d}:=\sum_{i=1}^{n}\sum_{k=1}^{\infty}\sum_{l=1}^{d},
\]
and analogously for other index combinations. Unless indicated otherwise, the indices $i,j,m$ run over the species $1,\dots,n$, the index $k$ over the noise modes $1,\dots,\infty$, and $l$ over the spatial coordinates $1,\dots,d$.
To avoid confusion, we will also highlight when we use the matrix divergence of a matrix-valued function $\R^{d}\ni x\mapsto M(x)\in \R^{n\times d}$, which is a vector in $\R^{n}$ with entries $\diverm(M)_{i}=\sum_{j=1}^{d}\partial_{x_{j}}M_{ij}$, for $i=1,\dots,n$.

Let $(\Omega,\mathcal{F},\Prob)$ be a probability space endowed with a complete
right-continuous filtration $\F=(\mathcal{F}_t)_{t\ge 0}$ and let 
$H$ be a Hilbert space.
Then $L^0(\Omega;H)$ consists of all measurable functions from $\Omega$ to $H$, and 
$L^2(\Omega;H)$ consists of all $H$-valued random variables $v$ such that
$\E\|v\|_H^2=\int_\Omega\|v(\omega)\|_H^2\Prob(\dd \omega)<\infty$. 
Let $U$ be a separable Hilbert space 
and $(e_k)_{k\in\N}$ be an orthonormal basis of $U$. 
The space of Hilbert--Schmidt operators from $U$ to $H$ is defined by
$$
  \L_2(U;H) = \bigg\{F:U\to H \mbox{ linear, continuous}:
	\sum_{k=1}^\infty\|Fe_k\|_{H}^2 < \infty\bigg\},
$$
and it is endowed with the norm $\|F\|_{\L_2(U;H)} 
= (\sum_{k=1}^\infty\|Fe_k\|_{H}^2)^{1/2}$. 

Let $W=(W_1,\ldots,W_n)$ be an $n$-dimensional $U$-cylindrical Wiener process,
taking values in the separable Hilbert space $U_0\supset U$ 
and adapted to the filtration $\F$. We can write
$W^{ij}=\sum_{k=1}^\infty e^{ij}_k W^{ij}_k$, where 
$(W^{ij}_k)$ is a sequence of independent standard
one-dimensional Brownian motions \citep[Section 4.1.2]{DaZa14}, such that $\E[W^{i_{1}j_{1}}_{k_{1}}(t)W^{i_{2}j_{2}}_{k_{2}}(s)]=\delta_{i_{1}i_{2}}\delta_{j_{1}j_{2}}\delta_{k_{1}k_{2}} t\wedge s$. 
Then $W^{ij}(\omega)\in C^0([0,\infty);U_0)$ for a.e.\ $\omega$ 
\citep[Section 2.5.1]{LiRo15}. The precise regularity required of the family $(e^{il}_k)_{ilk}$, which determines the spatial color of $W$, is given in Assumption \ref{Assumption:A5_noise_onb} below.

\subsection{Assumptions}

We impose the following assumptions:
\begin{assumption}

\begin{enumerate}[label=\normalfont(A\arabic*)]\leavevmode
\item \label{Assumption:A1_domain}  Domain: $\dom\subset\R^d$ ($d\ge 1$) is a bounded domain
with Lipschitz boundary. Let $T>0$ and set $Q_T=\dom\times(0,T)$. 

\item \label{Assumption:A2_initial_condition} Initial datum: $u^0=(u_1^0,\ldots,u_n^0)\in 
L^\infty(\Omega;L^2(\dom;\R^n))$ is a $\mathcal{F}_0$-measurable 
random variable satisfying $u^0(x)\ge 0$ for a.e.\ $x\in\dom$
$\Prob$-a.s.

\item \label{Assumption:A3_SKT_matrix} Diffusion matrix: $a_{ij}\ge 0$ for $i=1,\ldots,n$, $j=0,\ldots,n$ and
there exist $\pi_1,\ldots,\pi_n>0$ such that
$\pi_i a_{ij}=\pi_j a_{ji}$ for all $i,j=1,\ldots,n$ (detailed-balance condition).

\item \label{Assumption:A4_correction_factor} Correction factor: let $\lambda>\frac12$ and $p\ge p_0(d):=\max\{5,8/d\}$ be fixed. Further, let $0<\tau<\frac12$ be given. The correction factor $\correction$ satisfies the following condition, for some $0<\kappa<\frac{2\lambda-1}{4\lambda+3}$ (note that $\frac{2\lambda-1}{4\lambda+3}<\frac12$) and $0<\widetilde\kappa_3\le\frac12$:
\begin{align}
    \frac{4}{\sqrt{N}}\left(\sup_{i,l}\sum_{k=1}^{\infty}\|e^{il}_{k}\|_{L^{\infty}(\dom)}^{2}\right)^{\frac{1}{2}}\sum_{i=1}^{n}\sup_{1\leq j \leq n}a_{ji}\leq \tau,
\end{align}
\begin{align}
     \frac{c_p}{\sqrt N}\left(\sup_{i,l}\sum_{k=1}^{\infty}\|e^{il}_{k}\|_{L^{\infty}(\dom)}^{2}\right)^{\frac{1}{2}}\Big(a^\sharp+\frac1{4\min_{1\le i\le n}a_{i0}}\Big)\leq \tau,
\end{align}
\begin{align}
    \frac4{\sqrt N}\Big(\sup_{i,l}\sum_{k=1}^{\infty}\|e^{il}_{k}\|_{L^{\infty}(\dom)}^{2}\Big)^{\frac12}\leq 4\tau a_{i0}\qquad\text{for every }i=1,\ldots,n,
\end{align}
\begin{align}
    \correction\, K_j(\lambda,\kappa,\widetilde\kappa_3)\left(\sup_{i,l}\sum_{k=1}^{\infty}\|e^{il}_{k}\|_{L^{\infty}}^{2}\right)\leq \tau 2\pi_{j}a_{jj}\qquad\text{for every }j=1,\ldots,n,
\end{align}
where
\begin{align*}
    &K_j(\lambda,\kappa,\widetilde\kappa_3):=\Big(\frac{2\lambda+1}{8\kappa}+\frac18\Big)\beta_j+\frac{\widetilde\kappa_3}4\gamma_j+\frac{\lambda\kappa}4\pi_ja_{jj},\\
&\beta_j:=\sum_{i=1}^n\frac{\pi_ia_{ij}}{a_{i0}}\Big(\sum_{m=1}^na_{im}\Big),\qquad
\gamma_j:=\sum_{i=1}^n\pi_ia_{ij}\Big(\sum_{m=1}^na_{im}\Big),\\
&a^\sharp:=\sum_{i=1}^n\sup_{1\le j\le n}a_{ji},
\end{align*}
and $c_p$ is the constant in the Burkholder--Davis--Gundy inequality $\|\sup_{t\le T}|M(t)|\|_{L^p(\Omega)}\le c_p\|\langle M\rangle_T^{1/2}\|_{L^p(\Omega)}$ for continuous real martingales \citep[Theorem 6.1.2]{LiRo15}.

\item \label{Assumption:A5_noise_onb} There exists a separable Hilbert space $U$ with a continuous embedding $U\hookrightarrow W^{s,2}(\dom)^{n\times d}$ for some $s>d+1$, and the family $(e^{ij}_{k})_{ijk}$ ($i=1,\dots,n$, $j=1,\dots,d$, $k\in\N$) is an orthonormal basis of $U$ satisfying the summability conditions
\begin{align*}
    E_{\infty}:=\sup_{i,l}\sum_{k=1}^{\infty}\|e^{il}_{k}\|_{L^{\infty}(\dom)}^{2}<\infty,\qquad
    E_\infty':=\sup_{i,l}\sum_{k=1}^{\infty}\|\partial_{x_{l}}e^{il}_{k}\|_{L^{\infty}(\dom)}^{2}<\infty,
\end{align*}
and the no-flux condition
\begin{equation}\label{eq:noise_bc}
    e^{il}_{k}\,\nu_l=0\quad\text{on }\pa\dom,\qquad i=1,\dots,n,\ l=1,\dots,d,\ k\in\N,
\end{equation}
where $\nu$ is the outer unit normal, defined a.e.\ on $\pa\dom$. We write $F^{il}_1:=\sum_{k}(e^{il}_k)^2$, $F^{il}_2:=\sum_k e^{il}_k\partial_{x_l}e^{il}_k=\frac12\partial_{x_l}F^{il}_1$ and $F^{il}_3:=\sum_k(\partial_{x_l}e^{il}_k)^2$, following \citet{fehrman2024well}, so that $0\le F^{il}_1\le E_\infty$.
\end{enumerate}
\end{assumption}
\begin{remark}
    We will also refer to $W$ simply as spatially colored noise or $W^{s,2}(\dom)$-regular noise, instead of the $U$-cylindrical Wiener process of Assumption \ref{Assumption:A5_noise_onb}, where $U\hookrightarrow W^{s,2}(\dom)^{n\times d}$. A concrete admissible choice of $U$ and of the basis $(e^{ij}_{k})$ is exhibited in Corollary \ref{cor:A4_A5_example}.

\emph{On the no-flux condition \eqref{eq:noise_bc}.} The noise enters in divergence form. For $\phi\in D(L)$, the weak divergence satisfies
\[
  -\int_\dom\sigma_{ii}(u)e^{il}_k\,\partial_{x_l}\phi\dx=\int_\dom\partial_{x_l}\big(\sigma_{ii}(u)e^{il}_k\big)\phi\dx-\int_{\pa\dom}\sigma_{ii}(u)e^{il}_k\,\nu_l\,\phi\dd S .
\]
The It\^o correction in the entropy estimate (Lemma~\ref{lem.ab}) requires that this functional be given by the function $\partial_{x_l}(\sigma_{ii}(u)e^{il}_k)$, i.e.\ that the boundary integral vanish. The Neumann condition $\na\phi\cdot\nu=0$ built into $D(L)=H^m_N(\dom)$ does not achieve this: the boundary integral contains the trace of $\phi$, not its normal derivative, and this trace is arbitrary. Condition \eqref{eq:noise_bc} makes the boundary integral vanish identically. It is the no-flux condition for the noise, consistent with the no-flux condition of the deterministic part. The flux of $\mathcal T$ in \eqref{eq:T_intro} contains the factor $e^{il}_k$, so the same holds for $\mathcal T$. The price is that $F^{il}_1$ vanishes on the part of $\pa\dom$ where $\nu_l\neq0$. The good term of the correction therefore yields only a bound on $\sqrt{F^{il}_1}\,\partial_{x_l}\log u_i$ (Corollary~\ref{cor:nabla_log_u_estimate}), and this weighted quantity is all that is needed.
\end{remark}
We also introduce the following assumptions, which will be used in an intermediate step.
\begin{assumption}[Auxiliary Assumptions]\leavevmode
\begin{enumerate}[label=\normalfont(A\arabic*)]
\setcounter{enumi}{5}
\item \label{Assumption:A6_noise_coefficient} Multiplicative noise: $\Sigma=(\Sigma_{ij})$ is an $n\times n$ matrix, 
where $\Sigma_{ij}:L^2(\dom;\R^n)\to
\L_2(U;L^2(\dom))$ is $\mathcal{B}(L^2(\dom;\R^n))/$ $\mathcal{B}
(\L_2(U;L^2(\dom)))$-measurable and $\F$-adapted. 
Furthermore, there exists $C_\sigma>0$ such that
for all $u$, $v\in L^2(\dom;\R^n)$,
\begin{align*}
  \|\diver(\Sigma(u)-\Sigma(v))\|_{\L_2(U;L^2(\dom))} 
	&\leq C_\Sigma\|u-v\|_{D(L)}, \\
  \|\diver(\Sigma(v))\|_{\L_2(U;L^2(\dom))} &\leq C_\Sigma(1+\|v\|_{D(L)}).
\end{align*}
\item \label{Assumption:A7_correction_term} Correction term: $\overline{\mathcal{T}}=(\overline{\mathcal{T}}_{i})$ is an $n$ dimensional vector, 
where $\overline{\mathcal{T}}_{i}:D(L)^{\prime}\to
D(L)^{\prime}$ is measurable and $\F$-adapted. 
Furthermore, there exists $C_{\overline{\mathcal{T}}}>0$ such that
for all $v_{1}$, $v_{2}\in D(L)^{\prime}$,
\begin{align*}
  \|\overline{\mathcal{T}}(v_{1})-\overline{\mathcal{T}}(v_{2})\|_{D(L)^{\prime}} 
	&\leq C_{\overline{\mathcal{T}}}\|v_{1}-v_{2}\|_{D(L)^{\prime}}, \\
  \|\overline{\mathcal{T}}(v_{1})\|_{D(L)^{\prime}} &\leq C_{\overline{\mathcal{T}}}(1+\|v_{1}\|_{D(L)^{\prime}}).
\end{align*}
\end{enumerate}
\end{assumption}
\begin{remark}[Discussion of the assumptions]\leavevmode
\begin{itemize}
\item \ref{Assumption:A1_domain}: The Lipschitz regularity of the boundary $\pa\dom$ is needed to apply
the Sobolev and Gagliardo--Nirenberg inequalities.

\item \ref{Assumption:A2_initial_condition}: The regularity condition on $u^0$ can be
weakened to $u^0\in L^p(\Omega;L^2(\dom;\R^n))$ for sufficiently large $p\ge 2$
(only depending on the space dimension). It is used to derive the higher-order moment
estimates.

\item \ref{Assumption:A3_SKT_matrix} The detailed-balance condition is also needed in the deterministic
case to reveal the entropy structure of the system; see \citet{CDJ18}.

\item \ref{Assumption:A4_correction_factor} This is the most restrictive assumption, particularly for ``small'' population sizes $N$; as $N$ grows, the right-hand sides of the four inequalities are unaffected while the left-hand sides shrink like $1/\sqrt N$ and $1/N$, so the assumption becomes easy to satisfy. The restriction $\lambda>\frac12$ is what makes the entropy estimate of Proposition \ref{prop.ent} close (see the discussion in the introduction). By Proposition~\ref{prop:lambda_obstruction} it cannot in general be removed from this estimate. The fourth inequality (see the proof of Proposition \ref{prop.ent} for the origin of $\beta_j,\gamma_j,K_j$) becomes more restrictive as the number of species $n$ grows: under the additional uniform bound $a_{ij}\leq\bar a$ ($1\leq i,j\leq n$), $\beta_j,\gamma_j=O(n^2)$, so that for many species, a correspondingly smaller $1/N$ (i.e., a larger population) is required for our existence theorem to apply. We illustrate in Corollary \ref{cor:A4_A5_example} below that Assumption \ref{Assumption:A4_correction_factor} is non-vacuous: it is satisfied, for any fixed $n$, by an explicit family of noise coefficients once $1/N$ is small enough. The third inequality governs the absorption, into the dissipation term $4\pi_ia_{i0}|\nabla\sqrt{u_i^\eps}|^2$, of the term produced by the Burkholder--Davis--Gundy estimate in the proof of Proposition \ref{prop.ent} (see \eqref{eqn:BDG_stochastic_integral} there). All four inequalities use the same threshold $\tau$, so that the absorption step in the proof of Proposition \ref{prop.ent} produces the single explicit constant $c_\tau=\frac1{1-2\tau}$.
\item \ref{Assumption:A5_noise_onb} is a regularity requirement on the noise leading to the same conditions as in \citet{perkowski_2024weak}.
\item \ref{Assumption:A6_noise_coefficient} \ref{Assumption:A7_correction_term}: The Lipschitz continuity and linear growth of $\Sigma(u)$ and $\overline{\mathcal{T}}(v)$
are used for an auxiliary step, leading to a general statement which is easier citable later on.
\end{itemize}
\end{remark}

\begin{remark}[Essential versus technical hypotheses]\label{rem:sharpness}
It is worth separating the structural hypotheses from the technical ones. The genuinely essential assumptions are the detailed-balance condition~\ref{Assumption:A3_SKT_matrix}, which is precisely what endows the system with the entropy density~\eqref{1.h}, and whose removal is an causes technical difficulties already for the deterministic multi-species system~\citep{CDJ18}, and the spatial colouredness of the noise built into~\ref{Assumption:A5_noise_onb}, without which the equation would be supercritical in the sense of singular SPDEs, as for the Dean--Kawasaki equation~\citep{hairer2014theory,fehrman2024well}. The remaining hypotheses are technical: \ref{Assumption:A2_initial_condition} is only used through the initial entropy $\int_\dom h(u^0)\dx$ and the higher moments, and may be weakened accordingly. The Sobolev index $s>d+1$ in~\ref{Assumption:A5_noise_onb} is dictated by the general-domain eigenbasis of Corollary~\ref{cor:A4_A5_example} (more precisely, by the general-domain argument in the remark following that corollary. On the box it may be lowered to $s>\tfrac d2+1$), the operative requirement being only the two summability conditions there. \ref{Assumption:A6_noise_coefficient}--\ref{Assumption:A7_correction_term} merely record the hypotheses of the abstract existence theorem, verified for the concrete regularised coefficients in Lemma~\ref{lem:is_linear_growth_lipschitz}. The smallness condition~\ref{Assumption:A4_correction_factor} occupies an intermediate position: its constants are far from optimal and only the scaling of the left-hand sides in $1/\sqrt N$ and $1/N$ is essential, but whether existence survives at $1/N=O(1)$ remains open (Section~\ref{sec.outlook}). By Proposition~\ref{prop:lambda_obstruction}, the restriction $\lambda>\frac12$ cannot in general be removed from the entropy estimate. The no-flux condition \eqref{eq:noise_bc} in \ref{Assumption:A5_noise_onb} is structural for the treatment of the It\^o correction (see the remark after Assumption~\ref{Assumption:A5_noise_onb}).
\end{remark}

\begin{lemma}[An explicit admissible noise basis on a box]\label{lem:eigenbasis_summable}
Let $\dom=(0,\ell)^d$ for some $\ell>0$, and let $s>\tfrac d2+1$. For $m=(m_1,\dots,m_d)\in\N_0^d$ let
\begin{align*}
  \phi_m(x):=\prod_{r=1}^d c_{m_r}\cos\!\Big(\tfrac{\pi m_r x_r}{\ell}\Big),\qquad c_0:=\ell^{-1/2},\quad c_j:=(2/\ell)^{1/2}\ (j\ge1),
\end{align*}
be the $L^2(\dom)$-orthonormal Neumann eigenfunctions of $-\Delta$ on $\dom$, with $-\Delta\phi_m=\mu_m\phi_m$, $\mu_m=\pi^2|m|^2/\ell^2$, and $\na\phi_m\cdot\nu=0$ on $\pa\dom$. Then, for all $m\in\N_0^d$ and $r\in\{1,\dots,d\}$,
\begin{align*}
  \|\phi_m\|_{L^\infty(\dom)}\le c_\ell:=(2/\ell)^{d/2},\qquad \|\partial_{x_r}\phi_m\|_{L^\infty(\dom)}\le \tfrac{\pi}{\ell}\,c_\ell\,|m|,
\end{align*}
and the rescaled family $\psi_m:=(1+\mu_m)^{-s/2}\phi_m$ satisfies
\begin{align*}
  \sum_{m\in\N_0^d}\|\psi_m\|_{L^\infty(\dom)}^2&\le c_\ell^2\sum_{m\in\N_0^d}(1+\mu_m)^{-s},\\
  \sum_{m\in\N_0^d}\|\partial_{x_r}\psi_m\|_{L^\infty(\dom)}^2&\le \Big(\tfrac{\pi c_\ell}{\ell}\Big)^2\sum_{m\in\N_0^d}|m|^2(1+\mu_m)^{-s},
\end{align*}
both series being finite whenever $s>\tfrac d2+1$. Moreover, endowing $\operatorname{span}\{\phi_m\}$ with the inner product $\langle f,g\rangle_s:=\sum_m(1+\mu_m)^s\hat f_m\hat g_m$ (where $\hat f_m=(f,\phi_m)_{L^2(\dom)}$) and completing, the family $(\psi_m)_m$ is an orthonormal basis of the resulting Neumann-adapted space $\mathcal H^s_N(\dom):=\{f:\ \langle f,f\rangle_s<\infty\}$, which embeds continuously into $W^{s,2}(\dom)$.
\end{lemma}
\begin{proof}
The eigenpairs $(\mu_m,\phi_m)$ and their $L^2(\dom)$-orthonormality are the classical separation-of-variables computation for the Neumann Laplacian on a box. The two pointwise bounds follow from $|\cos|\le1$, $|\sin|\le1$, $\prod_r|c_{m_r}|\le(2/\ell)^{d/2}$, and $m_r\le|m|$ (for the derivative, $\partial_{x_r}\phi_m$ carries the single extra factor $\pi m_r/\ell$). Since $1+\mu_m=1+\pi^2|m|^2/\ell^2\ge\min(1,\pi^2/\ell^2)(1+|m|^2)$, we have $(1+\mu_m)^{-s}\le C_\ell(1+|m|^2)^{-s}$, so by comparison with $\int_{\R^d}(1+|x|^2)^{-s}\dx$ the first series converges iff $2s>d$ and the second (carrying an extra $|m|^2$) iff $2s>d+2$; both hold for $s>\tfrac d2+1$. Finally $\langle\psi_m,\psi_{m'}\rangle_s=(1+\mu_m)^{-s/2}(1+\mu_{m'})^{-s/2}(1+\mu_m)^s\delta_{mm'}=\delta_{mm'}$, so $(\psi_m)_m$ is orthonormal and, being the image of the $L^2$-complete system $(\phi_m)_m$, complete in $\mathcal H^s_N(\dom)$. The embedding $\mathcal H^s_N(\dom)\hookrightarrow W^{s,2}(\dom)$ is the equivalence of the spectral and Sobolev norms on this subspace.

Fix $l\in\{1,\dots,d\}$. The same bounds and conclusions hold, with the same proof, for the family
\[
  \phi^{(l)}_m(x):=c'_{m_l}\sin\!\Big(\tfrac{\pi m_lx_l}{\ell}\Big)\prod_{r\neq l}c_{m_r}\cos\!\Big(\tfrac{\pi m_rx_r}{\ell}\Big),\qquad m\in\N_0^d,\ m_l\ge1,\quad c'_j:=(2/\ell)^{1/2},
\]
which is the $L^2(\dom)$-orthonormal eigenbasis of $-\Delta$ with Dirichlet conditions on the faces $\{x_l\in\{0,\ell\}\}$ and Neumann conditions on the other faces, and for the rescaled family $\psi^{(l)}_m:=(1+\mu_m)^{-s/2}\phi^{(l)}_m$. The corresponding space $\mathcal H^{s,(l)}(\dom)$ embeds continuously into $W^{s,2}(\dom)$, because its elements extend to $2\ell$-periodic functions in $H^s$ that are odd in $x_l$ and even in the other variables.
\end{proof}

\begin{corollary}[Assumptions \ref{Assumption:A4_correction_factor}--\ref{Assumption:A5_noise_onb} are simultaneously satisfiable]\label{cor:A4_A5_example}
Let $\dom=(0,\ell)^d$ and fix any $s>d+1$ (so in particular $s>\tfrac d2+1$). For each pair $(i,l)$, enumerate the family $(\psi^{(l)}_m)_m$ of Lemma \ref{lem:eigenbasis_summable} as $(e^{il}_k)_{k\in\N}$, acting in the $(i,l)$-component. Then $(e^{ij}_k)_{ijk}$ is an orthonormal basis of $U:=\prod_{i,l}\mathcal H^{s,(l)}(\dom)$, which embeds continuously into $W^{s,2}(\dom)^{n\times d}$, and Lemma \ref{lem:eigenbasis_summable} yields
\begin{align*}
  E_\infty=\sup_{i,l}\sum_{k=1}^\infty\|e^{il}_k\|_{L^\infty(\dom)}^2<\infty, \qquad \sup_{i,l}\sum_{k=1}^\infty\|\partial_{x_l}e^{il}_k\|_{L^\infty(\dom)}^2<\infty,
\end{align*}
Moreover, $e^{il}_k$ vanishes on the faces $\{x_l\in\{0,\ell\}\}$, which are the only faces on which $\nu_l\neq0$, so \eqref{eq:noise_bc} holds. So Assumption \ref{Assumption:A5_noise_onb} holds. Consequently, for any fixed $n\in\N$, $\lambda>\tfrac12$, $0<\kappa<\frac{2\lambda-1}{4\lambda+3}$, $0<\widetilde\kappa_3\le\tfrac12$, and any coefficients $a_{ij},\pi_i$ satisfying Assumption \ref{Assumption:A3_SKT_matrix} with $a_{i0},a_{ii}>0$, the four inequalities of Assumption \ref{Assumption:A4_correction_factor} hold for all $1/N$ sufficiently small, since each has left-hand side proportional to $1/\sqrt N$ or $1/N$ and finite, $N$-independent right-hand side (through the now-finite constant $E_\infty$). In particular, Assumptions \ref{Assumption:A1_domain}--\ref{Assumption:A5_noise_onb} are simultaneously satisfiable, and Theorem \ref{thm:existence_SKT} is non-vacuous.
\end{corollary}

\begin{remark}[General domains]\rm
The box $\dom=(0,\ell)^d$ serves only to make the construction fully explici. It is a bounded Lipschitz domain, as required by Assumption \ref{Assumption:A1_domain}. On a bounded domain with $C^\infty$ boundary the same conclusion holds with the Dirichlet eigenbasis $(\phi_k)_{k\in\N}$ of $-\Delta$, which vanishes on $\pa\dom$ so that \eqref{eq:noise_bc} holds, in place of the product functions: by Weyl's law $\mu_k\sim c_{d,\dom}k^{2/d}$ \citep{netrusov2005weyl,ivrii2016,safarov_vassiliev}, together with the ultracontractivity bound $\|\phi_k\|_{L^\infty}\lesssim(1+\mu_k)^{d/4}$ and its gradient counterpart $\|\na\phi_k\|_{L^\infty}\lesssim(1+\mu_k)^{(d+2)/4}$. Both are obtained from the Gaussian bounds on the Dirichlet heat kernel and its spatial derivative on a smooth domain \citep{ouhabaz2009analysis}. The rescaled family $\psi_k=(1+\mu_k)^{-s/2}\phi_k$ satisfies $\sum_k\|\psi_k\|_{L^\infty}^2\lesssim\sum_k(1+\mu_k)^{d/2-s}$ and $\sum_k\|\na\psi_k\|_{L^\infty}^2\lesssim\sum_k(1+\mu_k)^{(d+2)/2-s}$, finite for $s>d$ and $s>d+1$ respectively. The derivative heat-kernel bound is the point at which $C^\infty$ (rather than merely Lipschitz) regularity is used. Since one admissible configuration already suffices for non-vacuousness, we do not pursue the general case.
\end{remark}

\subsection{Main results}

Let $T>0$, $m\in\N$ with $m>d/2+1$ and $D(L)=H_N^m(\dom)$.

\begin{definition}[Martingale solution]\label{def:martingale_solution_SKT}
Let $\lambda> \frac12$ be fixed as in Assumption \ref{Assumption:A4_correction_factor}. A {\em martingale solution} to \eqref{eq:SKT_intro}--\eqref{eq:SKT_intro_matrix_A} is the triple
$(\widetilde U,\widetilde W,\widetilde u)$, together with a field $\mathcal G=(\mathcal G_{il})_{i,l}$, such that $\widetilde U
=(\widetilde \Omega,\widetilde{\mathcal F},\widetilde\Prob,\widetilde\F)$
is a stochastic basis with filtration 
$\widetilde \F=(\widetilde{\mathcal F}_t)_{t\ge 0}$, 
$\widetilde W$ is an $n$-dimensional, spatially colored Wiener process ($W^{s,2}(\dom)$ cylindrical, where $s>d+1$, as in Assumption \ref{Assumption:A5_noise_onb}),
and $\widetilde u=(\widetilde u_1,\ldots,\widetilde u_n)$ 
is a continuous $D(L)'$-valued $\widetilde{\F}$-adapted process such that
$\widetilde u_i\ge 0$ a.e.\ in $\dom\times(0,T)$ $\widetilde\Prob$-a.s., 
\begin{equation}\label{2.regul}
\begin{split}
  \widetilde u_i&\in L^0(\widetilde \Omega;C^0([0,T];D(L)'))\cap L^0(\widetilde \Omega;L^2(0,T;H^1(\dom))),\\
  \mathcal G_{il}&\in L^0(\widetilde \Omega;L^2(0,T;L^2(\dom))),\quad l=1,\dots,d,
\end{split}
\end{equation}
$\mathcal G$ is $\widetilde\F$-adapted, $\mathcal G_{il}=0$ on $\{F^{il}_1=0\}$, and
\begin{equation}\label{2.Gcompat}
  \widetilde u_i\,\mathcal G_{il}=\sqrt{F^{il}_1}\,\partial_{x_l}\widetilde u_i\qquad\text{a.e.\ in }Q_T,\ \widetilde\Prob\text{-a.s.}
\end{equation}
On $\{\widetilde u_i>0\}$, \eqref{2.Gcompat} states that $\mathcal G_{il}=\sqrt{F^{il}_1}\,\partial_{x_l}\log\widetilde u_i$. On the "vacuum set" $\{\widetilde u_i=0\}$, $\mathcal G_{il}$ is not determined by $\widetilde u_i$ and is part of the solution.
the law of $\widetilde u_i(0)$ is the same as for $u_i^0$,
and for all $\phi\in D(L)$, $t\in(0,T)$, $i=1,\ldots,n$, $\widetilde \Prob$-a.s.,
\begin{align}\label{eqn:martingale_solution_weak}
  \langle\widetilde u_i(t),\phi\rangle_{D(L)',D(L)} 
	&= \langle\widetilde u_i(0),\phi\rangle_{D(L)',D(L)}
	- \sum_{j=1}^n\int_0^t\int_\dom A_{ij}(\widetilde u(s))
	\na\widetilde u_j(s)\cdot\na\phi\dx \ds\\
  &\phantom{xx}{}+ \sqrtcorrection \sum_{j=1}^n\int_\dom\bigg(\int_0^t\sigma_{ij}(\widetilde u(s))
	\d \widetilde W_j(s)\bigg)\cdot \nabla \phi\dx  \nonumber\\
    &\phantom{xx}+\correction\lambda \int_0^t\langle\mathcal{T}(\widetilde u(s))_{i},\phi\rangle\ds \nonumber,
\end{align}
where the $n\times 1$ dimensional $\lambda$-modified It{\^o}-Stratonovich correction $\mathcal T(\widetilde u(s))_i$ is given by \eqref{eq:T_intro}, now evaluated at $\widetilde u(s)$ (here $(\partial_{u_i}\sigma)(\widetilde u(s))_{ii}$ denotes the function $u\mapsto\sigma(u)_{ii}$ differentiated with respect to its $i$th argument $u_i$, with the resulting function subsequently evaluated at $u=\widetilde u(s)$; this is distinct from the spatial derivative $\partial_{x_l}$ appearing elsewhere in the same expression). The correction is understood in the weak form \eqref{eqn:T_closed_form} of Section~\ref{sec.conv1}, in which $F^{il}_1\partial_{x_l}\log\widetilde u_i$ is replaced by $\sqrt{F^{il}_1}\,\mathcal G_{il}$:
\[
  \langle\mathcal T(\widetilde u)_i,\phi\rangle=-\sum_{l=1}^d\int_\dom\Big[\frac14\vartheta_i(\widetilde u)\sqrt{F^{il}_1}\,\mathcal G_{il}+\frac{\vartheta_i(\widetilde u)}{4\Ared_i(\widetilde u)}\partial_{x_l}\Ared_i(\widetilde u)F^{il}_1+\frac12\vartheta_i(\widetilde u)F^{il}_2\Big]\partial_{x_l}\phi\dx,
\]
with $\vartheta_i(u):=\Ared_i(u)+a_{ii}u_i$. In the existence proof, $\mathcal G_{il}$ is the a.s.\ weak limit of $\sqrt{F^{il}_1}\partial_{x_l}\log u_i^\eps$ obtained in Section~\ref{sec.conv1}. Each term is integrable, since $\vartheta_i(\widetilde u)\in L^2(0,T;L^2(\dom))$, $\Ared_i(\widetilde u)\ge a_{i0}>0$ and $\widetilde u\in L^2(0,T;H^1(\dom))$.
\end{definition}

Our main results read as follows.

\begin{theorem}[Existence for the SKT model]\label{thm:existence_SKT}\sloppy
Let $\lambda>\frac12$, and let Assumptions \ref{Assumption:A1_domain}--\ref{Assumption:A5_noise_onb} be satisfied (with this $\lambda$) and $a_{i0}, a_{ii}>0$ for $i=1,\ldots,n$.
Then \eqref{eq:SKT_intro}--\eqref{eq:SKT_intro_diffusion_sigma}
has a global nonnegative martingale solution in the sense of Definition \ref{def:martingale_solution_SKT}.
\end{theorem}
\begin{remark}
    Equation \eqref{eqn:martingale_solution_weak} can also be written as
    \begin{align*}
  \langle\widetilde u_i(t),\phi\rangle_{D(L)',D(L)} 
	&= \langle\widetilde u_i(0),\phi\rangle_{D(L)',D(L)}
	- \int_0^t\int_\dom \widetilde u_i(s)
	\bigg(a_{i0}+\sum_{j=1}^n a_{ij}\widetilde u_j(s)\bigg)\Delta\phi\dx \ds\\
  &\phantom{xx}{}+ \sum_{j=1}^n\int_\dom\bigg(\int_0^t\sigma_{ij}(\widetilde u(s))
	\circ_{\lambda}\dd\widetilde W_j(s)\bigg)\cdot \nabla \phi\dx .
\end{align*}
for all $\phi\in D(L)\cap W^{2,\infty}(\dom)$. 
\end{remark}

\begin{remark}[Nonnegativity of the solution]\rm
The a.s.\ nonnegativity of the population densities is a consequence of the
entropy structure, since the approximate densities $u_i^\eps$ satisfy
$u_i^\eps = u_i(R_\eps(v^\eps)) = \exp(R_\eps(v^\eps)/\pi_i)>0$ a.e.\ in $Q_T$.
\end{remark}

\section{Operator setup}\label{sec.op}

In this section, we review the operators' properties in the approximate scheme laid out in \citet{braukhoff2024global}. 

\subsection{Definition of the connection operator \texorpdfstring{$L$}{L}}

We define an operator $L$ that ``connects'' two Hilbert spaces $V$ and $H$ satisfying
$V \subset H$. This abstract operator defines a regularization operator
that ``lifts'' the dual space $V'$ to $V$. 

\begin{proposition}[Operator $L$]\label{prop.L}
Let $V$ and $H$ be separable Hilbert spaces 
such that the embedding $V\hookrightarrow H$ 
is continuous and dense. Then there exists a bounded, self-adjoint, positive operator
$L:D(L)\to H$ with domain $D(L)=V$.
Moreover, it holds for $L$ and its dual operator $L^*:H\to V'$ (we identify
$H$ and its dual $H'$) that, for some $0<c<1$,
\begin{equation}\label{3.L}
  c\|v\|_V \leq \|L(v)\|_H = \|v\|_V, \quad \|L^*(w)\|_{V'}\leq \|w\|_H, 
	\quad v\in V,\ w\in H.
\end{equation}
\end{proposition}
\begin{proof}[Construction sketch]
The details are in \citet[Sec.~2.7]{huber2022stochastic} (see also \citet[Thm.~1.12]{KPS82}). We sketch the construction. By the Riesz representation theorem applied to the bounded linear maps $v\mapsto(v,w)_H$ on $V$, there is a bounded, injective operator $G\colon H\to V$, self-adjoint and positive as an operator on $H$, with $(v,w)_H=(v,Gw)_V$ for all $v\in V,\ w\in H$. Its inverse $\Lambda:=G^{-1}$ (with domain $\operatorname{ran}(G)\subset V$) is densely defined, self-adjoint and positive on $H$ and satisfies $(v,\Lambda w)_H=(v,w)_V$. Setting $L:=\Lambda^{1/2}$ yields a positive self-adjoint operator with $D(L)=V$ for which, by the spectral calculus, $\|Lv\|_H^2=(\Lambda v,v)_H=(v,v)_V=\|v\|_V^2$, i.e.\ $\|Lv\|_H=\|v\|_V$. The equivalence of $\|\cdot\|_V$ with the graph norm of $L$ gives the lower bound $c\|v\|_V\le\|Lv\|_H$, and $\|L^*w\|_{V'}=\sup_{\|v\|_V=1}|(w,Lv)_H|\le\|w\|_H$ the last bound in \eqref{3.L}. In the application below ($V=H^m_N(\dom)$, $H=L^2(\dom)$), $\Lambda$ is the self-adjoint elliptic operator associated with the $H^m_N$-inner product and $L=\Lambda^{1/2}$ is a corresponding order-$m$ Neumann operator.
\end{proof}

We apply Proposition \ref{prop.L} to $V=H^m_N(\dom)$ and $H=L^2(\dom)$, recalling
that $H^m_N(\dom)=\{v\in H^m(\dom):\na v\cdot\nu=0$ on $\pa\dom\}$ and $m>d/2+1$.
Then, by Sobolev's embedding, $D(L)\hookrightarrow W^{1,\infty}(\dom)$.
Note the following two properties, that will be used later:
\begin{align}\label{3.LL}
  & \|L^*L(v)\|_{V'}\leq \|v\|_V, \quad \|L^*(w)\|_{V'}\leq \|w\|_H
	\quad\mbox{for all }v\in V,\ w\in H. 
\end{align}

\begin{lemma}[Operator $L^{-1}$]\label{lem.L1}
Let $L^{-1}:\operatorname{ran}(L)\to D(L)$ be the inverse of $L$ and let
$D(L^{-1}):=\overline{D(L)}$ be the closure of $D(L)$ 
with respect to $\|L^{-1}(\cdot)\|_H$.
Then $D(L)'$ is isometric to $D(L^{-1})$.
In particular, it holds that
$(L^{-1}(v),L^{-1}(w))_{H} = (v,w)_{D(L)'}$ for all $v$, $w\in D(L)'$.
\end{lemma}

\begin{lemma}[Operator $u$]\label{lem.u}
The mapping $u:=(h')^{-1}$ from $D(L)$ to $L^\infty(\dom)$ is Fr\'echet 
differentiable and, as a mapping from $D(L)$ to $D(L)'$, monotone.
\end{lemma}

\subsection{Definition of the regularization operator \texorpdfstring{$R_\eps$}{R-epsilon}}

First, we define another operator, denoted by $Q_{\eps}$, that maps $D(L)$ to $D(L)'$. Its inverse
is the desired regularization operator.

\begin{lemma}[Operator $Q_\eps$]\label{lem.Qeps}
Let $\eps>0$ and define $Q_\eps:D(L)\to D(L)'$ by $Q_\eps(w)=u(w)+\eps L^*Lw$, where
$w\in D(L)$. Then $Q_\eps$ is Fr\'echet differentiable, strongly monotone,
coercive, and invertible. 
Its Fr\'echet derivative $\DD Q_\eps[w](\xi)=u'(w)\xi + \eps L^*L\xi$
for $w$, $\xi\in D(L)$ is continuous, strongly monotone, coercive, and
invertible.
\end{lemma}

Lemma \ref{lem.Qeps} yields the existence of the inverse of $Q_\eps$, which we denote by $R_\eps:=Q_\eps^{-1}:D(L)'\to D(L)$. This operator will be used to regularize our equation and has the following properties.

\begin{lemma}[Operator $R_\eps$]\label{lem.Reps}
The operator $R_\eps:D(L)'\to D(L)$ is Fr\'echet differentiable and strictly monotone.
In particular, it is Lipschitz continuous with Lipschitz constant
$C/\eps$, where $C>0$ does not depend on $\eps$. The Fr\'echet derivative satisfies the relation
$$
  \DD R_\eps[v] = (\DD Q_\eps[R_\eps(v)])^{-1}
	= (u'(R_\eps(v))+\eps L^*L)^{-1} \quad\mbox{for }v\in D(L)',
$$
and it is Lipschitz continuous with constant $C/\eps$, satisfying
$\|\DD R_\eps[v](\xi)\|_{D(L)}\leq \eps^{-1}C\|\xi\|_{D(L)'}$
for $v$, $\xi\in D(L)'$.
\end{lemma}

The operator $R_{\eps}$ can be interpreted as a smoothing mechanism acting on the entropy variables. Its role is to restore sufficient regularity to make the stochastic equation well-posed, while preserving the key structural properties of the original system.
\section{Existence of approximate solutions}\label{sec.approx}

Analogously to  \citet{braukhoff2024global}, we regularize equation \eqref{eq:SKT_intro}--\eqref{eq:SKT_intro_diffusion_sigma} on the level of the so-called entropy variables by introducing the regularized
variable $R_\eps(v)$ for $v\in D(L)'$. 
Setting $v = u(R_\eps(v)) + \eps L^*LR_\eps(v)$, we
consider the regularized problem
\begin{align}\label{eqn:approximate_equation_corr}
  & \dd v = \diverm\big(B(R_\eps(v))\na R_\eps(v)\big)\dt\\
	&\phantom{\dd v ={}}+ \sqrtcorrection\diverm\left(\Sigma\big(u(R_\eps(v))\big) \dW(t)\right)+\correction\lambda \overline{\mathcal{T}}(u(R_\eps(v)))\dt,\\
    &\quad\mbox{in }\dom,\ t\in[0,T\wedge\tau), \nonumber\\
	& v(0) = u^0\quad\mbox{in }\dom, \quad 
	\na R_\eps(v)\cdot\nu = 0 \quad\mbox{on }
	\pa\dom,\ t>0, \label{regularized_corr.bic}
\end{align}
recalling that $B(w)=A(u(w))h''(u(w))^{-1}$ for $w\in\R^n$.

$\overline{\mathcal{T}}(v)$ denotes an abstract It{\^o}-Stratonovich type correction term, satisfying Assumption \ref{Assumption:A7_correction_term}. We use an overline throughout this section to distinguish this abstract placeholder from the concrete correction term $\mathcal T$ of \eqref{eq:T_intro}, which we instantiate, in regularized form, in Section \ref{sec.reg_sigma} below.
Let us clarify the notion of solution for the equation \eqref{eqn:approximate_equation_corr}--\eqref{regularized_corr.bic}.
Let $T>0$, let $\tau$ be an $\F$-adapted stopping time,
and let $v$ be a continuous, $D(L)'$-valued, $\F$-adapted process. We call $(v,\tau)$
a local strong solution to \eqref{eqn:approximate_equation_corr}--\eqref{regularized_corr.bic} if
$$
  v(\omega,\cdot,\cdot)\in L^2([0,T\wedge\tau(\omega));D(L)')\cap
	C^0([0,T\wedge\tau(\omega));D(L)')
$$
for a.e.\ $\omega\in\Omega$ and for all $t\in[0,T\wedge\tau)$,
\begin{align}
   v(t) &= v(0) + \int_0^t\diverm\big(B(R_\eps(v(s)))\na R_\eps(v(s))\big)\dd s\\
 &\phantom{v(t) ={}}+ \sqrtcorrection\int_0^t\diverm\left(\Sigma\big(u(R_\eps(v(s))\big) \dW(s)\right) \label{4.defsol} \\
 &\phantom{xx}+\correction \lambda \int_0^t\overline{\mathcal{T}}(v) \ds ,\nonumber \\
	& \na R_\eps(v)\cdot\nu = 0\quad\mbox{on }\pa\dom\quad
	\Prob\mbox{-a.s.} \label{4.bc}
\end{align}
 $R_\eps$ can be shown to be strongly measurable and, if $v$ is 
progressively measurable, also progressively measurable. For progressively measurable $w$, $u(w)$ inherits this property, and if $v\in C^0([0,T];D(L)')$,
we have $R_\eps(v)\in C^0([0,T];D(L))$ and $u(R_\eps(v))\in L^\infty(Q_T)$. 
Finally, if $v\in L^0(\Omega;L^p(0,T;D(L)'))$ for $1\leq p\le\infty$, then
$\diver(B(u(R_\eps(v)))\na R_\eps(v))\in L^0(\Omega;L^p(0,T;D(L)')))$.
Therefore, the expressions in \eqref{4.defsol} are well defined.
We call a local strong solution {\em global strong solution} 
if $\Prob(\tau=\infty)=1$.
Given $t>0$ and $v\in L^2(\Omega;C^0([0,t];D(L)'))$, we introduce the 
stopping time
$$
  \tau_R := \inf\{s\in[0,t]:\|v(s)\|_{D(L)'}>R\}\quad\mbox{for }R>0,
$$ 
which is $\Prob$-a.s. positive. This claim was already verified in \citet{braukhoff2024global}.

Before we go into further specifics of the system, we want to state a general existence theorem for equations of the form \eqref{eqn:approximate_equation_corr}--\eqref{regularized_corr.bic}.
\begin{theorem}[Existence of approximate solutions]\label{thm:approx_solution_general}
Let Assumptions \ref{Assumption:A1_domain}--\ref{Assumption:A7_correction_term} be satisfied and 
let $\eps>0$, $R>0$.
Then problem \eqref{eqn:approximate_equation_corr}--\eqref{regularized_corr.bic} has a unique local solution 
$(v^\eps,\tau_R)$.
\end{theorem}
\begin{proof}
The proof is based on \citet[Theorem 13]{braukhoff2024global}, which in turn uses \citet[Theorem 4.2.4, Proposition 4.1.4]{LiRo15}. The necessary conditions for the operator
$M:D(L)'\to D(L)'$, $M(v) := \diver(B(R_\eps(v))\na R_\eps(v))$ are verified in the proof of \citet[Theorem 13]{braukhoff2024global}. These conditions concern only the drift operator $M$, which is built from $B$, $R_\eps$ and $L$ exactly as in \citet{braukhoff2024global} and is therefore unaffected by the present choice of noise: the different noise coefficient $\Sigma$ and correction term $\overline{\mathcal T}$ enter the fixed-point argument solely through their own Lipschitz-continuity and linear-growth bounds, which are supplied here by Assumptions~\ref{Assumption:A6_noise_coefficient}--\ref{Assumption:A7_correction_term} (verified for the concrete $\sigma_\delta$ and $\mathcal T$ in Lemma~\ref{lem:is_linear_growth_lipschitz}). Hence the verification of the conditions on $M$ carries over verbatim. The Lipschitz continuity and linear growth of $\Sigma$ guaranteed by Assumption \ref{Assumption:A6_noise_coefficient}, together with the Lipschitz continuity of $\sigma$ and Lemma \ref{lem.u}, now yield that for 
$v$, $\bar v\in D(L)'$ with  $\|v\|_{D(L)'}\leq K$ and $\|\bar v\|_{D(L)'}\leq K$, 
\begin{align*}
  \|&\diverm\left(\Sigma(u(R_\eps(v)))-\Sigma(u(R_\eps(\bar v)))\right)\|_{\mathcal{L}_2(U;D(L)')}\\
  &\quad\leq C\|\Sigma(u(R_\eps(v)))-\Sigma(u(R_\eps(\bar v)))
	\|_{\mathcal{L}_2(U;L^2(\dom))} \\
	&\leq C(K)\|u(R_\eps(v)))-u(R_\eps(\bar v))\|_{L^2(\dom)} \\
	&\leq C(K)\|R_\eps(v)-R_\eps(\bar v)\|_{D(L)}
	\leq C(\eps,K)\|v-\bar v\|_{D(L)'},
\end{align*}
where $C(K)$ also depends on the $L^\infty(\dom)$ norms of $u'(R_\eps(v))$
and $u'(R_\eps(\bar v))$. Assumption \ref{Assumption:A7_correction_term} guarantees that
\begin{align*}
  \|&\overline{\mathcal{T}}(v)-\overline{\mathcal{T}}(\bar v)\|_{\mathcal{L}_2(U;D(L)')}
	\leq C(\eps,K)\|v-\bar v\|_{D(L)'}.
\end{align*}

Hence, the assumptions of \citet[Theorem 4.2.4]{LiRo15}
are satisfied in the ball $\{v\in D(L)':\|v\|_{D(L)'}\leq K\}$. 
These {\em local} bounds are sufficient to
conclude the existence of a {\em local} solution $v$ up to the stopping time $\tau_R$. 
The boundary conditions follow from $R_\eps(v)\in D(L)=H^m_N(\dom)$ and the
definition of the space $H^m_N(\dom)$.
\end{proof}

\subsection{The regularized noise coefficient}\label{sec.reg_sigma}
Motivated by Remark~\ref{rem:particle_motivation}, we set
\begin{align*}
    \sigma(u(R_{\eps}(v^{\eps})))_{ij}=\delta_{ij}\sqrt{u_{i}\left(a_{i0}+\sum_{k=1}^{n}a_{ik}u^{\eps}_{k}\right)},
\end{align*} ($i,j=1,\dots,n$), with $\delta_{ij}$ being the Kronecker delta. For notational convenience, we write $u^{\eps}$ instead of $u(R_{\eps}(v^{\eps}))$.  To apply the results of the previous section, we regularize the components of $\sigma$ in the following way:
Set $\sigma_{\delta}(u^{\eps})_{ii}:= g_{\delta}\left(u^{\eps}_{i}\left(a_{i0}+\sum_{k=1}^{n}a_{ik}u^{\eps}_{k}\right)\right)$, where, for $\delta>0$,
\begin{align*}
    g_{\delta}(x):=\sqrt{x+\delta}-\sqrt\delta,\qquad x\geq0,
\end{align*}
and $g_\delta(x):=0$ for $x<0$ (this case does not occur along solutions, since $u^\eps>0$, but is included for definiteness). For notational convenience, let us define $\Ared_{i}(u^{\eps}):=a_{i0}+\sum_{k=1}^{n}a_{ik}u^{\eps}_{k}$. Unlike the piecewise regularizations used e.g.\ in \citet{perkowski_2024weak} for the Dean--Kawasaki equation, $g_\delta$ is a single smooth expression on $[0,\infty)$, with
\[
g_\delta'(x)=\frac1{2\sqrt{x+\delta}},\qquad x\geq0,
\]
so that the (partial) derivatives of $g_{\delta}(u^{\eps}_{i}\Ared_{i}(u^{\eps}))$, with respect to $u^{\eps}_{i}$ and $u^{\eps}_{j}$ ($j\neq i$), are given by
\begin{align*}
    \partial_{u^{\eps}_{i}}g_{\delta}(u^{\eps}_{i}\Ared_{i}(u^{\eps}))= \frac{\Ared_{i}(u^{\eps})+a_{ii}u^{\eps}_{i}}{2\sqrt{u^{\eps}_{i}\Ared_{i}(u^{\eps})+\delta}},\qquad
    \partial_{u^{\eps}_{j}}g_{\delta}(u^{\eps}_{i}\Ared_{i}(u^{\eps}))= \frac{a_{ij}u^{\eps}_{i}}{2\sqrt{u^{\eps}_{i}\Ared_{i}(u^{\eps})+\delta}}.
\end{align*}
It can be verified that $\left(\sigma_{\delta}(u^{\eps})\right)_{ii}$ is locally Lipschitz continuous in $u^{\eps}$, for every $i=1,\dots,n$, due to the Lipschitz continuity of $g_{\delta}$ and the local Lipschitz continuity of its argument. Further, $g_\delta\in C^\infty([0,\infty))$, and, for every $x>0$ and $\delta>0$,
\begin{align}\label{eqn:sqrt_approximation_properties}
\left\|\partial_{x}g_{\delta}\right\|_{L^{\infty}([0,\infty))} = \frac1{2\sqrt{\delta}}, \quad \partial_{x} g_{\delta}(x) \leq \frac1{2\sqrt{x}},\quad
 0\leq g_{\delta}(x) \leq \sqrt{x},\quad 0\leq \sqrt x-g_\delta(x)\leq\sqrt\delta,
\end{align}
the first three following directly from the formulas above (the last since $0\leq\sqrt{x+\delta}-\sqrt x\leq\sqrt\delta$, by subadditivity of $\sqrt\cdot$), the last of which also gives the uniform convergence $g_\delta\to\sqrt\cdot$ on $[0,\infty)$ as $\delta\to0$. In addition,
\begin{align}\label{eqn:sqrt_approximation_properties2}
g_\delta'(x)^2x=\frac{x}{4(x+\delta)}<\frac14,\qquad g_\delta'(x)g_\delta(x)=\frac12\Big(1-\frac{\sqrt\delta}{\sqrt{x+\delta}}\Big)<\frac12,
\end{align}
 for every $x>0,\delta>0$, both approaching their respective bounds as $\delta\to0$ or $x\to\infty$, but never attaining them for $\delta>0$ fixed; this sharpens the corresponding (only asymptotic, $x\geq\delta$) identities available for a piecewise construction, and is used without further comment in the proof of Lemma~\ref{lem:is_correction_entropy_estimate} below.
Our goal now is to obtain a local solution $(v^{\eps},\tau)$ of
\begin{align}\label{eqn:approximate_equation_strat}
  & \dd v^{\eps} = \diverm\big(B(R_\eps(v^{\eps}))\na R_\eps(v^{\eps})\big)\dt
	+ \sqrtcorrection\diverm\left(\sigma_{\delta}\big(u(R_\eps(v^{\eps}))\big) \dW(t)\right) +\correction \lambda \mathcal{T}(v^{\eps}) \dt,\\
    &\quad\mbox{in }\dom,\ t\in[0,T\wedge\tau), \nonumber\\
	& v^{\eps}(0) = u^0\quad\mbox{in }\dom, \quad 
	\na R_\eps(v^{\eps})\cdot\nu = 0 \quad\mbox{on }
	\pa\dom,\ t>0, \label{regularized_stratonovich.bic}
\end{align}
In addition, we aim to obtain an entropy estimate, allowing us to extend the solution globally in time and to pass to the limit, as the regularization(s) vanish.  
 The components of the $n\times d$ matrix $W$ are given by $W^{ij}= \sum_{k=1}^{\infty}e^{ij}_{k} W^{ij}_{k}(s)$ ($i=1,\dots,n$ and $j=1,\dots,d$). The entries of the $n\times 1$ dimensional \emph{modified} It{\^o}-Stratonovich correction are given by
\begin{align}\label{eqn:IS_correction}
    \mathcal{T}(v^{\eps})_{i}=\sum_{k,l=1}^{\infty,d}\partial_{x_{l}}\left(\left(\partial_{u_{i}}\sigma\right)(u^{\eps})_{ii}\,e^{il}_{k}e^{il}_{k}\partial_{x_{l}}\left(\sigma(u^{\eps})\right)_{ii}+\left(\partial_{u_{i}}\sigma\right)(u^{\eps})_{ii}\,e^{il}_{k}\partial_{x_{l}}e^{il}_{k}\left(\sigma(u^{\eps})\right)_{ii}\right).
\end{align}
\begin{remark}
    Note that $\mathcal{T}$ differs from the natural It{\^o}-Stratonovich correction, given by
    \begin{align*}
    \mathcal{T}(v^{\eps})_{i}=\sum_{k,l=1}^{\infty,d}\partial_{x_{l}}\left(\left(\partial_{v_{i}^{\eps}}\sigma\right)(u^{\eps})_{ii}\,e^{il}_{k}e^{il}_{k}\partial_{x_{l}}\left(\sigma(u^{\eps})\right)_{ii}+\left(\partial_{v_{i}^{\eps}}\sigma\right)(u^{\eps})_{ii}\,e^{il}_{k}\partial_{x_{l}}e^{il}_{k}\left(\sigma(u^{\eps})\right)_{ii}\right).
\end{align*}
This term includes functional derivatives in the direction of $v^{\eps}$. However, for our analysis, working with \eqref{eqn:IS_correction} is more convenient, as this term is already the appropriate correction for the limiting equation.
\end{remark}

Note that the correction term will not include any regularization.
We will not highlight the second regularization in $v^{\eps}$ or $u^{\eps}$, as $\delta$ can be chosen as a function of $\eps$, and both limits can be performed at the same time. In the following estimates, however, we want to keep $\delta$ as a separate parameter as this increases the readability of our arguments.
\begin{lemma}\label{lem:is_linear_growth_lipschitz}
    Let $v,v^{1},v^{2}\in D(L)^{\prime}$, such that $\|v\|_{D(L)^{\prime}},\|v^{1}\|_{D(L)^{\prime}}, \|v^{2}\|_{D(L)^{\prime}}\leq R$, for some $R>0$, then there exist a constant $C_{R,\eps}>0$, such that
    \begin{align*}
        &\|\mathcal{T}(v)\|_{D(L)^{\prime}}^{2}\leq C_{R,\eps}\left(1+\|v\|_{D(L)^{\prime}}^{2}\right)\\
        &\|\mathcal{T}(v^{1})-\mathcal{T}(v^{2})\|_{D(L)^{\prime}}^{2}\leq C_{R,\eps}\left\|v^{1}-v^{2}\right\|_{D(L)^{\prime}}^{2}.
    \end{align*}
\end{lemma}
\begin{proof}
    The proof is postponed to Appendix (\ref{proof:is_linear_growth_lipschitz}).
\end{proof}
\begin{proposition}\label{prop:approx_solution_specific_case}
    Let Assumptions \ref{Assumption:A1_domain}--\ref{Assumption:A5_noise_onb} hold. For every $\eps>0$ and $\delta>0$, there exists a unique local solution $(v^{\eps},\tau_{R})$ to problem \eqref{eqn:approximate_equation_strat}--\eqref{regularized_stratonovich.bic}.
\end{proposition}
\begin{proof}
    By \citet{braukhoff2024global}, the regularization of $\sigma$ and Lemma \ref{lem:is_linear_growth_lipschitz}, Assumptions \ref{Assumption:A6_noise_coefficient}, \ref{Assumption:A7_correction_term} are satisfied and Theorem \ref{thm:approx_solution_general} (with $\Sigma=\sigma_{\delta}$) yields the existence and uniqueness of a local solution.
\end{proof}

\section{Uniform estimates}\label{sec.uniform}

To pass to the limit in the regularization, we require suitable estimates. Hence, the goal of the following pages will be to establish an entropy estimate for the solution of \eqref{eqn:approximate_equation_strat}--\eqref{regularized_stratonovich.bic}, which is uniform in the regularization parameters.
\begin{lemma}\label{lem:is_correction_entropy_estimate}
Let $(v^\eps,\tau_R)$ be a local solution to \eqref{eqn:approximate_equation_strat}--\eqref{regularized_stratonovich.bic} and set
$v^R(t)=v^\eps(\omega,t\wedge\tau_R(\omega))$ for $\omega\in\Omega$, 
$t\in(0,\tau_R(\omega))$. We define
\begin{align*}
    \operatorname{IC}_{i,k,l}:=\pi_{i}\left(g_{\delta}(u^{\eps}_{i}\Ared_{i}(u^{\eps}))^{\prime}\left(\Ared_{i}(u^{\eps})\partial_{x_{l}}u^{\eps}_{i}+u^{\eps}_{i}\partial_{x_{l}}\Ared_{i}(u^{\eps})\right)e^{il}_{k}+g_{\delta}(u^{\eps}_{i}\Ared_{i}(u^{\eps}))\partial_{x_{l}}e^{il}_{k}\right)^{2}\frac{1}{u^{\eps}_{i}}.
\end{align*}
 Let $0< \kappa,\widetilde{\kappa}_{3} \leq \frac{1}{2}$. Then 
    \begin{align*}
     \lambda&\int_{0}^{t}\langle \mathcal{T}(u^{\eps}),R_{\eps}\rangle_{L^{2}(\dom)} \ds+\frac{1}{2}\int_{0}^{t}\int_{\dom}\sum_{i,k,l=1}^{n,\infty,d}\operatorname{IC}_{i,k,l}\dx\ds\\
        &\leq -\left(\frac\lambda4-\frac18-\frac{\lambda\kappa}2-\frac{3\kappa}8\right)\int_{0}^{t}\int_{\dom}\sum_{i,k,l=1}^{n,\infty,d}\pi_{i}\Ared_{i}(u^{\eps})\left(\frac{\partial_{x_{l}}u^{\eps}_{i}}{u^{\eps}_{i}}\right)^{2}\left(e^{il}_{k}\right)^{2}\dx \ds \\
&\phantom{xx}-\lambda\int_{0}^{t}\int_{\dom}\sum_{i,k,l=1}^{n,\infty,d}\pi_{i}a_{ii}\left(\partial_{x_{l}}\sqrt{u^{\eps}_{i}}\right)^{2}\left(e^{il}_{k}\right)^{2}\dx \ds \\
    &\phantom{xx}+\left(\frac{2\lambda+1}{8\kappa}+\frac18\right)\sum_{i,k,l=1}^{n,\infty,d}\pi_i\int_{0}^{t}\int_{\dom}\frac{\left(\partial_{x_{l}}\Ared_{i}(u^{\eps})\right)^{2}}{\Ared_i(u^\eps)}\left(e^{il}_{k}\right)^{2}\dx\ds\\
&\phantom{xx}+\frac{\widetilde\kappa_3}{4}\sum_{i,k,l=1}^{n,\infty,d}\pi_i\int_0^t\int_\dom\left(\partial_{x_l}\Ared_i(u^\eps)\right)^2\left(e_k^{il}\right)^2\dx\ds\\
    &\phantom{xx}+\frac{\lambda\kappa}{4}\sum_{i,k,l=1}^{n,\infty,d}\pi_{i}a_{ii} \int_{0}^{t}\int_{\dom}
 \left(\partial_{x_{l}}u^{\eps}_{i}\right)^{2}
 \left(e^{il}_{k}\right)^2\dx\ds\\
    &\phantom{xx}+  \sum_{i,k,l=1}^{n,\infty,d}\pi_{i}\left(\frac{\lambda a_{ii}}{4\kappa} +\frac{1}{4\widetilde{\kappa}_{3}}\right)\int_{0}^{t}\left\|\partial_{x_{l}}e^{il}_{k}\right\|^{2}_{L^{2}} \ds\\
&\phantom{xx}+C_a\left(\frac{\lambda+1}{4\kappa}+\frac12\right)d\,E_\infty'\sum_{i=1}^n\pi_{i}\int_{0}^{t}\int_{\dom}\big(u^{\eps}_{i}\log u^{\eps}_{i}-u^{\eps}_{i}+2\big)\dx \ds,
\end{align*}
where $C_a:=\max_{1\le j\le n}\big(a_{j0}+(e-1)\sum_{m=1}^na_{jm}\big)$ is the constant in the elementary bound $\sum_i\pi_i\Ared_i(u)\le C_a\sum_i\pi_i(u_i\log u_i-u_i+2)$, $u\in(0,\infty)^n$ (see the proof).
\end{lemma}
The proof is in the Appendix (see \ref{proof:is_correction_entropy_estimate}).

\begin{proposition}[Entropy inequality]\label{prop.ent}
Let $(v^\eps,\tau_R)$ be a local solution to \eqref{eqn:approximate_equation_strat}--\eqref{regularized_stratonovich.bic} and set
$v^R(t)=v^\eps(\omega,t\wedge\tau_R(\omega))$ for $\omega\in\Omega$, 
$t\in(0,\tau_R(\omega))$. Let $\lambda$, $\kappa$, $\widetilde\kappa_3$ and $\tau$ be as in Assumption~\ref{Assumption:A4_correction_factor}.
Then there exists a constant $C(u^0,T)>0$, depending on $u^0,\lambda,\kappa,\widetilde\kappa_3,\tau$ and $T$ but 
not on $\eps, \delta$ and $R$, such that
\begin{align*}
  \E&\sup_{0<t<T\wedge\tau_R}\int_\dom h(u^\eps(t))\dx 
	+ \frac{\eps}{2}\E\sup_{0<t<T\wedge\tau_R}\|Lw^\eps(t)\|_{L^2(\dom)}^2 \\
	&\phantom{xx}+ \E\sup_{0<t<T\wedge\tau_R} \int_0^t\int_\dom \Big[\sum_{i=1}^n\pi_i\big(4a_{i0}|\na(u^\eps)^{1/2}|^2 + 2a_{ii}|\na u^\eps|^2\big)\\
	&\phantom{xxxxxxxx}+ 2\sum_{i\neq j}\pi_i a_{ij}|\na (u^\eps_i u^\eps_j)^{1/2}|^2\Big]\dx\ds \\
 &\phantom{xx}+\correction C_{\lambda,\kappa}\E\sup_{0<t<T\wedge\tau_R}\int_{0}^{t}\int_{\dom}\sum_{i,k,l=1}^{n,\infty,d}\pi_{i}\Ared_{i}(u^{\eps})\left(\frac{\partial_{x_{l}}u^{\eps}_{i}}{u^{\eps}_{i}}\right)^{2}\left(e^{il}_{k}\right)^{2}\dx \ds\\
 &\phantom{xx}+\correction\lambda\E\sup_{0<t<T\wedge\tau_R}\int_{0}^{t}\int_{\dom}\sum_{i,k,l=1}^{n,\infty,d}\pi_{i}a_{ii}\left(\partial_{x_{l}}\sqrt{u^{\eps}_{i}}\right)^{2}\left(e^{il}_{k}\right)^{2}\dx \ds \\
 &\leq C\left(u^0,T,\kappa^{-1},\lambda,\tau\right),
\end{align*}
where $C_{\lambda,\kappa}:=\frac\lambda4-\frac18-\frac{\lambda\kappa}2-\frac{3\kappa}8$, and where $u^\eps:=u(R_\eps(v^R))$ and $w^\eps:=R_\eps(v^R)$.
In particular, we have that for all $\eps>0$ and
$i,j=1,\ldots,n$ with $i\neq j$,
\begin{align}
  \E\|u_i^\eps\|_{L^\infty(0,T;L^1(\dom))} &\leq C(u^0,T), \label{5.L1} \\
	a_{i0}^{1/2}\E\|(u_i^\eps)^{1/2}\|_{L^2(0,T;H^1(\dom))}
	+ a_{ii}^{1/2}\E\|u_i^\eps\|_{L^2(0,T;H^1(\dom))} &\leq C(u^0,T), \label{5.H1} \\
	a_{ij}^{1/2}\E\|\na(u_i^\eps u_j^\eps)^{1/2}\|_{L^2(0,T;L^2(\dom))} &\leq C(u^0,T).
	\nonumber 
\end{align}
Moreover, we have the estimate
\begin{equation}\label{5.v}
  \eps\E\|LR_\eps(v^\eps)\|_{L^\infty(0,T;L^2(\dom))}^2
	+ \E\|v^\eps\|_{L^\infty(0,T;D(L)')}^2 \leq C(u^0,T).
\end{equation}
\end{proposition}
\begin{corollary}\label{cor:nabla_log_u_estimate}
    For every $\eps>0$, $i,j=1,\dots,n$, the estimate
    \begin{align}
       \E\sum_{l=1}^d\left(\big\|\sqrt{F^{il}_1}\,\partial_{x_l}w^{\eps}_i\big\|_{L^{2}(0,T;L^{2}(\dom))}+\big\|\sqrt{\pi_{i}a_{ij}u^{\eps}_{j}F^{il}_1}\,\partial_{x_l}w^{\eps}_{i}\big\|_{L^{2}(0,T;L^{2}(\dom))}\right)\leq C(u^{0},T)
    \end{align}
    holds.
\end{corollary}
\begin{proof}
    The bound comes from the good term $\mathrm{IS}_1'$ in Proposition~\ref{prop.ent}, which is available because $C_{\lambda,\kappa}>0$. Since $w_i^\eps=\pi_i\log u_i^\eps$, we have $\na w_i^\eps=\pi_i\na u_i^\eps/u_i^\eps$. Since $\Ared_i(u^\eps)\ge a_{i0}$ and $\Ared_i(u^\eps)\ge a_{ij}u_j^\eps$,
    \[
      \frac{C_{\lambda,\kappa}}{N}\sum_{i,l=1}^{n,d}\pi_i\,\E\int_0^{T\wedge\tau_R}\int_\dom\big(a_{i0}+a_{ij}u_j^\eps\big)F^{il}_1\Big|\frac{\partial_{x_l}u_i^\eps}{u_i^\eps}\Big|^2\dx\ds\le2\,\E\,\mathrm{IS}_1'(T\wedge\tau_R)\le C(u^0,T).
    \]
    The claim follows by the Cauchy--Schwarz inequality in $\omega$. The constant depends in addition on $N$, $C_{\lambda,\kappa}^{-1}$ and $\min_ia_{i0}$, and the bound holds on $(0,T)$ once the solution is global (Proposition~\ref{prop:approximate_equation_global_solution}).
\end{proof}
\begin{proof}[Proof of Proposition \ref{prop.ent}]
The proof has six steps. Step~1 is It\^o's formula for the regularized entropy, and Step~2 identifies the dissipation of the deterministic part. Step~3 bounds the entropy production of the noise and of the correction term. This is the only step that uses $\lambda>\frac12$. Step~4 bounds the stochastic integral. Step~5 absorbs the remaining terms into the left-hand side using Assumption~\ref{Assumption:A4_correction_factor}, and Step~6 applies Gronwall's lemma. Throughout, $T_R:=T\wedge\tau_R$, $u^\eps:=u(R_\eps(v^R))$ and $w^\eps:=R_\eps(v^R)$. We write
\[
  \mathcal D(u):=\sum_{i=1}^n\pi_i\big(4a_{i0}|\na u_i^{1/2}|^2+2a_{ii}|\na u_i|^2\big)+2\sum_{i\neq j}\pi_ia_{ij}|\na(u_iu_j)^{1/2}|^2
\]
for the dissipation density, and we use $F^{il}_1\le E_\infty$ and $E_\infty'$ from Assumption~\ref{Assumption:A5_noise_onb}.

\emph{Step 1: It\^o's formula.} By the no-flux condition \eqref{eq:noise_bc}, the weak divergence of $(\sigma_\delta)_{ii}(u^\eps)e^{il}_k$ is the function $\partial_{x_l}((\sigma_\delta)_{ii}(u^\eps)e^{il}_k)\in L^2(\dom)$, so that Lemma~\ref{lem.ab} applies to the It\^o correction below. We apply the It\^o lemma in the version of 
\citet[Theorem 3.1]{Kry13}, with $V=H=D(L)'$
and the regularized entropy
\begin{equation}\label{3.ent}
  {\mathcal H}(v) := \int_\dom h(u(R_\eps(v)))\dx  
	+ \frac{\eps}{2}\|LR_\eps(v)\|_{L^2(\dom)}^2, \quad v\in D(L)'.
\end{equation}
Recall that $R_\eps(v) = h'(u(R_\eps(v)))$ for $v\in D(L)'$, since
$u=u(w)$ is the inverse of $h'$. The necessary conditions can be verified similarly to in \citet{braukhoff2024global}. 
For convenience, we will outline the steps. ${\mathcal H}$ is Fr\'echet differentiable and using Lemma \ref{lem.Reps}, its derivative can be expressed as
\begin{align*}
  \DD {\mathcal H}[v](\xi) &= \int_\dom\big(h'(u(R_\eps(v)))u'(R_\eps(v))
	\DD R_\eps[v](\xi)
	+ \eps L\DD R_\eps[v](\xi)\cdot LR_\eps(v)\big)\dx  \\
	&= \big\langle (u'(R_\eps(v))+\eps L^*L)\DD R_\eps[v](\xi),R_\eps(v) 
	\big\rangle_{D(L)',D(L)} \\
	&= \big\langle \DD Q_\eps[R_\eps(v)]\DD R_\eps[v](\xi),R_\eps(v)
	\big\rangle_{D(L)',D(L)}
	= \langle \xi,R_\eps(v)\rangle_{D(L)',D(L)},
\end{align*}
where $v$, $\xi\in D(L)'$.
In other words, $\DD {\mathcal H}[v]$ can be identified with 
$R_\eps(v)\in D(L)$.
In a similar way, we can prove that $\DD {\mathcal H}[v]$ 
is Fr\'echet differentiable and
$$
  \DD^2 {\mathcal H}[v](\xi,\eta) = \langle \xi,\DD R_\eps[v](\eta)\rangle_{D(L)',D(L)}
	\quad\mbox{for } v,\,\xi,\,\eta\in D(L)'.
$$
Thanks to the Lipschitz continuity of $R_\eps$
and $\DD R_\eps[v]$ (see Lemma \ref{lem.Reps})
for all $v$, $\xi\in D(L)'$ with $\|v\|_{D(L)'}\leq K$ for some $K>0$, we have
\begin{align*}
  |\DD {\mathcal H}[v](\xi)| &\leq \|R_\eps(v)\|_{D(L)}\|\xi\|_{D(L)'}
	\leq C(\eps)(1+\|v\|_{D(L)'})\|\xi\|_{D(L)'} \leq C(\eps,K)\|\xi\|_{D(L)'}, \\
  |\DD^2 {\mathcal H}[v](\xi,\xi)| &\leq \|\DD R_\eps[v](\xi)\|_{D(L)}\|\xi\|_{D(L)'}
	\leq C(\eps)\|\xi\|_{D(L)'}^2.
\end{align*}

To bound the mapping $D(L)'\to\R$, 
$v\mapsto \DD {\mathcal H}[v](\eta)$ for any $\eta\in D(L)'$, 
as in Lemma \ref{lem.L1}, we use that the operator $L$ can be 
constructed in such a way that the Riesz representative in $D(L)'$ of a functional
acting on $D(L)'$ can be expressed via the application of $L^*L$ to an element of
$D(L)$. Indeed, for $F\in D(L)$ and $\xi\in D(L)'$, Lemma \ref{lem.L1} yields that
\begin{align*}
  \langle\xi,F\rangle_{D(L)',D(L)} &= (L^{-1}\xi,LF\rangle_{D(L)',D(L)}
	= ((LL^{-1})L^{-1}\xi,LF)_{L^2(\dom)} \\
	&= (L^{-1}\xi,L^{-1}L^*LF)_{L^2(\dom)} = (L^*LF,\xi)_{D(L)'}.
\end{align*}
Hence, we can associate $\DD {\mathcal H}[v]$ with $L^*LR_\eps(v) \in D(L)'$, since we can identify $\DD {\mathcal H}[v]$ with $R_\eps(v)\in D(L)$.
Then, by the first estimate in \eqref{3.LL} and the Lipschitz continuity of 
$R_\eps$,
\begin{align*}
  \|L^*LR_\eps(v)\|_{D(L)'} 
	&\leq C\|R_\eps(v)\|_{D(L)} 
	\leq C\|R_\eps(v)-R_\eps(0)\|_{D(L)} + C\|R_\eps(0)\|_{D(L)} \\
	&\leq C(\eps)(1+\|v\|_{D(L)'})\quad\mbox{for all }v\in D(L)',
\end{align*}
giving the desired estimate for $\DD{\mathcal H}[v]$ in $D(L)'$.
Thus, the assumptions of the It\^o lemma, as stated in \citet{Kry13}, are satisfied.

To simplify the notation, we set $u^\eps := u(R_\eps(v^R))$ and 
$w^\eps := R_\eps(v^R)$ in the following.
By It\^o's lemma, using $\DD {\mathcal H}[v^R]=h'(u^\eps)$, 
$\DD^2 {\mathcal H}[v^R]=\DD R_\eps(v^R)$, and already applying Lemma \ref{lem.ab} to the correction term, we have
\begin{align}\label{eqn:ito_entropy}
  {\mathcal H}(v^R(t))  
	&\leq {\mathcal H}(v(0)) + \int_0^t\big\langle
	\diverm\big(B(w^\eps)\na h'(u^\eps(s))\big),w^\eps(s)
	\big\rangle_{D(L)',D(L)}\ds\\
	&\phantom{xx}{}+ \sqrtcorrection\sum_{i,k,l=1}^{n,\infty,d}\int_0^t \int_\dom\partial_{x_{l}}\left((\sigma_{\delta})_{ii}(u^{\eps}) e^{il}_{k}\right) \left( h^{\prime}(u^{\eps}(s))\right)_{i}\dx\d W^{il}_{k}(s)
  \nonumber \\
	&\phantom{xx}{}+\correction\lambda\sum_{i=1}^n\int_0^t \int_\dom
	\mathcal{T}(v^{\eps}(s))_{i}\left( h^{\prime}(u^{\eps}(s))\right)_{i}\dx  \ds
  \nonumber \\
	&\phantom{xx}{}
	+ \frac{1}{2}\correction\sum_{k=1}^\infty\int_0^t\int_\dom \sum_{i,l=1}^{n,d}\left(\partial_{x_{l}}\left((\sigma_\delta)_{ii}(u^{\eps}) e^{il}_{k} \right)\right)^{2}\frac{\pi_{i}}{u^{\eps}_{i}}
	\dx \ds. \nonumber
\end{align}
\emph{Step 2: the deterministic part.} By Lemma \ref{lem.v0}, the first term on the right-hand side can be
estimated from above by $\int_\dom h(u^0)\dx $.
Using $w^\eps = R_\eps(v^R)=h'(u^\eps)$ and integrating by parts,
the second term on the right-hand side can be rewritten as
\begin{align*}
  \int_0^t\big\langle &
	\diver\big(B(w^\eps)\na h'(u^\eps(s))\big),w^\eps(s)
	\big\rangle_{D(L)',D(L)}\ds\\
	&= -\int_0^t\int_\dom \na w^\eps(s):B(w^\eps)\na w^\eps(s)\dx \ds\leq 0.
\end{align*}
We observe that $R_\eps(v^\eps) = h'(u(R_\eps(v^\eps))) = h'(u^\eps)$
implies that $\na R_\eps(v^\eps) = h''(u^\eps)\na u^\eps$.
It is shown in \citet[Lemma 4]{CDJ18}
that for all $z\in\R^n$ and $u\in(0,\infty)^n$,
$$
  z^T h''(u)A(u)z \ge \sum_{i=1}^n\pi_i\bigg(a_{i0}\frac{z_i^2}{u_i} 
	+ 2a_{ii}z_i^2\bigg) + \frac12\sum_{i,j=1,\,i\neq j}^n
	\pi_i a_{ij}\bigg(\sqrt{\frac{u_j}{u_i}}z_i + \sqrt{\frac{u_i}{u_j}}z_j\bigg)^2.
$$
Using $B(R_\eps(v^\eps))=A(u^\eps)h''(u^\eps)^{-1}$ 
and the previous inequalities with $z=\na u^\eps$, we find that
\begin{align}\label{5.RBR}
  \na R_\eps(v^\eps)&:B(R_\eps(v^\eps))\na R_\eps(v)
	= \na u^\eps:h''(u^\eps)\big(A(u^\eps)h''(u^\eps)^{-1}\big)h''(u^\eps)\na u^\eps \\
	&= \na u^\eps:h''(u^\eps)A(u^\eps)\na u^\eps \nonumber \\
	&\geq \sum_{i=1}^n\pi_i\big(4a_{i0}|\na(u^\eps)^{1/2}|^2 + 2a_{ii}|\na u^\eps|^2\big)
	+ 2\sum_{i\neq j}\pi_i a_{ij}|\na (u^\eps_i u^\eps_j)^{1/2}|^2.
\end{align}
Due to the choice of space, the boundary integral vanishes. We write the stochastic integral in \eqref{eqn:ito_entropy} in weak form and obtain, for every $t\in[0,T_R]$,
\begin{equation}\label{eqn:ito_no_sup}
\begin{aligned}
  &\int_\dom h(u^\eps(t))\dx+\frac\eps2\|Lw^\eps(t)\|_{L^2(\dom)}^2+\int_0^t\int_\dom\mathcal D(u^\eps)\dx\ds\\
  &\qquad\le\int_\dom h(u^0)\dx+I(t)+II(t)+M(t),
\end{aligned}
\end{equation}
where
\begin{align*}
  I(t)&:=\frac\lambda N\int_0^t\big\langle\mathcal T(u^\eps),w^\eps\big\rangle_{L^2(\dom)}\ds,\qquad
  II(t):=\frac1{2N}\int_0^t\int_\dom\sum_{i,k,l=1}^{n,\infty,d}\operatorname{IC}_{i,k,l}\dx\ds,\\
  M(t)&:=-\frac1{\sqrt N}\sum_{i,k,l=1}^{n,\infty,d}\int_0^t\int_\dom(\sigma_\delta)_{ii}(u^\eps)\,e^{il}_k\,\partial_{x_l}\big(h'(u^\eps)\big)_i\dx\,\dd W^{il}_k(s),
\end{align*}
and $\operatorname{IC}_{i,k,l}=\pi_i\big(\partial_{x_l}((\sigma_\delta)_{ii}(u^\eps)e^{il}_k)\big)^2/u_i^\eps$ is as in Lemma~\ref{lem:is_correction_entropy_estimate}. The term $I+II$ is the entropy production of the noise and of the correction term. In the notation of Section~\ref{sec.lambda}, it is $\frac1N\int_0^t\mathfrak P_\lambda(u^\eps)\ds$, with $\sigma_\delta$ in place of $\sigma$ in the It\^o correction.

\emph{Step 3: the entropy production of the noise.} This is the only step that uses $\lambda>\frac12$. Lemma~\ref{lem:is_correction_entropy_estimate} bounds $N(I+II)$. Its first term is negative, with coefficient $C_{\lambda,\kappa}>0$, since $\lambda>\frac12$ and $\kappa<\frac{2\lambda-1}{4\lambda+3}$. We keep it on the left-hand side. We write it, and the second term, with $F^{il}_1=\sum_k(e^{il}_k)^2$:
\begin{align*}
  \mathrm{IS}_1'(t)&:=\frac{C_{\lambda,\kappa}}N\sum_{i,l=1}^{n,d}\pi_i\int_0^t\int_\dom\Ared_i(u^\eps)\Big(\frac{\partial_{x_l}u_i^\eps}{u_i^\eps}\Big)^2F^{il}_1\dx\ds,\\
  \mathrm{IS}_2'(t)&:=\frac\lambda N\sum_{i,l=1}^{n,d}\pi_ia_{ii}\int_0^t\int_\dom\big(\partial_{x_l}\sqrt{u_i^\eps}\big)^2F^{il}_1\dx\ds .
\end{align*}
The remaining gradient terms of Lemma~\ref{lem:is_correction_entropy_estimate} are bounded by the self-diffusion part $2\pi_ja_{jj}|\na u_j^\eps|^2$ of $\mathcal D(u^\eps)$. Since $\partial_{x_l}\Ared_i(u)=\sum_ja_{ij}\partial_{x_l}u_j$, the Cauchy--Schwarz inequality and $\Ared_i(u)\ge a_{i0}$ give
\begin{align*}
  \big(\partial_{x_l}\Ared_i(u^\eps)\big)^2&\le\Big(\sum_{m=1}^na_{im}\Big)\sum_{j=1}^na_{ij}(\partial_{x_l}u^\eps_j)^2,\\
  \frac{(\partial_{x_l}\Ared_i(u^\eps))^2}{\Ared_i(u^\eps)}&\le\frac1{a_{i0}}\Big(\sum_{m=1}^na_{im}\Big)\sum_{j=1}^na_{ij}(\partial_{x_l}u^\eps_j)^2 .
\end{align*}
We multiply by $\pi_iF^{il}_1\le\pi_iE_\infty$, sum over $i$ and $l$, and exchange the sums over $i$ and $j$. With $\beta_j$, $\gamma_j$ as in \ref{Assumption:A4_correction_factor}, this gives
\begin{align*}
  \sum_{i,l=1}^{n,d}\pi_i\frac{(\partial_{x_l}\Ared_i(u^\eps))^2}{\Ared_i(u^\eps)}F^{il}_1&\le E_\infty\sum_{j=1}^n\beta_j|\na u_j^\eps|^2,\\
  \sum_{i,l=1}^{n,d}\pi_i(\partial_{x_l}\Ared_i(u^\eps))^2F^{il}_1&\le E_\infty\sum_{j=1}^n\gamma_j|\na u_j^\eps|^2,
\end{align*}
and $\sum_{i,l}\pi_ia_{ii}(\partial_{x_l}u_i^\eps)^2F^{il}_1\le E_\infty\sum_j\pi_ja_{jj}|\na u_j^\eps|^2$. The terms of Lemma~\ref{lem:is_correction_entropy_estimate} that contain no derivative of $u^\eps$ are bounded using $\|\partial_{x_l}e^{il}_k\|_{L^2(\dom)}^2\le|\dom|\,\|\partial_{x_l}e^{il}_k\|_{L^\infty(\dom)}^2$ and $u\log u-u+2=(u\log u-u+1)+1$. Altogether,
\begin{equation}\label{eqn:I_II_bound}
\begin{aligned}
  I(t)+II(t)\le{}&-\mathrm{IS}_1'(t)-\mathrm{IS}_2'(t)+\frac{E_\infty}N\sum_{j=1}^nK_j(\lambda,\kappa,\widetilde\kappa_3)\int_0^t\int_\dom|\na u_j^\eps|^2\dx\ds\\
  &+C_{\rm det}\,t+C_{\kappa,a}\int_0^t\int_\dom h(u^\eps)\dx\ds,
\end{aligned}
\end{equation}
with $K_j$ as in \ref{Assumption:A4_correction_factor} and
\begin{align*}
  C_{\rm det}&:=\frac{d|\dom|E_\infty'}N\Big[\sum_{i=1}^n\pi_i\Big(\frac{\lambda a_{ii}}{4\kappa}+\frac1{4\widetilde\kappa_3}\Big)+C_a\Big(\frac{\lambda+1}{4\kappa}+\frac12\Big)\sum_{i=1}^n\pi_i\Big],\\
  C_{\kappa,a}&:=\frac{dE_\infty'}NC_a\Big(\frac{\lambda+1}{4\kappa}+\frac12\Big).
\end{align*}
By the fourth inequality of \ref{Assumption:A4_correction_factor}, $\frac{E_\infty}NK_j(\lambda,\kappa,\widetilde\kappa_3)\le\tau\,2\pi_ja_{jj}$. Hence the third term on the right of \eqref{eqn:I_II_bound} is at most $\tau$ times the self-diffusion part of the dissipation. We set
\[
  \mathcal D_\tau(u):=\mathcal D(u)-\tau\sum_{j=1}^n2\pi_ja_{jj}|\na u_j|^2\ \ge\ (1-\tau)\,\mathcal D(u)\ \ge0 .
\]
Note that $\mathcal D_\tau$ still contains the term $\sum_i4\pi_ia_{i0}|\na(u_i)^{1/2}|^2$ with coefficient one. Inserting \eqref{eqn:I_II_bound} into \eqref{eqn:ito_no_sup}, we obtain, pathwise and for every $t\in[0,T_R]$,
\begin{equation}\label{eqn:ito_no_sup_estimate}
  Y(t)\le\int_\dom h(u^0)\dx+C_{\rm det}\,t+C_{\kappa,a}\int_0^t\int_\dom h(u^\eps)\dx\ds+M(t),
\end{equation}
where
\begin{align*}
  Y(t):={}&\int_\dom h(u^\eps(t))\dx+\frac\eps2\|Lw^\eps(t)\|_{L^2(\dom)}^2\\
  &+\int_0^t\int_\dom\mathcal D_\tau(u^\eps)\dx\ds+\mathrm{IS}_1'(t)+\mathrm{IS}_2'(t)\ \ge0 .
\end{align*}
Inequality \eqref{eqn:ito_no_sup_estimate} is used again in the proof of Lemma~\ref{lem.hom}.

\emph{Step 4: the stochastic integral.} The process $M$ is a martingale. We bound $\E\sup_t|M(t)|$ by the entropy and by the part $4\pi_ia_{i0}|\na(u_i^\eps)^{1/2}|^2$ of $\mathcal D_\tau$. Fix $T'\in(0,T]$ and set $T_R':=T'\wedge\tau_R$. By the Burkholder--Davis--Gundy inequality \citep[Theorem 6.1.2]{LiRo15},
\begin{equation}\label{eqn:BDG_stochastic_integral}
  \E\sup_{0\le t\le T_R'}|M(t)|\le\frac4{\sqrt N}\,\E\bigg(\int_0^{T_R'}\sum_{i,k,l=1}^{n,\infty,d}\Big(\int_\dom(\sigma_\delta)_{ii}(u^\eps)\,e^{il}_k\,\partial_{x_l}\big(h'(u^\eps)\big)_i\dx\Big)^2\ds\bigg)^{1/2}.
\end{equation}
We bound the integrand for fixed $i$, $l$ and $s$. Since $\partial_{x_l}(h'(u^\eps))_i=2\pi_i\partial_{x_l}\sqrt{u_i^\eps}/\sqrt{u_i^\eps}$ and $(\sigma_\delta)_{ii}(u^\eps)\le\sqrt{u_i^\eps\Ared_i(u^\eps)}$, the Cauchy--Schwarz inequality gives
\[
  \sum_{k=1}^\infty\Big(\int_\dom(\sigma_\delta)_{ii}(u^\eps)\,e^{il}_k\,\partial_{x_l}\big(h'(u^\eps)\big)_i\dx\Big)^2
  \le4E_\infty\,\big\|\pi_i\Ared_i(u^\eps)\big\|_{L^1(\dom)}\,\pi_i\big\|\partial_{x_l}\sqrt{u_i^\eps}\big\|_{L^2(\dom)}^2 .
\]
We sum over $i$ and $l$, take the supremum in time of the $L^1$ norm, and use $\sqrt{ab}\le\frac12(a+b)$. Then the right-hand side of \eqref{eqn:BDG_stochastic_integral} is at most
\[
  \frac4{\sqrt N}E_\infty^{1/2}\,\E\Big[\sup_{0\le s\le T_R'}\Xi(s)+\Theta\Big],
\]
where
\[
  \Xi(s):=\sum_{i=1}^n\big\|\pi_i\Ared_i(u^\eps(s))\big\|_{L^1(\dom)},\qquad
  \Theta:=\int_0^{T_R'}\sum_{i=1}^n\pi_i\|\na\sqrt{u_i^\eps}\|_{L^2(\dom)}^2\ds .
\]
In particular, the quadratic variation of $M$ satisfies $\langle M\rangle_{T_R'}\le\frac{4E_\infty}N\sup_{s\le T_R'}\Xi(s)\,\Theta$.
It remains to bound $\Xi$ by the entropy. By the detailed-balance condition, $\sum_i\pi_i\Ared_i(u)=\sum_i\pi_ia_{i0}+\sum_j\pi_j\big(\sum_ia_{ji}\big)u_j$. Moreover, $u\le(u\log u-u+1)+(e-1)$ for $u\ge0$, since the difference of the two sides attains its maximum $0$ at $u=e$. With $a^\sharp\ge\max_j\sum_ia_{ji}$ from \ref{Assumption:A4_correction_factor}, we obtain
\[
  \Xi(s)\le|\dom|\sum_{i=1}^n\pi_i\big(a_{i0}+(e-1)a^\sharp\big)+a^\sharp\int_\dom h(u^\eps(s))\dx .
\]
Altogether,
\begin{equation}\label{eqn:BDG_final_form}
  \E\sup_{0\le t\le T_R'}|M(t)|\le\frac1{\sqrt N}\Big(\alpha_1+\alpha_2\,\E\sup_{0\le t\le T_R'}\int_\dom h(u^\eps(t))\dx+\alpha_3\,\E\,\Theta\Big),
\end{equation}
with $\alpha_1:=4E_\infty^{1/2}|\dom|\sum_i\pi_i(a_{i0}+(e-1)a^\sharp)$, $\alpha_2:=4E_\infty^{1/2}a^\sharp$ and $\alpha_3:=4E_\infty^{1/2}$.

\emph{Step 5: absorption.} Let $F(T'):=\E\sup_{0\le t\le T_R'}Y(t)$. For fixed $\eps$, $\delta$ and $R$, all quantities in $Y$ are bounded up to $\tau_R$, so $F(T')<\infty$. We take the supremum over $t\in[0,T_R']$ in \eqref{eqn:ito_no_sup_estimate}, then the expectation, and insert \eqref{eqn:BDG_final_form}. The first and third inequalities of \ref{Assumption:A4_correction_factor} state that $\alpha_2/\sqrt N\le\tau$ and $\alpha_3/\sqrt N\le4\tau a_{i0}$ for every $i$. Therefore
\begin{align*}
  \frac{\alpha_2}{\sqrt N}\,\E\sup_{t\le T_R'}\int_\dom h(u^\eps(t))\dx&\le\tau F(T'),\\
  \frac{\alpha_3}{\sqrt N}\,\E\,\Theta\le\tau\,\E\int_0^{T_R'}\int_\dom\sum_{i=1}^n4\pi_ia_{i0}|\na(u_i^\eps)^{1/2}|^2\dx\ds&\le\tau F(T').
\end{align*}
This gives
\begin{align*}
  F(T')\le{}&\E\int_\dom h(u^0)\dx+C_{\rm det}T+\frac{\alpha_1}{\sqrt N}\\
  &+C_{\kappa,a}\,\E\int_0^{T_R'}\int_\dom h(u^\eps(s))\dx\ds+2\tau F(T').
\end{align*}
Since $\tau<\frac12$, we may move the last term to the left-hand side. With $c_\tau:=\frac1{1-2\tau}$ and $\int_0^{T_R'}\int_\dom h(u^\eps(s))\dx\ds\le\int_0^{T'}\sup_{0\le r\le s\wedge\tau_R}\int_\dom h(u^\eps(r))\dx\ds$,
\begin{equation}\label{eqn:e_sup_ito_estimate_final}
  F(T')\le c_\tau\Big(\E\int_\dom h(u^0)\dx+C_{\rm det}T+\frac{\alpha_1}{\sqrt N}\Big)+c_\tau C_{\kappa,a}\int_0^{T'}F(s)\ds
\end{equation}
for all $T'\in(0,T]$.

\emph{Step 6: conclusion.} Gronwall's lemma applied to \eqref{eqn:e_sup_ito_estimate_final} gives $F(T)\le C(u^0,T,\kappa^{-1},\lambda,\tau)$. Every term on the left-hand side of the estimate in the proposition is bounded by $F(T)/(1-\tau)$, because $Y\ge0$ and $\mathcal D_\tau\ge(1-\tau)\mathcal D$, and because the time integrals in $Y$ are nondecreasing in $t$. This proves the first estimate. The bound \eqref{5.L1} follows from $u\le(u\log u-u+1)+e-1$, and \eqref{5.H1} and the bound on $\na(u_i^\eps u_j^\eps)^{1/2}$ follow from the dissipation bound together with \eqref{5.L1} and the Poincar\'e--Wirtinger inequality. Finally, \eqref{5.v} follows from Lemma~\ref{lem.hom} below with second moments, its proof uses only Steps~1--4 above.
\end{proof}

The entropy inequality allows us to extend the local solution to a global one.

\begin{proposition}\label{prop:approximate_equation_global_solution}
Let $(v^\eps,\tau_R)$ be a local solution to \eqref{eqn:approximate_equation_strat}--\eqref{regularized_stratonovich.bic},
constructed in Proposition \ref{prop:approx_solution_specific_case}. Then $v^\eps$ can be extended to 
a global solution to \eqref{eqn:approximate_equation_strat}--\eqref{regularized_stratonovich.bic}.
\end{proposition}

\begin{proof}
With the notation $u^\eps = u(R_\eps(v^\eps))$ and $w^\eps=R_\eps(v^\eps)$, 
we observe that $v^\eps=Q_\eps(R_\eps(v^\eps))$ 
$= u(R_\eps(v^\eps))+\eps L^*LR_\eps(v^\eps) = u^\eps + \eps L^*Lw^\eps$.
Thus, we have for $T_R\leq T\wedge\tau_R$,
\begin{align*}
  \E&\sup_{0<t<T_R}\|v^\eps(t)\|_{D(L)'}
	\leq \E\sup_{0<t<T_R}\|u^\eps\|_{D(L)'} 
	+ \eps\E\sup_{0<t<T_R}\|L^*Lw^\eps(t)\|_{D(L)'} \\
	&\leq C\E\sup_{0<t<T_R}\|u^\eps\|_{L^1(\dom)}
	+ \eps \E\sup_{0<t<T_R}\|L^*Lw^\eps(t)\|_{D(L)'}.
\end{align*}
We know from the elementary inequality $x\leq C(1+x\log x-x)$ for $x>0$, applied componentwise to $u^\eps$, that $|u^\eps|\leq C(1+h(u^\eps))$. Therefore,
taking into account the entropy inequality and the second inequality
in \eqref{3.LL},
\begin{align*}
  \E\sup_{0<t<T_R}\|v^\eps(t)\|_{D(L)'} 
	&\leq C\E\sup_{0<t<T_R}\|h(u^\eps(t))\|_{L^1(\dom)}+ \eps C\,\E\sup_{0<t<T_R}\|Lw^\eps(t)\|_{L^2(\dom)}\\
	&\leq C(u^0,T).
\end{align*}
This allows us to perform the limit $R\to\infty$ and to conclude that the solution $v^\eps$ exists in $(0,T)$ for any $T>0$.
\end{proof}

\subsection{Higher order uniform estimates}

Let $v^\eps$ be a global solution to \eqref{eqn:approximate_equation_strat}--\eqref{regularized_stratonovich.bic}
and set $u^\eps=u(R_\eps(v^\eps))$. We assume that $A(u)$ is given by
\eqref{eq:SKT_intro_matrix_A} and that $a_{ii}>0$ for $i=1,\ldots,n$.
We start with some uniform estimates, which are a consequence
of the entropy inequality in Proposition \ref{prop.ent}.

\begin{lemma}[Higher-order moments I]\label{lem.hom}
Let $p>2$ be as in Assumption~\ref{Assumption:A4_correction_factor}, and let $\lambda,\kappa,\widetilde\kappa_3,\tau$ be as in Proposition \ref{prop.ent}.
There exists a constant $C(p,u^0,T,\kappa^{-1},\lambda,\tau)$, which is independent of $\eps$, such that
\begin{align}
  \E\|u^\eps\|_{L^\infty(0,T;L^1(\dom))}^p &\leq C(p,u^0,T,\kappa^{-1},\lambda,\tau), \label{5.L1p} \\
	a_{i0}^{p/2}\E\|(u_i^\eps)^{1/2}\|_{L^2(0,T;H^1(\dom))}^p
	+ a_{ii}^{p/2}\E\|u_i^\eps\|_{L^2(0,T;H^1(\dom))}^p &\leq C(p,u^0,T,\kappa^{-1},\lambda,\tau), \label{5.H1p} \\
	a_{ij}^{p/2}\E\|\na(u_i^\eps u_j^\eps)^{1/2}\|_{L^2(0,T;L^2(\dom))}^p &\leq C(p,u^0,T,\kappa^{-1},\lambda,\tau).
	\label{5.nablap}
\end{align}
Moreover, we have
\begin{equation}
  \E\bigg(\eps\sup_{0<t<T}\|LR_\eps(v^\eps(t))\|_{L^2(\dom)}^2\bigg)^p
	+ \E\bigg(\sup_{0<t<T}\|v^\eps(t)\|_{D(L)'}\bigg)^p \leq C(p,u^0,T,\kappa^{-1},\lambda,\tau), \label{5.vp}
\end{equation}
    \begin{align}\label{eqn:higher_oder_bounds_nabla_log_u}
        \E\bigg(\sum_{l=1}^d\Big(\big\|\sqrt{F^{il}_1}\,\partial_{x_l}w^{\eps}_i\big\|_{L^{2}(Q_T)}+\big\|\sqrt{\pi_{i}a_{ij}u^{\eps}_{j}F^{il}_1}\,\partial_{x_l}w^{\eps}_{i}\big\|_{L^{2}(Q_T)}\Big)\bigg)^{p}\leq C.
    \end{align}
    The constant in \eqref{eqn:higher_oder_bounds_nabla_log_u} depends on $p$, $u^0$, $T$, $\kappa^{-1}$, $\lambda$, $\tau$, $N$, $C_{\lambda,\kappa}^{-1}$ and $\min_ia_{i0}$.
\end{lemma}

\begin{proof}
We start from \eqref{eqn:ito_no_sup_estimate}, which holds pathwise for every $t\in[0,T_R]$:
\[
  Y(t)\le\int_\dom h(u^0)\dx+C_{\rm det}\,t+C_{\kappa,a}\int_0^t\int_\dom h(u^\eps)\dx\ds+M(t),
\]
with $Y$, $\mathcal D_\tau$, $\mathrm{IS}_1'$, $\mathrm{IS}_2'$, $C_{\rm det}$, $C_{\kappa,a}$ and $M$ as defined there. The correction terms controlled by the fourth inequality of \ref{Assumption:A4_correction_factor} have already been absorbed into $\mathcal D_\tau\ge(1-\tau)\mathcal D$.

We work with the norms $\|Z\|_p:=(\E|Z|^p)^{1/p}$, with $p$ as in Assumption~\ref{Assumption:A4_correction_factor}. For $T'\in(0,T]$ let $T_R':=T'\wedge\tau_R$ and $G(T'):=\|\sup_{0\le t\le T_R'}Y(t)\|_p$, which is finite for fixed $\eps$, $\delta$ and $R$. Since $Y\ge0$, the inequality above gives
\begin{align*}
  \sup_{0\le t\le T_R'}Y(t)\le{}&\int_\dom h(u^0)\dx+C_{\rm det}T\\
  &+C_{\kappa,a}\int_0^{T'}\sup_{0\le r\le s\wedge\tau_R}\int_\dom h(u^\eps(r))\dx\ds+\sup_{0\le t\le T_R'}|M(t)| .
\end{align*}
By Minkowski's inequality, and Minkowski's integral inequality for the time integral,
\[
  G(T')\le\Big\|\int_\dom h(u^0)\dx\Big\|_p+C_{\rm det}T+C_{\kappa,a}\int_0^{T'}G(s)\ds+\Big\|\sup_{0\le t\le T_R'}|M(t)|\Big\|_p .
\]

\emph{The stochastic term.} By the Burkholder--Davis--Gundy inequality in $L^p$ \citep[Theorem 6.1.2]{LiRo15}, $\|\sup_{t\le T_R'}|M(t)|\|_p\le c_p\|\langle M\rangle_{T_R'}^{1/2}\|_p$. At the end of Step~4 of the proof of Proposition~\ref{prop.ent} we showed $\langle M\rangle_{T_R'}\le\frac{4E_\infty}N\sup_{s\le T_R'}\Xi(s)\,\Theta$, hence $\langle M\rangle_{T_R'}^{1/2}\le(E_\infty/N)^{1/2}(\sup_s\Xi(s)+\Theta)$. Moreover, $\sup_s\Xi(s)\le\alpha_1'+a^\sharp\sup_{t\le T_R'}\int_\dom h(u^\eps(t))\dx\le\alpha_1'+a^\sharp\sup_tY(t)$, with $\alpha_1':=|\dom|\sum_i\pi_i(a_{i0}+(e-1)a^\sharp)$, and $\Theta\le(4\min_ia_{i0})^{-1}\sup_tY(t)$, since $\mathcal D_\tau$ contains $\sum_i4\pi_ia_{i0}|\na(u_i^\eps)^{1/2}|^2$. Therefore
\begin{align*}
  \Big\|\sup_{0\le t\le T_R'}|M(t)|\Big\|_p&\le c_p\Big(\frac{E_\infty}N\Big)^{1/2}\alpha_1'+c_p\Big(\frac{E_\infty}N\Big)^{1/2}\Big(a^\sharp+\frac1{4\min_ia_{i0}}\Big)G(T')\\
  &\le c_p\Big(\frac{E_\infty}N\Big)^{1/2}\alpha_1'+2\tau\,G(T'),
\end{align*}
by the second inequality of \ref{Assumption:A4_correction_factor}.

\emph{Conclusion.} Since $2\tau<1$, we obtain, with $c_\tau=\frac1{1-2\tau}$,
\[
  G(T')\le c_\tau\Big(\Big\|\int_\dom h(u^0)\dx\Big\|_p+C_{\rm det}T+c_p\Big(\frac{E_\infty}N\Big)^{1/2}\alpha_1'\Big)+c_\tau C_{\kappa,a}\int_0^{T'}G(s)\ds .
\] By Gronwall's lemma, $G(T)\le C(p,u^0,T,\kappa^{-1},\lambda,\tau)$, uniformly in $R$. For the global solution we let $R\to\infty$ by monotone convergence. Each summand of $Y$ has its $p$-th moment bounded by $G(T)^p$, and the dissipation by $(1-\tau)^{-p}G(T)^p$. Since the entropy dominates the $L^1(\dom)$ norm up to a constant and by the Poincar\'e--Wirtinger inequality, this gives \eqref{5.L1p}--\eqref{5.nablap} and the bound on $\E(\eps\sup_t\|LR_\eps(v^\eps(t))\|_{L^2(\dom)}^2)^p$ in \eqref{5.vp}.

Combining the previous estimates and arguing as in the proof of Proposition \ref{prop:approximate_equation_global_solution},
we have
\begin{align*}
  \E&\bigg(\sup_{0<t<T}\|v^\eps(t)\|_{D(L)'}\bigg)^p
	= \E\bigg(\sup_{0<t<T}\|u^\eps(t)+\eps L^*LR_\eps(v^\eps(t))\|_{D(L)'}\bigg)^p \\
	&\leq C\E\bigg(\sup_{0<t<T}\|u^\eps(t)\|_{L^1(\dom)}\bigg)^p
	+ C\E\bigg(\eps^2\sup_{0<t<T}\|LR_\eps(v^\eps(t))\|_{L^2(\dom)}^2\bigg)^{p/2}
	\leq C(p,u^0,T).
\end{align*}
The last claim follows from the $p$-th moment of $\mathrm{IS}_1'$ by the same arguments as in Corollary \ref{cor:nabla_log_u_estimate}. This ends the proof.
\end{proof}

Using the Gagliardo--Nirenberg inequality, we can derive further estimates. 
We recall that $Q_T=\dom\times(0,T)$.

\begin{lemma}[Higher-order moments II]\label{lem.hom2}
Let $p$ be as in Assumption~\ref{Assumption:A4_correction_factor} and let $q\ge1$. Then \eqref{5.22d} holds with $p$ replaced by $q$ if $\frac{2dq}{d+1}\le p$, and \eqref{5.24d} holds with $p$ replaced by $q$ if $\max\{4q/(d+2),\,2dq/(d+2)\}\le p$. The proof below uses the moments of order $2(1-\theta)q$ and $2\theta q$ from Lemma~\ref{lem.hom}, with $\theta=\frac d{d+1}$ and $\theta=\frac d{d+2}$, respectively. In particular, $q=p/2$ is admissible in both, and $q=2+\frac4d$ is admissible in \eqref{5.24d} since $p\ge p_0(d)$. In the statement we write $p$ for $q$. There exists a constant
$C(p,u^0,T)$ $>0$, which is independent of $\eps$, such that
\begin{align}
  \E\|u_i^\eps\|_{L^{2+2/d}(Q_T)}^p &\leq C(p,u^0,T), \label{5.22d} \\
	\E\|u_i^\eps\|_{L^{2+4/d}(0,T;L^2(\dom))}^p &\leq C(p,u^0,T). \label{5.24d}
\end{align}
\end{lemma}

\begin{proof}
We apply the Gagliardo--Nirenberg inequality:
\begin{align*}
  \E&\bigg(\int_0^T\|u_i^\eps\|_{L^r(\dom)}^s\dt\bigg)^{p/s}
	\leq C\E\bigg(\int_0^T\|u_i^\eps\|_{H^1(\dom)}^{\theta s}
	\|u_i^\eps\|_{L^1(\dom)}^{(1-\theta)s}\dt\bigg)^{p/s} \\
	&\leq C\E\bigg(\|u_i^\eps\|_{L^\infty(0,T;L^1(\dom))}^{(1-\theta)s}\int_0^T
	\|u_i^\eps\|_{H^1(\dom)}^{2}\dt\bigg)^{p/s} \\
	&\leq C\big(\E\|u_i^\eps\|_{L^\infty(0,T;L^1(\dom))}^{2(1-\theta)p}\big)^{1/2}
	\big(\E\|u_i^\eps\|_{L^2(0,T;H^1(\dom))}^{4p/s}\big)^{1/2} \leq C,
\end{align*}
where $r>1$ and $\theta\in(0,1]$ satisfy $1/r=1-\theta(d+2)/(2d)$ 
and $s=2/\theta\ge 2$. The right-hand side is
bounded by \eqref{5.L1p} and \eqref{5.H1p}.
Setting $r=s$ yields estimate \eqref{5.22d}, implying that $r=2+2/d$, and
\eqref{5.24d} follows from the choice $s=2+4/d$, implying that
$r=2$. 
\end{proof}

Next, we prove bounds on the fractional time derivative of $u^\eps$. In combination with the previous estimates, this will allow us to establish the tightness of the laws of $(u^\eps)$ in a suitable space. For a vector space $X$, and constants $p\ge 1$, $\alpha\in(0,1)$, the Sobolev--Slobodeckij space $W^{\alpha,p}(0,T;X)$ is the set of
all functions $v\in L^p(0,T;X)$ for which 
\begin{align*}
  \|v\|_{W^{\alpha,p}(0,T;X)}^p 
	&= \|v\|_{L^p(0,T;X)}^p + |v|_{W^{\alpha,p}(0,T;X)}^p \\
	&= \int_0^T\|v\|_X^p \dt
	+ \int_0^T\int_0^T\frac{\|v(t)-v(s)\|_X^p}{|t-s|^{1+\alpha p}}\dt\ds< \infty.
\end{align*}
With this norm, $W^{\alpha,p}(0,T;X)$ becomes a Banach space.
For $\beta \in (0,1)$, we also introduce the H{\"o}lder space $C^{0,\beta}([0,T],X)$ of continuous, $X$--valued functions $v$ which satisfy
\begin{align*}
  \|v\|_{C^{0,\beta}([0,T];X)}^p 
	&= \|v\|_{C([0,T];X)}^p + |v|_{C^{\beta}([0,T];X)}^p \\
	&= \sup_{t\in [0,T]}\|v(t)\|_X^p
	+ \sup_{s,t \in [0,T], s\neq t}\frac{\|v(t)-v(s)\|_X^p}{|t-s|^{\beta p}}< \infty.
\end{align*}
\begin{lemma}[{\citealp[Lemma 21]{braukhoff2024global}}]\label{lem.g}
Let $g\in L^1(0,T)$ and $\delta<2$, $\delta\neq 1$. Then
\begin{equation}\label{5.int}
  \int_0^T\int_0^T|t-s|^{-\delta}\int_{s\wedge t}^{t\lor s}g(r)\dr\dt\ds
	< \infty.
\end{equation}
\end{lemma}

\begin{lemma}[Time regularity]\label{lem.frac}
There exists $0<\beta$ and a constant $C(u^0,T)>0$ such that, for $p:=\frac{d+2}{d+1}$,
\begin{align}
  \E\|u^\eps\|_{C^{0,\beta}(0,T;D(L)')}^p &\leq C(u^0,T), \nonumber\\ 
	\eps^p\E\|L^*LR_\eps(v^\eps)\|_{C^{0,\beta}(0,T;D(L)')}^p 
	+ \E\|v^\eps\|_{C^{0,\beta}(0,T;D(L)')}^p &\leq C(u^0,T). \label{5.LLR}
\end{align}
\end{lemma}

Combining the three time-regularity estimates established in the proof below -- one for each of the terms $J_1,J_2,J_3$, with respective exponent pairs $(\alpha_j,p_j)$ satisfying $\alpha_jp_j>1$ and the embedding $W^{\alpha_j,p_j}(0,T)\hookrightarrow C^{0,\beta_j}([0,T])$ with $\beta_j=\alpha_j-1/p_j>0$ -- and setting $\beta:=\min\{\beta_1,\beta_2,\beta_3\}>0$, we obtain, with $p:=\frac{d+2}{d+1}$, 
\begin{equation}\label{5.C0}
  \E\|u^\eps\|_{C^{0,\beta}([0,T];D(L)')}^p\leq C(u^0,T).
\end{equation}

\begin{proof}
We recall the continuous embedding $W^{\alpha_j,p_j}(0,T)\hookrightarrow C^{0,\beta_j}([0,T])$
for $\beta_j=\alpha_j-1/p_j>0$, which will be used throughout the proof (with $\alpha_j p_j>1$ chosen separately for each term $J_j$ below; note $p_j$ need not exceed $2$, so $\alpha_j$ is in general close to $1$ rather than below $1/2$). 

First, we derive the $W^{\alpha,p}$ estimate for $v^\eps$ and then we conclude
the estimate for $u^\eps$ from the definition 
$v^\eps = u^\eps + \eps L^*LR_\eps(v^\eps)$
and Lemma \ref{lem.hom2}.
We know from \eqref{5.vp} that $\E\|v^\eps\|_{L^\infty(0,T;D(L)')}^p$
is bounded. Hence it remains to bound the seminorm-part of the $W^{\alpha,p}(0,T;D(L)')$ norm:
\begin{align*}
  \E|v_i^\eps&|_{W^{\alpha,p}(0,T;D(L)')}^p
	= \E\int_0^T\int_0^T\frac{\|v_i^\eps(t)-v_i^\eps(s)\|_{D(L)'}^p}{
	|t-s|^{1+\alpha p}}\dt\ds.
\end{align*}
Instead of bounding each term in the same seminorm, we will perform a more fine-grained analysis:
\begin{align*}
& J_{1}:=\E\int_0^T\int_0^T|t-s|^{-1-\alpha_{1} p_{1}}\left\|\int_{s\wedge t}^{t\lor s}
	\diverm\sum_{j=1}^n A_{ij}(u^\eps(r))\na u_j^\eps(r)\dr\right\|_{D(L)'}^{p_{1}}
	\dt\ds\\
	&J_{2}:= \E\int_0^T\int_0^T\frac{
	\left\|\sqrtcorrection\int_{s\wedge t}^{t\lor s}
	\left(\diverm\left(\sum_{j=1}^n\sigma_{\delta}(u^\eps(r))\dW(r)\right)\right)_{i}\right\|_{D(L)'}^{p_{2}}}{|t-s|^{1+\alpha_{2} p_{2}}}\dt\ds\\
 &J_{3}:=\E\int_0^T\int_0^T|t-s|^{-1-\alpha_{3} p_{3}}\left\|\lambda\correction\int_{s\wedge t}^{t\lor s}
	\mathcal{T}(u^{\eps}(r))_{i}\dr\right\|_{D(L)'}^{p_{3}}
	\dt\ds.
\end{align*}
The bound for $J_{1}$, we set
$$
  g(t) = \int_0^t \left\|\left(a_{i0}+2\sum_{j=1}^n a_{ij}u_j^\eps\right)\na u_i^\eps
	+ \sum_{j\neq i}a_{ij}u_i^\eps\na u_j^\eps\right\|_{L^1(\dom)}\dd r.
$$
Then, using $D(L)\subset W^{1,\infty}(\dom)$ (due to the assumption
$m>d/2+1$),
\begin{align*}
  \E&\int_0^T\int_0^T|t-s|^{-1-\alpha_{1} p_{1}}\left\|\int_{s\wedge t}^{t\lor s}
	\diver\sum_{j=1}^n A_{ij}(u^\eps(r))\na u_j^\eps(r)\dd r\right\|_{D(L)'}^{p_{1}}\dd t\dd s \\
	&\leq C\int_0^T\int_0^T|t-s|^{-1-\alpha_{1} p_{1}}\\
	&\phantom{xx}\times\left(\int_{s\wedge t}^{t\lor s}
	 \left\|\left(a_{i0}+2\sum_{j=1}^n a_{ij}u_j^\eps\right)\na u_i^\eps
	+ \sum_{j\neq i}a_{ij}u_i^\eps\na u_j^\eps\right\|_{L^1(\dom)}\dd r\right)^{p_{1}}\dd t\dd s \\
	&\leq C\E\int_0^T\int_0^T\frac{|g(t)-g(s)|^p}{|t-s|^{1+\alpha_{1} p_{1}}}\dd t\dd s
	\le C\E\|g\|_{W^{\alpha_{1},p_{1}}(0,T;\R)}^p.
\end{align*}
The embedding $W^{1,p_{1}}(0,T;\R)\hookrightarrow W^{\alpha_{1},p_{1}}(0,T;\R)$ 
and estimates \eqref{5.H1p}, \eqref{5.24d}  
show that for $1\leq p_{1}\leq (d+2)/(d+1)$,
\begin{align*}
  \E&\|g\|_{W^{\alpha_{1},p_{1}}(0,T;\R)}^{p_{1}}
	\le C\E\|g\|_{W^{1,p_{1}}(0,T;\R)}^{p_{1}}
	= C\E\|\pa_t g\|_{L^p(0,T;\R)}^{p_{1}} + C\E\|g\|_{L^{p_{1}}(0,T;\R)}^{p_{1}} \\
	&\le C\E\int_0^T \left\|\left(a_{i0}+2\sum_{j=1}^n a_{ij}u_j^\eps\right)\na u_i^\eps
	+ \sum_{j\neq i}a_{ij}u_i^\eps\na u_j^\eps\right\|_{L^1(\dom)}^{p_{1}} \dd t \\
	&\phantom{xx}{}+ C\E\int_0^T\int_0^t \left\|\left(a_{i0}+2\sum_{j=1}^n a_{ij}u_j^\eps\right)\na u_i^\eps
	+ \sum_{j\neq i}a_{ij}u_i^\eps\na u_j^\eps\right\|_{L^1(\dom)}^{p_{1}} \dd r\dd t
	\le C.
\end{align*}
Hence, we can choose a $\alpha_{1}$, such that $\alpha_{1}p_{1}>1$ and in turn a $\beta_{1}>0$, such that $W^{\alpha_{1},p_{1}}(0,T)\hookrightarrow C^{0,\beta_{1}}([0,T])$.

In the following estimates, we will neglect the parameters $\sqrtcorrection$ and $\correction$ for simplicity.
To estimate $J_2$, we use the embedding $L^2(\dom)\hookrightarrow W^{-1,2}(\dom)\hookrightarrow D(L)'$, the
Burkholder--Davis--Gundy inequality, and the H\"older inequality:
To estimate $J_2$, fix $s<t$ and $i$. Since $\|\partial_{x_l}f\|_{D(L)'}\le C\|f\|_{L^2(\dom)}$ for $f\in L^2(\dom)$, the Burkholder--Davis--Gundy inequality gives
\begin{align*}
  &\E\Big\|\sum_{k,l=1}^{\infty,d}\int_s^t\partial_{x_l}\big((\sigma_\delta)_{ii}(u^\eps(r))e^{il}_k\big)\dd W^{il}_k(r)\Big\|_{D(L)'}^{p_2}\\
  &\qquad\le C\,\E\bigg(\int_s^t\sum_{k,l=1}^{\infty,d}\big\|(\sigma_\delta)_{ii}(u^\eps(r))\,e^{il}_k\big\|_{L^2(\dom)}^2\dr\bigg)^{p_2/2}.
\end{align*}
We use $(\sigma_\delta)_{ii}(u)^2\le u_i\Ared_i(u)$, $\sum_{k,l}(e^{il}_k)^2\le dE_\infty$ and $\int_\dom u_i\Ared_i(u)\dx\le C(1+\|u\|_{L^2(\dom)}^2)$. Then H\"older's inequality in time gives
\[
  \E\bigg(\int_s^t\sum_{k,l=1}^{\infty,d}\big\|(\sigma_\delta)_{ii}(u^\eps(r))\,e^{il}_k\big\|_{L^2(\dom)}^2\dr\bigg)^{p_2/2}
  \le C\,|t-s|^{\frac{p_2}2-1}\int_s^t\E\big(1+\|u^\eps(r)\|_{L^2(\dom)}^{p_2}\big)\dr .
\]
Therefore
\[
  J_2\le C\int_0^T\int_0^T|t-s|^{-(2+\alpha_2p_2-\frac{p_2}2)}\int_{s\wedge t}^{t\vee s}\E\big(1+\|u^\eps(r)\|_{L^2(\dom)}^{p_2}\big)\dr\dt\ds .
\]
We choose $p_2=2+\frac4d$, so that $\E\int_0^T\|u^\eps(r)\|_{L^2(\dom)}^{p_2}\dr$ is bounded by \eqref{5.24d}. By \eqref{5.int}, the right-hand side is finite if $2+\alpha_2p_2-\frac{p_2}2<2$ (and $\neq1$), i.e.\ if $\alpha_2<\frac12$. Since $p_2>2$, we can choose $\alpha_2\in(\frac1{p_2},\frac12)$, so that $\alpha_2p_2>1$. This yields $\beta_2=\alpha_2-\frac1{p_2}>0$ such that $W^{\alpha_2,p_2}(0,T)\hookrightarrow C^{0,\beta_2}([0,T])$.

Lastly, we bound the correction term. In this step, we set $g(t)=\int_{0}^{t}\left\|
	\mathcal{T}(u^{\eps}(r))_{i}\right\|_{D(L)'}\dr$.
\begin{align*}
    J_{3}&=\E\int_0^T\int_0^T|t-s|^{-1-\alpha_{3} p_{3}}\left\|\int_{s\wedge t}^{t\lor s}
	\mathcal{T}(u^{\eps}(r))_{i}\dr\right\|_{D(L)'}^{p_{3}}
	\dt\ds\\
 &\leq C\E\int_0^T\int_0^T|t-s|^{-1-\alpha_{3} p_{3}}\left(\int_{s\wedge t}^{t\lor s}\left\|
	\mathcal{T}(u^{\eps}(r))_{i}\right\|_{D(L)'}\dr\right)^{p_{3}}
	\dt\ds\\
	&\leq C\E\int_0^T\int_0^T\frac{|g(t)-g(s)|^{p_{3}}}{|t-s|^{1+\alpha_{3} p_{3}}}\dd t\dd s
	\le C\E\|g\|_{W^{\alpha_{3},p_{3}}(0,T;\R)}^{p_{3}}.
\end{align*}
Again by the embedding $W^{1,p}(0,T;\R)\hookrightarrow W^{\alpha,p}(0,T;\R)$, we obtain
\begin{align*}
  \E&\|g\|_{W^{\alpha_{3},p_{3}}(0,T;\R)}^{p_{3}}
	\le C\E\|g\|_{W^{1,p_{3}}(0,T;\R)}^{p_{3}}
	= C\E\|\pa_t g\|_{L^{p_3}(0,T;\R)}^{p_{3}} + C\E\|g\|_{L^{p_3}(0,T;\R)}^{p_{3}} \\
	&\leq C\E\int_0^T\left\|\mathcal{T}(u^{\eps}(t))_{i}\right\|_{D(L)'}^{p_{3}} \dd t + C\E\int_0^T\int_0^t\left\|\mathcal{T}(u^{\eps}(r))_{i}\right\|_{D(L)'}^{p_{3}} \dd r\dd t.
\end{align*}
To bound these terms we use the closed form \eqref{eqn:T_closed_form}, which gives, with $\vartheta_i(u):=\Ared_i(u)+a_{ii}u_i\le2\Ared_i$, $F^{il}_1\le E_\infty$ and $|F^{il}_2|\le(E_\infty E_\infty')^{1/2}$,
\begin{align*}
  \|\mathcal T(u^\eps)_i\|_{D(L)'}&\le C\Big(\|\vartheta_i(u^\eps)\|_{L^2(\dom)}\sum_{l=1}^d\big\|\sqrt{F^{il}_1}\,\partial_{x_l}\log u_i^\eps\big\|_{L^2(\dom)}\\
  &\phantom{xx}+\|\na u^\eps\|_{L^1(\dom)}+\|\vartheta_i(u^\eps)\|_{L^1(\dom)}\Big).
\end{align*}
Let $p_3:=\frac{d+2}{d+1}$, so that $\frac{p_3}{2+4/d}+\frac{p_3}2=1$. By H\"older's inequality in time,
\begin{align*}
  \int_0^T\|\mathcal T(u^\eps)_i\|_{D(L)'}^{p_3}\dt&\le C\big(1+\|u^\eps\|_{L^{2+4/d}(0,T;L^2(\dom))}\big)^{p_3}\\
  &\phantom{xx}\times\sum_{l=1}^d\big\|\sqrt{F^{il}_1}\,\partial_{x_l}\log u_i^\eps\big\|_{L^2(Q_T)}^{p_3}\\
  &\phantom{xx}+C\big(1+\|u^\eps\|_{L^2(0,T;H^1(\dom))}\big)^{p_3}.
\end{align*}
By H\"older's inequality in $\omega$ with exponents $\frac{p-p_3}{p}$ and $\frac{p_3}p$, the expectation of the right-hand side is bounded by \eqref{5.24d} of order $\frac{pp_3}{p-p_3}\le\frac p2$, which holds since $p\ge p_0(d)\ge3p_3$, and by \eqref{5.H1p} and \eqref{eqn:higher_oder_bounds_nabla_log_u}. Hence $g\in W^{1,p_3}(0,T)$ with bounded $p_3$-th moment of its norm, and $W^{1,p_3}(0,T)\hookrightarrow C^{0,\beta_3}([0,T])$ with $\beta_3:=1-\frac1{p_3}=\frac1{d+2}$.

We now simply set $\beta:=\min\{\beta_{1},\beta_{2},\beta_{3}\}$.
The uniform bounds for $u^\eps$ follow by the definition of $v^\eps$
and the $C^{\beta}$ seminorm,
\begin{align*}
  \E|u^\eps|_{C^{\beta}([0,T];D(L)')}^p
	&= \E|v^\eps - \eps L^*LR_\eps(v^\eps)|_{C^{\beta}([0,T];D(L)')}^p \\
	&\leq C\E\sup_{s,t \in [0,T], s\neq t}\frac{\|v^\eps(t)-v^\eps(s)\|_{D(L)'}^p}{|t-s|^{\beta p}}\\
	&\phantom{xx}{}+ C\E\sup_{s,t \in [0,T], s\neq t}	\frac{\eps^p\|L^*LR_\eps(v^\eps(t))
	-L^*LR_\eps(v^\eps(s))\|_{D(L)'}^p}{|t-s|^{\beta p}}.
\end{align*}
It follows from \eqref{3.LL} 
and the Lipschitz continuity of $R_\eps$ (Lemma \ref{lem.Reps}) that
\begin{align*}
  \|L^*LR_\eps(v^\eps(t))-L^*LR_\eps(v^\eps(s))\|_{D(L)'}
	&\leq \|R_\eps(v^\eps(t))-R_\eps(v^\eps(s))\|_{L^2(\dom)} \\
	&\leq \eps^{-1}C\|v^\eps(t)-v^\eps(s)\|_{D(L)'}.
\end{align*}
Then we find that
$$
  \E|u^\eps|_{C^{\beta}([0,T];D(L)')}^p
	\leq  C\E|v^\eps|_{C^{\beta}([0,T];D(L)')}^p.
$$   

\end{proof}


\subsection{Tightness of the laws of \texorpdfstring{$(u^\eps)$}{(u-epsilon)}}\label{sec.tight}

We show that the laws of $(u^\eps)$ are tight in a certain sub-Polish space. 
For this, we introduce the following spaces:
\begin{itemize}
\item $C^0([0,T];D(L)')$ is the space of continuous functions $u:[0,T]\to D(L)'$
with the topology $\mathbb{T}_1$ induced by the norm $\|u\|_{C^0([0,T];D(L)')}
=\sup_{0<t<T}\|u(t)\|_{D(L)'}$;
\item $L_w^2(0,T;H^1(\dom))$ is the space $L^2(0,T;H^1(\dom))$ with the weak topology
$\mathbb{T}_2$.
\end{itemize}
We define the space
$$
  \widetilde{Z}_T := C^0([0,T];D(L)')\cap L_w^2(0,T;H^1(\dom)),
$$
endowed with the topology $\widetilde{\mathbb{T}}$ that is the maximum of the
topologies $\mathbb{T}_1$ and $\mathbb{T}_2$. The space $\widetilde{Z}_T$
is a sub-Polish space, since $C^0([0,T];D(L)')$ is 
separable and metrizable and 
$$
  f_m(u) = \int_0^T(u(t),v_m(t))_{H^1(\dom)}\dt, \quad
	u\in L_w^2(0,T;H^1(\dom)),\ m\in\N,
$$
where $(v_m)_m$ is a dense subset of $L^2(0,T;H^1(\dom))$,
is a countable family $(f_m)$ of
point-separating functionals acting on $L^2(0,T;H^1(\dom))$.
In the following, we choose a number $s^*\ge 1$ such that
\begin{equation}\label{5.sstar}
  s^* < \frac{2d}{d-2}\quad\mbox{if }d\ge 3, \quad s^*<\infty\quad\mbox{if }d=2, \quad
	s^* \leq \infty\quad\mbox{if }d=1.
\end{equation}
Then the embedding $H^1(\dom)\hookrightarrow L^{s^*}(\dom)$ is compact.

\begin{lemma}\label{lem.tight1}
The set of laws of $(u^\eps)$ is tight in 
$$
  Z_T = \widetilde{Z}_T\cap L^2(0,T;L^{s^*}(\dom))
$$
with the topology $\mathbb{T}$ that is the maximum of $\widetilde{\mathbb{T}}$ 
and the topology induced by the $L^2(0,T;$ $L^{s^*}(\dom))$ norm, where
$s^*$ is given by \eqref{5.sstar}. 
\end{lemma}

\begin{proof}
We apply Chebyshev's inequality for the first moment and use estimate
\eqref{5.C0}, for any $\eta>0$ and $\delta>0$,
\begin{align*}
  \sup_{\eps>0}\,&\Prob\bigg(\sup_{\substack{s,t\in[0,T], \\ |t-s|\le\delta}}
	\|u^\eps(t)-u^\eps(s)\|_{D(L)'}>\eta\bigg)
	\leq \sup_{\eps>0}\frac{1}{\eta}\E\bigg(
	\sup_{\substack{s,t\in[0,T], \\ |t-s|\le\delta}}
	\|u^\eps(t)-u^\eps(s)\|_{D(L)'}\bigg) \\
	&\leq \frac{\delta^\beta}{\eta}\sup_{\eps>0}\E\bigg(\sup_{\substack{s,t\in[0,T], \\ 
	|t-s|\le\delta}}\frac{\|u^\eps(t)-u^\eps(s)\|_{D(L)'}}{|t-s|^\beta}\bigg)
	\leq \frac{\delta^\beta}{\eta}\sup_{\eps>0}\E\|u^\eps\|_{C^{0,\beta}([0,T];D(L)'))}
	\leq C\frac{\delta^\beta}{\eta}.
\end{align*}
This means that for all $\theta>0$ and
all $\eta>0$, there exists $\delta>0$ such that
$$
  \sup_{\eps>0}\,\Prob\bigg(\sup_{s,t\in[0,T], \, |t-s|\le\delta}
	\|u^\eps(t)-u^\eps(s)\|_{D(L)'}>\eta\bigg) \leq \theta,
$$
which is equivalent to the Aldous 
condition \citep[Section 2.2]{BrMo14}. Applying \citep[Lemma 5, Theorem 3]{Sim87}
with the spaces $X=H^1(\dom)$ and $B=D(L)'$, we conclude that
$(u^\eps)$ is precompact in $C^0([0,T];D(L)')$. 
Then, proceeding as in the proof of the basic criterion for tightness 
\citep[Chapter II, Section 2.1]{Met88}, we see that 
the set of laws of $(u^\eps)$ is tight in $C^0([0,T];D(L)')$.

Next, by Chebyshev's inequality again and estimate \eqref{5.H1}, for all $K>0$,
$$
  \Prob\big(\|u^\eps\|_{L^2(0,T;H^1(\dom))} > K\big)
	\leq \frac{1}{K^2}\E\|u^\eps\|_{L^2(0,T;H^1(\dom))}^2 \leq \frac{C}{K^2}.
$$
This implies that for any $\delta>0$, there exists $K>0$ such that
$\Prob(\|u^\eps\|_{L^2(0,T;H^1(\dom))}\leq K)\leq 1-\delta$. Since closed balls in the norm of $L^2(0,T;H^1(\dom))$ are weakly compact, we infer that
the set of laws of $(u^\eps)$ is tight in $L^2_w(0,T;H^1(\dom))$. 

The tightness in $L^2(0,T;L^{s^*}(\dom))$ follows from standard Sobolev embedding arguments. 
\end{proof}

\begin{lemma}\label{lem.tighteps}
The set of laws of $(\sqrt{\eps}L^*LR_\eps(v^\eps))$ is tight in 
$$
  Y_T:= L_w^2(0,T;D(L)')\cap L^\infty_{w*}(0,T;D(L)')
$$ 
with the associated topology $\mathbb{T}_Y$. In addition, the laws of $\mathcal G^\eps:=\big(\sqrt{F^{il}_1}\,\partial_{x_l}R_\eps(v^\eps)_i\big)_{i,l}$ are tight in $L^{2}_{w}(0,T;L^{2}(\dom)^{n\times d})$ (equipped with the weak topology).
\end{lemma}

\begin{proof}
We apply the Chebyshev inequality and use the inequality
$\|L^*LR_\eps(v^\eps)\|_{D(L)'}\leq C\|LR_\eps(v^\eps)\|_{L^2(\dom)}=C\|R_\eps(v^\eps)\|_{D(L)}$ and
estimate \eqref{5.v}:
$$
  \Prob\big(\sqrt{\eps}\|L^*LR_\eps(v^\eps)\|_{L^2(0,T;D(L)')}>K\big)
	\leq \frac{\eps}{K^2}\E\|L^*LR_\eps(v^\eps)\|_{L^2(0,T;D(L)')}^2
	\leq \frac{C}{K^2}
$$
for any $K>0$. Since closed balls in $L^2(0,T;D(L)')$ are weakly compact, 
the set of laws of $(\sqrt{\eps}L^*LR_\eps(v^\eps))$ is tight in $L_w^2(0,T;D(L)')$.
The remaining claims follow from an analogous argument, using \eqref{eqn:higher_oder_bounds_nabla_log_u} for $\mathcal G^\eps$.
\end{proof}

\subsection{Convergence of \texorpdfstring{$(u^\eps)$}{(u-epsilon)}}\label{sec.conv1}

Let $\mbox{P}(X)$ be the space of probability measures on $X$.
We consider the space $Z_T\times Y_T\times L^{2}_{w}(0,T;L^{2}(\dom)^{n\times d})\times C^0([0,T];U_0)$, equipped with the
probability measure $\mu^\eps$, the joint law of $(u^\eps,\sqrt\eps L^*LR_\eps(v^\eps),\mathcal G^\eps,W)$, whose marginals are
\begin{align*}
  \mu_u^\eps(\cdot) &= \Prob(u^\eps\in\cdot)\in \mbox{P}(Z_T), \\
	\mu_w^\eps &= \Prob(\sqrt{\eps} L^*LR_\eps(v^\eps)\in\cdot)\in
	\mbox{P}(Y_T), \\
    \mu_{\mathcal G}^\eps&=\Prob( \mathcal G^\eps\in\cdot)\in
	\mbox{P}(L^{2}_{w}(0,T;L^{2}(\dom)^{n\times d})),\\
	\mu_W^\eps(\cdot) &= \Prob(W\in\cdot) \in \mbox{P}(C^0([0,T];U_0)),
\end{align*}
recalling the choice \eqref{5.sstar} of $s^*$.

The family of measures $(\mu^\eps)$ is tight since the marginals are tight, i.e. the laws of the sequences $(u^\eps)$, $(\sqrt{\eps}L^*LR_\eps(v^\eps))$, and $(\mathcal G^\eps)$ are tight in the respective spaces $(Z_T,\mathbb{T})$, $(Y_T,\mathbb{T}_Y)$, and $L^{2}_{w}(0,T;L^{2}(\dom)^{n\times d})$. Furthermore, $(\mu_W^\eps)$ is a singleton and, therefore, weakly compact in $C^0([0,T];U_0)$. By Prokhorov's theorem, $(\mu_W^\eps)$ is also tight. Consequently, the product space $Z_T\times Y_T\times L^{2}_{w}(0,T;L^{2}(\dom)^{n\times d})\times C^0([0,T];U_0)$ satisfies the assumptions of the Skorokhod--Jakubowski theorem \citep[Theorem C.1]{BrOn10}.  

Applying this theorem, we obtain a subsequence of $(u^\eps,\sqrt{\eps}L^*LR_\eps(v^\eps))$ (not relabeled), along with a probability space $(\widetilde{\Omega},\widetilde{\mathcal F},\widetilde{\Prob})$ and random variables $(\widetilde{u},\widetilde{w}, \widetilde{\mathcal G},\widetilde{W})$ and $(\widetilde{u}^\eps, \widetilde{w}^\eps,\widetilde{\mathcal G}^\eps, \widetilde{W}^\eps)$ taking values in $Z_T\times Y_T\times L^{2}_{w}(0,T;L^{2}(\dom)^{n\times d}) \times C^0([0,T];U_0)$. These random variables satisfy the property that $(\widetilde{u}^\eps, \widetilde{w}^\eps, \widetilde{\mathcal G}^\eps, \widetilde{W}^\eps)$ has the same law as $(u^\eps,\sqrt{\eps}L^*LR_\eps(v^\eps),\mathcal G^\eps,W)$ on ${\mathcal B}(Z_T\times Y_T\times L^{2}_{w}(0,T;L^{2}(\dom)^{n\times d}) \times C^0([0,T];U_0))$. Furthermore, as $\eps \to 0$, we have  
\begin{align*}
  &(\widetilde{u}^\eps,\widetilde{w}^\eps,\widetilde{\mathcal G}^\eps,\widetilde{W}^\eps) \to 
	(\widetilde{u},\widetilde{w},\widetilde{\mathcal G}, \widetilde{W})\quad\widetilde{\Prob}\mbox{-a.s. in}\\
	&\qquad Z_T\times Y_T\times L^{2}_{w}(0,T;L^{2}(\dom)^{n\times d})\times C^0([0,T];U_0).
\end{align*}  

By the definitions of $Z_T$ and $Y_T$, this convergence implies the following $\widetilde{\Prob}$-a.s. limits:  
\begin{align*}
  \widetilde{u}^\eps &\to \widetilde{u} \quad\mbox{strongly in }C^0([0,T];D(L)'), \\
	\widetilde{u}^\eps &\rightharpoonup \widetilde{u} \quad\mbox{weakly in } L^2(0,T;H^1(\dom)), \\
	\widetilde{u}^\eps &\to \widetilde{u} \quad\mbox{strongly in } L^2(0,T;L^{s^*}(\dom)), \\
	\widetilde{w}^\eps &\rightharpoonup \widetilde{w} \quad\mbox{weakly in } L^2(0,T;D(L)'), \\
	\widetilde{\mathcal G}^\eps &\rightharpoonup \widetilde{\mathcal G} \quad\mbox{weakly in } L^2(0,T;L^{2}(\dom)^{n\times d}), \\
	\widetilde{w}^\eps &\rightharpoonup \widetilde{w} \quad\mbox{weakly* in } L^\infty(0,T;D(L)'), \\
	\widetilde{W}^\eps &\to \widetilde{W} \quad\mbox{strongly in }C^0([0,T];U_0).
\end{align*}  

The third component deserves a comment. By construction $\widetilde{\mathcal G}^\eps$ has the law of $\mathcal G^\eps$, whose components are $\pi_i\sqrt{F^{il}_1}\,\partial_{x_l}\log u_i^\eps$. It is the weighted gradient of the entropy variable, not of the $Y_T$-component $\widetilde w^\eps=\sqrt\eps L^*LR_\eps(v^\eps)$. The joint laws of $(\widetilde u^\eps,\widetilde{\mathcal G}^\eps)$ and $(u^\eps,\mathcal G^\eps)$ coincide, and the set of pairs satisfying $\widetilde{\mathcal G}^\eps_{il}=\pi_i\sqrt{F^{il}_1}\,\partial_{x_l}\log\widetilde u_i^\eps$ is Borel in $Z_T\times L^2_w(0,T;L^2(\dom)^{n\times d})$. Hence this relation holds $\widetilde\Prob$-a.s. We set $\mathcal G_{il}:=\pi_i^{-1}\widetilde{\mathcal G}_{il}$. It vanishes on $\{F^{il}_1=0\}$, as a weak limit of functions vanishing there. Multiplying the relation above by $\widetilde u_i^\eps$ gives $\widetilde u_i^\eps\,\widetilde{\mathcal G}^\eps_{il}=\pi_i\sqrt{F^{il}_1}\,\partial_{x_l}\widetilde u_i^\eps$. Here $\widetilde u_i^\eps\to\widetilde u_i$ strongly in $L^2(Q_T)$ (we take $s^*\ge2$ in \eqref{5.sstar}), $\widetilde{\mathcal G}^\eps_{il}\rightharpoonup\widetilde{\mathcal G}_{il}$ and $\partial_{x_l}\widetilde u_i^\eps\rightharpoonup\partial_{x_l}\widetilde u_i$ weakly in $L^2(Q_T)$, and $\sqrt{F^{il}_1}$ is bounded. Passing to the limit in $L^1(Q_T)$-weak gives the compatibility relation \eqref{2.Gcompat}, $\widetilde\Prob$-a.s. In particular, $\sqrt{F^{il}_1}\,\partial_{x_l}\log\widetilde u_i^\eps\rightharpoonup\mathcal G_{il}$ weakly in $L^2(Q_T)$, $\widetilde\Prob$-a.s.

To establish regularity properties for the limit function $\widetilde{u}$, we observe that $\widetilde{u}$ is a $Z_T$-Borel random variable, since ${\mathcal B}(Z_T\times Y_T\times C^0([0,T];U_0))$ is contained in ${\mathcal B}(Z_T)\times{\mathcal B}(Y_T)\times{\mathcal B}(C^0([0,T];U_0))$. From estimates \eqref{5.L1} and \eqref{5.H1}, together with the fact that $u^\eps$ and $\widetilde{u}^\eps$ have the same law, we obtain  
\begin{align*}
  \sup_{\eps>0}\widetilde{\E}\|\widetilde{u}^\eps\|_{L^2(0,T;H^1(\dom))}^p
	+ \sup_{\eps>0}\widetilde{\E}\|\widetilde{u}^\eps\|_{L^\infty(0,T;D(L)')}^p 
	< \infty.
\end{align*}  
Thus, there exists a further subsequence of $(\widetilde{u}^\eps)$ (not relabeled) that converges weakly in $L^p(\widetilde{\Omega};L^2(0,T;H^1(\dom)))$ and weakly* in $L^p(\widetilde{\Omega};C^0([0,T];D(L)'))$ as $\eps\to 0$. Since $\widetilde{u}^\eps \to \widetilde{u}$ in $Z_T$ $\widetilde{\Prob}$-a.s., it follows that  
\begin{align*}
  \widetilde{\E}\|\widetilde{u}\|_{L^2(0,T;H^1(\dom))}^p
	+ \widetilde{\E}\|\widetilde{u}\|_{L^\infty(0,T;D(L)')}^p < \infty.
\end{align*}  

Let $\widetilde{\F}$ and $\widetilde{\F}^\eps$ denote the filtrations generated by the restrictions to $[0,t]$ of $(\widetilde{u},\widetilde{w},\widetilde{\mathcal G},\widetilde{W})$ and $(\widetilde{u}^\eps,\widetilde{w}^\eps,\widetilde{\mathcal G}^\eps,\widetilde{W})$, respectively. Following the arguments of \citet[Proposition B4]{BrOn13}, we verify that these variables define stochastic processes. The progressive measurability of $\widetilde{u}^\eps$ follows from \citet[Appendix B]{BHM13}. Setting $\widetilde{W}_j^{\eps,k}(t):=\langle\widetilde{W}^\eps(t),e_k\rangle_{U}$, we claim that for each $k\in\N$, $\widetilde{W}_j^{\eps,k}(t)$ forms an independent, standard $\widetilde{\mathcal F}_t$-Wiener process. The adaptedness follows by definition, while the independence of increments and the limiting characteristic function imply that $\widetilde{W}(t)$ is a spatially colored Wiener ($U$-cylindrical) process by L\'evy's characterization theorem.


By definition, $u_i^\eps=u_i(R_\eps(v^\eps))=\exp(R_\eps(v^\eps)/\pi_i)$ is positive in 
$Q_T$ a.s. We claim that also $\widetilde{u}_i$ is nonnegative in $\dom$ a.s.

\begin{lemma}[Nonnegativity]
It holds that $\widetilde{u}_i\ge 0$ a.e.\ in $Q_T$
$\widetilde{\Prob}$-a.s. for all $i=1,\ldots,n$.
\end{lemma}

\begin{proof}
Let $i\in\{1,\ldots,n\}$. Since $u_i^\eps > 0$ in $Q_T$ a.s., we have
$\E\|(u_i^\eps)^-\|_{L^2(0,T;L^2(\dom))}=0$, where $z^-=\min\{0,z\}$.
The function $u_i^\eps$ is $Z_T$-Borel measurable, and so is its negative part.
Therefore, using the equivalence of the laws of $u_i^\eps$ and $\widetilde{u}_i^\eps$
in $Z_T$ and writing $\mu_i^\eps$ and $\widetilde{\mu}_i^\eps$ for the laws of
$u_i^\eps$ and $\widetilde{u}_i^\eps$, respectively, we obtain
\begin{align*}
  \widetilde{\E}\|(\widetilde{u}_i^\eps)^-\|_{L^2(Q_T)}
	&= \int_{L^2(Q_T)}\|y^-\|_{L^2(Q_T)}\dd\widetilde{\mu}_i^\eps(y)\\
	&= \int_{L^2(Q_T)}\|y^-\|_{L^2(Q_T)}\dd\mu_i^\eps(y)
	= \E\|(u_i^\eps)^-\|_{L^2(Q_T)} = 0.
\end{align*}
This shows that $\widetilde{u}_i^\eps\ge 0$ a.e.\ in $Q_T$ $\widetilde{\Prob}$-a.s.
The convergence (up to a subsequence) $\widetilde{u}^\eps\to \widetilde{u}$
a.e.\ in $Q_T$ $\widetilde{\Prob}$-a.s.\ then implies that $\widetilde{u}_i\ge 0$
in $Q_T$ $\widetilde{\Prob}$-a.s.
\end{proof}

The following lemma is needed to verify that $(\widetilde{u},\widetilde{W})$ 
is a martingale solution to \eqref{eq:SKT_intro}--\eqref{eq:SKT_intro_bic}.

\begin{lemma}\label{lem.E}
It holds for all $t\in[0,T]$, $i=1,\ldots,n$, and all
$\phi_1\in L^2(\dom)$ and all $\phi_2\in D(L)$ that
\begin{align}
  \lim_{\eps\to 0}\widetilde\E\int_0^T\big(\widetilde u_{i}^\eps(t)
	-\widetilde u_{i}(t),	\phi_1\big)_{L^2(\dom)}\dt&= 0, \label{5.E1} \\
	\lim_{\eps\to 0}\widetilde\E\big\langle\widetilde u_{i}^\eps(0)
	-\widetilde u_{i}(0),\phi_2\big\rangle_{D(L)',D(L)} &= 0, \label{5.E2} \\
	\lim_{\eps\to 0}\widetilde\E\int_0^T\big\langle\sqrt{\eps}\widetilde{w}_i^\eps(t),
	\phi_2\big\rangle_{D(L)',D(L)}\dt&= 0, \label{5.E3} \\
	\lim_{\eps\to 0}\widetilde\E\langle\sqrt{\eps}\widetilde{w}_i^\eps(0),
	\phi_2\rangle_{D(L)',D(L)} &= 0, \label{5.E3a} \\
	\lim_{\eps\to 0}\widetilde\E\int_0^T\bigg|\sum_{j=1}^n\int_0^t\int_\dom\big(
	A_{ij}(\widetilde u^\eps(s))\na\widetilde u_j^\eps(s)
	- A_{ij}(\widetilde u(s))\na\widetilde u_j(s)\big)\cdot
	\na\phi_2	\dx \ds\bigg|\dt&= 0, \label{5.E4} \\
	\lim_{\eps\to 0}\widetilde\E\int_0^T\bigg|\sum_{j=1}^n\int_0^t
	\Big(\sigma_{\delta}(\widetilde u^\eps(s))_{ij}\dd\widetilde W_j^\eps(s)
	-\sigma(\widetilde u(s))_{ij}\dd\widetilde W_j(s),\nabla \phi_2\Big)_{L^2(\dom)}\bigg|^2 
	\dt&= 0, \label{5.E5}\\
    \lim_{\eps\to 0}\widetilde{\E}\int_{0}^{T}\left|\int_{0}^{t}\left\langle \mathcal{T}(\widetilde{u}^\eps(s))_{i}-\mathcal{T}(\widetilde{u}(s))_{i},\phi_{2}\right\rangle\ds\right|\dt&=0
\end{align}
\end{lemma}

\begin{proof}
The proof is a combination of the uniform bounds and Vitali's 
convergence theorem. Convergences \eqref{5.E1} and \eqref{5.E2} have been shown
in the proof of \citet[Lemma 16]{DJZ19}, and \eqref{5.E3} is a direct consequence of
\eqref{5.LLR} and
$$
  \widetilde{\E}\bigg(\int_0^T\langle\sqrt{\eps}\widetilde{w}_i^\eps(t),\phi_2
	\rangle_{D(L)',D(L)}\dt\bigg)^p 
	\leq \eps^{p/2}\widetilde{\E}\bigg(\int_0^T\|\widetilde{w}_i^\eps(t)\|_{D(L)'}
	\|\phi_2\|_{D(L)}\dt\bigg)^p \leq \eps^{p/2}C.
$$
Convergence \eqref{5.E3a} follows from $\widetilde{w}_i^\eps\rightharpoonup
\widetilde{w}_i$ weakly* in $L^\infty(0,T;D(L)')$.
We establish \eqref{5.E4}:
\begin{align*}
  \int_0^T&\bigg|\sum_{j=1}^n\int_0^t\int_\dom\big(
	A_{ij}(\widetilde u^\eps(s))\na\widetilde u_j^\eps(s)
	- A_{ij}(\widetilde u(s))\na\widetilde u_j(s)\big)\cdot\na\phi_2\dx \ds\bigg|\dt \\
	&\leq \int_0^T\|A_{ij}(\widetilde u^\eps(s))-A_{ij}(\widetilde u(s))\|_{L^2(\dom)}
	\|\na \widetilde u_j^\eps(s)\|_{L^2(\dom)}\|\na\phi_2\|_{L^\infty(\dom)}\ds\\
	&\phantom{xx}{}+ \bigg|\int_0^T\int_\dom A_{ij}(\widetilde u(s))
	\na(\widetilde u^\eps(s) - \widetilde u(s))\cdot\na\phi_2 \dx \ds\bigg| 
	=: I_1^\eps+I_2^\eps.
\end{align*}
By the Lipschitz continuity of $A$ and the uniform bound for $\na\widetilde u^\eps$, 
we have $I_1^\eps\to 0$ as $\eps\to 0$ $\widetilde{\Prob}$-a.s. 
At this point, we use the embedding $D(L)\hookrightarrow W^{1,\infty}(\dom)$.
Also the second
integral $I_2^\eps$ converges to zero, since $A_{ij}(\widetilde u)\na\phi_2
\in L^2(0,T;L^2(\dom))$ and $\na\widetilde u_j^\eps\rightharpoonup\na\widetilde u_j$
weakly in $L^2(0,T;L^2(\dom))$. This shows that $\widetilde \Prob$-a.s., 
$$
  \lim_{\eps\to 0}\int_0^T\int_\dom A_{ij}(\widetilde u^\eps(s))
	\na\widetilde u_j^\eps(s)\cdot\na\phi_2\dx \ds
	= \int_0^T\int_\dom A_{ij}(\widetilde u(s))\na\widetilde u_j(s)
	\cdot\na\phi_2\dx \ds.
$$
A straightforward estimation and bound \eqref{5.H1p} lead to
\begin{align*}
  \widetilde \E&\bigg|\int_0^T\int_\dom A_{ij}(\widetilde u^\eps(s))
	\na\widetilde u_j^\eps(s)\cdot\na\phi_2\dx \ds\bigg|^p \\
	&\leq \|\na\phi_2\|_{L^\infty(\dom)}^p\widetilde \E\bigg(\int_0^T
	\left\|\sum_{j=1}^n A_{ij}(\widetilde u^\eps(s))\na\widetilde u_j^\eps(s)
	\right\|_{L^1(\dom)}\ds\bigg)^p \leq C,
\end{align*}
Hence, Vitali's convergence theorem gives \eqref{5.E4}. 

It remains to prove \eqref{5.E5}.

We have that $\widetilde \Prob$-a.s.,
\begin{align*}
      &\int_0^T\left\|\sigma_{ij}(\widetilde u^\eps(s))-\sigma_{ij}(\widetilde u(s))
	\right\|_{{\mathcal L}_2(U;L^2(\dom))}^2\dd s
	\leq \sum_{i=1}^{n}\left\|\widetilde{u}^\eps_{i}\Ared_{i}(\widetilde{u}^\eps)-\widetilde{u}_{i}\Ared_{i}(\widetilde{u})\right\|_{L^1(0,T;L^1(\dom))}\\
 &\leq \sum_{i=1}^{n}\left\|\widetilde{u}^\eps_{i}\right\|_{L^2(0,T;L^2(\dom))}\left\|\Ared_{i}(\widetilde{u}^\eps)-\Ared_{i}(\widetilde{u})\right\|_{L^2(0,T;L^2(\dom))}\\
 &\phantom{xx}+\sum_{i=1}^{n}\left\|\widetilde{u}^\eps_{i}-\widetilde{u}_{i}\right\|_{L^2(0,T;L^2(\dom))}\left\|\Ared_{i}(\widetilde{u})\right\|_{L^2(0,T;L^2(\dom))}.
\end{align*}
By the Cauchy--Schwarz inequality in $\omega$ and the moment bounds of Lemma~\ref{lem.hom}, the right-hand side has bounded moments of order $p/2>1$, uniformly in $\eps$, and it tends to zero $\widetilde\Prob$-a.s. by the strong convergence of $\widetilde u^\eps$. Vitali's convergence theorem now implies that
$$
  \lim_{\eps\to 0}\widetilde \E\int_0^T\|\sigma_{\delta}(\widetilde u^\eps(s))_{ij}
	-\sigma(\widetilde u(s))_{ij}\|_{{\mathcal L}_2(U;L^2(\dom))}^2\ds= 0.
$$
The estimate
\begin{align*}
  \widetilde \E&\bigg|\bigg(\int_0^T\sigma_{\delta}(\widetilde u^\eps(s))_{ij}
	\dd \widetilde W_j^\eps(s) - \int_0^T\sigma(\widetilde u(s))_{ij}
	\dd\widetilde W_j(s),\nabla\phi_2\bigg)_{L^2(\dom)}\bigg|^2 \\
	&\leq C\|\phi_2\|_{D(L)}^2\widetilde \E\int_0^T
	\big(\|\sigma_{\delta}(\widetilde u^\eps(s))_{ij}\|_{{\mathcal L}_2(U;L^2(\dom))}^2
	+ \|\sigma(\widetilde u(s))_{ij}\|_{{\mathcal L}_2(U;L^2(\dom))}^2\big)\ds\\
	&\leq C\|\phi_2\|_{D(L)}^2\bigg\{1 + \widetilde \E\bigg(\int_0^T\big(
	\|\widetilde u^\eps(s)\|_{L^2(\dom)}^2 + \|\widetilde u(s)\|_{L^2(\dom)}^2\big)
	\ds\bigg)\bigg\} \leq C
\end{align*}
for all $\phi_2\in D(L)$, the dominated convergence, $\widetilde W^\eps\to\widetilde W$ in
$C^0([0,T];U_0)$, the uniform convergence of $\sigma_{\delta}\rightarrow \sigma$ and \citet[Lemma 2.1]{DGT11} imply
$$
  \int_0^T\sigma_{\delta}(\widetilde u^\eps)_{ij}\dd\widetilde W^\eps
	\to\int_0^T\sigma(\widetilde u)_{ij}\dd\widetilde W\quad\mbox{in }
	L^2(0,T;L^2(\dom))\ \widetilde \Prob\mbox{-a.s.}
$$ 
We are left with the correction term. We first write $\langle\mathcal T(u)_i,\phi\rangle$ in closed form. Let $\vartheta_i(u):=\Ared_i(u)+a_{ii}u_i$. Since $\sigma(u)_{ii}=\sqrt{u_i\Ared_i(u)}$, we have $(\partial_{u_i}\sigma)(u)_{ii}\,\sigma(u)_{ii}=\frac12\vartheta_i(u)$ and $(\partial_{u_i}\sigma)(u)_{ii}\,\partial_{x_l}\sigma(u)_{ii}=\frac14\vartheta_i(u)\partial_{x_l}\log u_i+\frac{\vartheta_i(u)}{4\Ared_i(u)}\partial_{x_l}\Ared_i(u)$. Summing over $k$ in \eqref{eq:T_intro}, with the weak convention stated after \eqref{eq:T_intro}, gives, for $\phi\in D(L)$,
\begin{equation}\label{eqn:T_closed_form}
  \langle\mathcal T(u)_i,\phi\rangle=-\sum_{l=1}^d\int_\dom\Big[\frac14\vartheta_i(u)\,\partial_{x_l}\log u_i\,F^{il}_1+\frac{\vartheta_i(u)}{4\Ared_i(u)}\,\partial_{x_l}\Ared_i(u)\,F^{il}_1+\frac12\vartheta_i(u)\,F^{il}_2\Big]\partial_{x_l}\phi\dx .
\end{equation}
For the limit $\widetilde u$ we take \eqref{eqn:T_closed_form}, with $F^{il}_1\partial_{x_l}\log\widetilde u_i$ replaced by $\sqrt{F^{il}_1}\,\mathcal G_{il}$, as the definition of $\mathcal T(\widetilde u)$, as in Definition~\ref{def:martingale_solution_SKT}. The coefficients satisfy $F^{il}_1\le E_\infty$ and $|F^{il}_2|\le(F^{il}_1F^{il}_3)^{1/2}\le(E_\infty E_\infty')^{1/2}$.

Fix $t\in[0,T]$ and let $X_\eps(t):=\int_0^t\langle\mathcal T(\widetilde u^\eps(s))_i-\mathcal T(\widetilde u(s))_i,\phi_2\rangle\ds$. We show that $X_\eps(t)\to0$ $\widetilde\Prob$-a.s. Each of the three terms in \eqref{eqn:T_closed_form} is the integral over $Q_t$ of a product $f_\eps g_\eps\psi$, where $\psi:=F^{il}_1\partial_{x_l}\phi_2$ or $F^{il}_2\partial_{x_l}\phi_2$ lies in $L^\infty(Q_T)$, since $D(L)\hookrightarrow W^{1,\infty}(\dom)$. In the first term, $f_\eps=\frac14\vartheta_i(\widetilde u^\eps)\sqrt{F^{il}_1}$ converges strongly in $L^2(Q_T)$, since $\vartheta_i$ is affine and $\widetilde u^\eps\to\widetilde u$ strongly in $L^2(0,T;L^{s^*}(\dom))$, where we may take $s^*\ge2$ in \eqref{5.sstar}, and $g_\eps=\sqrt{F^{il}_1}\,\partial_{x_l}\log\widetilde u_i^\eps$ converges weakly in $L^2(Q_T)$ to $\mathcal G_{il}$, as shown after the Skorokhod--Jakubowski step. Here $\psi=\partial_{x_l}\phi_2$. In the second term, $f_\eps=\vartheta_i(\widetilde u^\eps)/\Ared_i(\widetilde u^\eps)\in[1,2]$, and $u\mapsto\vartheta_i(u)/\Ared_i(u)$ is globally Lipschitz on $[0,\infty)^n$ because $\Ared_i\ge a_{i0}>0$. Hence $f_\eps\psi\to f\psi$ strongly in $L^2(Q_T)$, and $g_\eps=\partial_{x_l}\Ared_i(\widetilde u^\eps)$ converges weakly in $L^2(Q_T)$. In the third term, $g_\eps=1$. In all three cases, a strongly convergent factor times a weakly convergent factor gives $\int_{Q_t}f_\eps g_\eps\psi\to\int_{Q_t}fg\psi$. Hence $X_\eps(t)\to0$ $\widetilde\Prob$-a.s.

It remains to pass to the limit under the expectation. By the Cauchy--Schwarz inequality,
\begin{align*}
  \sup_{t\le T}|X_\eps(t)|&\le C\|\phi_2\|_{D(L)}\Big(\|\vartheta_i(\widetilde u^\eps)\|_{L^2(Q_T)}\sum_l\big\|\sqrt{F^{il}_1}\,\partial_{x_l}\log\widetilde u_i^\eps\big\|_{L^2(Q_T)}\\
  &\phantom{xx}+\|\na\widetilde u^\eps\|_{L^2(Q_T)}+\|\vartheta_i(\widetilde u^\eps)\|_{L^1(Q_T)}\Big)\\
  &\phantom{xx}+(\text{the same terms with }\widetilde u\text{ in place of }\widetilde u^\eps),
\end{align*}
with $C$ depending on $E_\infty$, $E_\infty'$ and $(a_{ij})$. Let $q:=p/2>1$, with $p>2$ from Assumption~\ref{Assumption:A4_correction_factor}. By the Cauchy--Schwarz inequality in $\omega$,
\[
  \widetilde\E\big(\|\vartheta_i(\widetilde u^\eps)\|_{L^2}\|\sqrt{F^{il}_1}\partial_{x_l}\log\widetilde u_i^\eps\|_{L^2}\big)^q\le\big(\widetilde\E\|\vartheta_i(\widetilde u^\eps)\|_{L^2}^p\big)^{1/2}\big(\widetilde\E\|\sqrt{F^{il}_1}\partial_{x_l}\log\widetilde u_i^\eps\|_{L^2}^p\big)^{1/2}.
\] Both factors are bounded uniformly in $\eps$ by \eqref{5.H1p} and by \eqref{eqn:higher_oder_bounds_nabla_log_u}, since $\widetilde{\mathcal G}^\eps$ has the same law as $\mathcal G^\eps$. The terms with $\widetilde u$ are integrable by weak lower semicontinuity of the norms and Fatou's lemma. Thus $(\int_0^T|X_\eps(t)|\dt)_\eps$ is uniformly integrable, and Vitali's convergence theorem gives the last limit in the lemma.
\end{proof}

To show that the limit is indeed a solution, we define, for $t\in[0,T]$,
$i=1,\ldots,n$, and $\phi\in D(L)$,
\begin{align*}
  \Lambda_i^\eps(\widetilde u^\eps,\widetilde w^\eps,\widetilde W^\eps,\phi)(t)
	&:= \langle\widetilde u_i(0),\phi\rangle 
	+ \sqrt{\eps}\langle\widetilde w^\eps(0),\phi\rangle \\
	&\phantom{xx}{}- \sum_{j=1}^n\int_0^t\int_\dom A_{ij}(\widetilde u^\eps(s))
	\na\widetilde u_j^\eps(s)\cdot\na\phi\dx \ds\\
	&\phantom{xx}{}+ \sum_{j=1}^n\bigg(\int_0^t\sigma_{\delta}(\widetilde u^\eps(s))_{ij}\dd
	\widetilde W^\eps_j(s),\nabla\phi\bigg)_{L^2(\dom)} \\
    &\phantom{xx}+\frac\lambda N\int_{0}^{t}\langle\mathcal{T}(\widetilde{u}^\eps(s))_{i},\phi\rangle\ds,\\
	\Lambda_i(\widetilde u,\widetilde w,\widetilde W,\phi)(t) 
	&:= \langle\widetilde u_i(0),\phi\rangle
	- \sum_{j=1}^n\int_0^t\int_\dom A_{ij}(\widetilde u(s))\na\widetilde u_j(s)
	\cdot\nabla\phi\dx \ds\\
	&\phantom{xx}{}+ \sum_{j=1}^n\bigg(\int_0^t\sigma(\widetilde u(s))_{ij}\dd
	\widetilde W_j(s),\nabla\phi\bigg)_{L^2(\dom)}\\
    &\phantom{xx}+\frac\lambda N\int_{0}^{t}\langle\mathcal{T}(\widetilde{u}(s))_{i},\phi\rangle\ds.
\end{align*}
The following corollary is a consequence of Lemma \ref{lem.E}.

\begin{corollary}\label{coro.skt}
It holds for any $\phi_1\in L^2(\dom)$ and $\phi_2\in D(L)$ that
\begin{align*}
  \lim_{\eps\to 0}\big\|(\widetilde u_i^\eps,\phi_1)_{L^2(\dom)}
	- (\widetilde u_i,\phi_1)_{L^2(\dom)}\big\|_{L^1(\widetilde\Omega\times(0,T))} 
	&= 0, \\
	\lim_{\eps\to 0}\|\Lambda_i^\eps(\widetilde u^\eps,\sqrt{\eps}\widetilde w^\eps,
	\widetilde W^\eps,\phi_2) - \Lambda_i(\widetilde u,0,\widetilde W,\phi_2)
	\|_{L^1(\widetilde \Omega\times(0,T))} &= 0.
\end{align*}
\end{corollary}

Since $v^\eps$ is a strong solution to \eqref{eqn:approximate_equation_strat}, it satisfies
for a.e.\ $t\in[0,T]$ $\Prob$-a.s., $i=1,\ldots,n$, and $\phi\in D(L)$,
$$
  (v_i^\eps(t),\phi)_{L^2(\dom)}
	= \Lambda_i^\eps(u^\eps,\eps L^*LR_\eps(v^\eps),W,\phi)(t)
$$
and in particular,
$$
  \int_0^T\E\big|(v_i^\eps(t),\phi)_{L^2(\dom)} 
	- \Lambda_i^\eps(u^\eps,\eps L^*LR_\eps(v^\eps),W,\phi)(t)\big|\dt= 0.
$$
We deduce from the equivalence of the laws of $(u^\eps,\eps L^*LR_\eps(v^\eps),W)$
and $(\widetilde u^\eps,\sqrt{\eps}\widetilde w^\eps,\widetilde W)$ that
$$
  \int_0^T\widetilde \E\big|\big(\widetilde u^\eps_i(t)
	+\sqrt{\eps}\widetilde w_i^\eps,\phi\big)_{L^2(\dom)} 
	- \Lambda_i^\eps\big(\widetilde u^\eps,\sqrt{\eps}\widetilde w^\eps,
	\widetilde W^\eps,\phi\big)(t)\big|\dt= 0.
$$
By Corollary \ref{coro.skt}, we can pass to the limit $\eps\to 0$ to obtain
$$
  \int_0^T\widetilde \E\big|(\widetilde u_i(t),\phi)_{L^2(\dom)} 
	- \Lambda_i(\widetilde u,0,\widetilde W,\phi)(t)\big|\dt= 0.
$$
This identity holds for all $i=1,\ldots,n$ and all $\phi\in D(L)$. This shows that 
$$
  \big|(\widetilde u_i(t),\phi)_{L^2(\dom)} 
	- \Lambda_i(\widetilde u,0,\widetilde W,\phi)(t)\big| = 0
	\quad\mbox{for a.e. }t\in[0,T]\ \widetilde\Prob\mbox{-a.s.},\ i=1,\ldots,n.
$$
We infer from the definition of $\Lambda_i$ that
\begin{align*}
  (\widetilde u_i(t),\phi)_{L^2(\dom)} &= (\widetilde u_i(0),\phi)_{L^2(\dom)} 
	- \sum_{j=1}^n\int_0^t \int_\dom A_{ij}(\widetilde u(s))\na\widetilde u_j(s)
	\cdot\na\phi\dx \ds\\
	&\phantom{xx}{}+ \sum_{j=1}^n\bigg(\int_0^t\sigma
	(\widetilde u(s))_{ij}\dd\widetilde W_j(s),\nabla \phi\bigg)_{L^2(\dom)} \\
    &\phantom{xx}+\frac\lambda N\int_{0}^{t}\langle\mathcal{T}(\widetilde{u}(s))_{i},\phi\rangle\ds.
\end{align*}
for a.e.\ $t\in[0,T]$ and all $\phi\in D(L)$. 
Set $\widetilde U=(\widetilde\Omega,\widetilde{\mathcal F},\widetilde\Prob,
\widetilde\F)$. Then $(\widetilde U,\widetilde W,\widetilde u)$ is a martingale
solution to \eqref{eq:SKT_intro}--\eqref{eq:SKT_intro_matrix_A}.


\section{The role of the restriction \texorpdfstring{$\lambda>\frac12$}{lambda>1/2}}\label{sec.lambda}

In this section we explain where the restriction $\lambda>\frac12$ enters the proof. We show that, in general, the entropy estimate of Proposition~\ref{prop.ent} cannot be closed for $\lambda\le\frac12$ with the noise coefficient \eqref{eq:SKT_intro_diffusion_sigma}, and we show how close to the Stratonovich value $\lambda=\frac12$ Theorem~\ref{thm:existence_SKT} reaches. The computations concern the structure of the estimate, not the approximation scheme, so we carry them out for strictly positive Lipschitz functions and without the regularizations $\eps$ and $\delta$. We use the notation $F^{il}_1,F^{il}_2,F^{il}_3$ of Assumption~\ref{Assumption:A5_noise_onb}, and we assume $a_{i0}>0$ for all $i$, as in Theorem~\ref{thm:existence_SKT}.

\subsection{The noise-induced entropy production}
For $u\in W^{1,\infty}(\dom;\R^n)$ with $\operatorname*{ess\,inf}_\dom u_i>0$ for all $i$, we set, for $i=1,\dots,n$ and $l=1,\dots,d$,
\begin{equation}\label{6.XY}
  X_{il}:=\partial_{x_l}u_i,\qquad Y_{il}:=\sum_{j\neq i}a_{ij}\,\partial_{x_l}u_j,\qquad
  \vartheta_i:=\Ared_i(u)+a_{ii}u_i.
\end{equation}
Then $\partial_{x_l}\Ared_i(u)=a_{ii}X_{il}+Y_{il}$ and $\partial_{x_l}\big(u_i\Ared_i(u)\big)=\vartheta_iX_{il}+u_iY_{il}$. The quantity $Y_{il}$ is the part of $\partial_{x_l}\sigma(u)_{ii}^2$ that comes from the dependence of $\sigma(u)_{ii}=\sqrt{u_i\Ared_i(u)}$ on the other species. We define the noise-induced entropy production
\begin{equation}\label{6.P}
  \mathfrak P_\lambda(u):=\sum_{i,k,l=1}^{n,\infty,d}\pi_i\int_\dom\bigg[-\lambda\,(\partial_{u_i}\sigma)(u)_{ii}\,e^{il}_k\,\partial_{x_l}\big(\sigma(u)_{ii}e^{il}_k\big)\frac{\partial_{x_l}u_i}{u_i}
  +\frac{\big(\partial_{x_l}(\sigma(u)_{ii}e^{il}_k)\big)^2}{2u_i}\bigg]\dx
\end{equation}
and the entropy dissipation
\begin{equation}\label{6.D}
  \mathfrak D(u):=\int_\dom\nabla u:h''(u)A(u)\nabla u\dx
  =\sum_{i,l=1}^{n,d}\pi_i\int_\dom\frac{X_{il}\,\partial_{x_l}\big(u_i\Ared_i(u)\big)}{u_i}\dx .
\end{equation}
The first term in \eqref{6.P} is $\lambda\langle\mathcal T(u),h'(u)\rangle$, written in weak form as in Definition~\ref{def:martingale_solution_SKT}. The second term is the It\^o correction $\frac12\sum_{k,l}\int_\dom h''(u)[\partial_{x_l}(\sigma(u)e^{\cdot l}_k)]^2\dx$ of the noise. Hence $\frac1N\mathfrak P_\lambda(u)$ is the rate at which noise and correction term change $\int_\dom h(u)\dx$, and $\mathfrak D(u)$ is the rate at which the deterministic part dissipates it. In the proof of Proposition~\ref{prop.ent}, $N(I+II)$ is the time integral of $\mathfrak P_\lambda(u^\eps)$, with $\sigma_\delta$ in place of $\sigma$ in the It\^o correction. By \eqref{5.RBR}, $\mathfrak D(u)$ is bounded below by the dissipation terms in \eqref{eqn:ito_no_sup}.

\begin{lemma}[Exact form of the entropy production]\label{lem:P_identity}
Let Assumption~\ref{Assumption:A5_noise_onb} hold and let $u$ be as above. Then, with $X=X_{il}$, $Y=Y_{il}$,
\begin{align}
  \mathfrak P_\lambda(u)=\sum_{i,l=1}^{n,d}\pi_i\int_\dom\bigg[&F^{il}_1\,\frac{(1-2\lambda)\vartheta_i^2X^2+2(1-\lambda)\vartheta_iu_iXY+u_i^2Y^2}{8u_i^2\Ared_i(u)}\nonumber\\
  &+F^{il}_2\,\frac{(1-\lambda)\vartheta_iX+u_iY}{2u_i}+\frac12\Ared_i(u)F^{il}_3\bigg]\dx. \label{6.identity}
\end{align}
\end{lemma}

\begin{proof}
Fix $i$, $k$, $l$ and write $f:=\sigma(u)_{ii}=\sqrt{u_i\Ared_i(u)}$, $e:=e^{il}_k$, $s:=\partial_{x_l}e^{il}_k$ and $Z:=\vartheta_iX+u_iY$. Since $f^2=u_i\Ared_i(u)$, we have $\partial_{u_i}f=\vartheta_i/(2f)$, $\partial_{x_l}f=Z/(2f)$, and therefore $\partial_{x_l}(fe)=Ze/(2f)+fs$. Inserting this into the integrand of \eqref{6.P} and using $f^2=u_i\Ared_i(u)$ gives
\[
  -\lambda\frac{\vartheta_i}{2f}\,e\Big(\frac{Ze}{2f}+fs\Big)\frac{X}{u_i}+\frac1{2u_i}\Big(\frac{Ze}{2f}+fs\Big)^2
  =e^2\,\frac{Z^2-2\lambda\vartheta_iXZ}{8u_i^2\Ared_i(u)}+es\,\frac{Z-\lambda\vartheta_iX}{2u_i}+\frac{\Ared_i(u)}{2}s^2 .
\]
Expanding gives $Z^2-2\lambda\vartheta_iXZ=(1-2\lambda)\vartheta_i^2X^2+2(1-\lambda)\vartheta_iu_iXY+u_i^2Y^2$ and $Z-\lambda\vartheta_iX=(1-\lambda)\vartheta_iX+u_iY$. Finally, we sum over $k$. The series converge absolutely by Assumption~\ref{Assumption:A5_noise_onb}, and $\sum_ke^2=F^{il}_1$, $\sum_kes=F^{il}_2$, $\sum_ks^2=F^{il}_3$.
\end{proof}

\begin{remark}\label{rem:P_reading}
Identity \eqref{6.identity} isolates the dependence on $\lambda$. Since $\Ared_i(u)\ge a_{i0}$ and $\vartheta_i\ge\Ared_i(u)$, the only terms that are singular as $u_i\to0$ are those carrying $X/u_i$.
\begin{enumerate}[label=(\alph*)]
\item\label{rem:P_reading_a} \emph{$\lambda>\frac12$.} The term $(1-2\lambda)F^{il}_1\vartheta_i^2X^2/(8u_i^2\Ared_i(u))$ is negative and bounded above by $-\frac{(2\lambda-1)a_{i0}}{8}F^{il}_1\,|\partial_{x_l}\log u_i|^2$. Young's inequality absorbs into it the mixed terms in $XY$ and in $F^{il}_2X$, with constants of order $(2\lambda-1)^{-1}$. This is the mechanism behind Lemma~\ref{lem:is_correction_entropy_estimate} and the constant $K_j$ in \ref{Assumption:A4_correction_factor}. What remains of the negative term is the good term $\mathrm{IS}_1'$ of Proposition~\ref{prop.ent}. It yields the weighted bound on $\sqrt{F^{il}_1}\,\partial_{x_l}\log u_i^\eps$ in Corollary~\ref{cor:nabla_log_u_estimate}. This bound is what gives a meaning to $\mathcal T(\widetilde u)$ in Definition~\ref{def:martingale_solution_SKT}.
\item\label{rem:P_reading_b} \emph{$\lambda=\frac12$ (Stratonovich).} The term in $X^2$ cancels exactly, and \eqref{6.identity} becomes
\begin{equation}\label{6.strat}
  \mathfrak P_{1/2}(u)=\sum_{i,l=1}^{n,d}\pi_i\int_\dom\bigg[F^{il}_1\,\frac{\vartheta_iXY+u_iY^2}{8u_i\Ared_i(u)}+F^{il}_2\Big(\frac{\vartheta_iX}{4u_i}+\frac Y2\Big)+\frac12\Ared_i(u)F^{il}_3\bigg]\dx .
\end{equation}
Two singular terms remain. The first, $F^{il}_1\vartheta_iXY/(8u_i\Ared_i(u))$, is present only because $\sigma(u)_{ii}$ depends on $u_j$, $j\neq i$. The second, $F^{il}_2\vartheta_iX/(4u_i)$, is present only if the noise is not stationary in the direction $x_l$, i.e.\ if $F^{il}_1$ depends on $x_l$. If both vanish ($Y\equiv0$ and $F^{il}_2\equiv0$), then $\mathfrak P_{1/2}(u)=\frac12\sum_{i,l}\pi_i\int_\dom\Ared_i(u)F^{il}_3\dx$. This is bounded by $C(1+\int_\dom h(u)\dx)$, and the entropy estimate closes without any good term.
\item \emph{$\lambda<\frac12$, including the It\^o equation $\lambda=0$.} The term in $X^2$ is positive and of order $|\partial_{x_l}\log u_i|^2$.
\end{enumerate}
\end{remark}

\begin{remark}[Comparison with the Dean--Kawasaki equation]\label{rem:DK_comparison}
For $n=1$, $\pi_1=1$, $a_{10}=1$ and $a_{11}=0$, the noise coefficient is $\sigma(u)=\sqrt u$ and $Y\equiv0$. Then \eqref{6.strat} reads $\mathfrak P_{1/2}(u)=\sum_l\int_\dom\big[\frac18\partial_{x_l}F^{1l}_1\,\partial_{x_l}\log u+\frac12F^{1l}_3\big]\dx$. This is the entropy production of the Stratonovich Dean--Kawasaki equation with correlated noise, as obtained from its It\^o form in \citet[(1.7), (3.1)]{fehrman2024well}. There the equation is posed on the torus, and the noise is assumed stationary, $\nabla\cdot F_2=\frac12\Delta F_1=0$ \citep[Remark~2.2]{fehrman2024well}. Then the first term vanishes after an integration by parts, consistent with \ref{rem:P_reading_b}. For non-stationary noise, the singular term $\frac18\partial_{x_l}F^{1l}_1\,\partial_{x_l}\log u$ in the entropy production persists. It is handled in \citet{fehrman2025nonstationary} by a regularization of the logarithm of the solution. The SKT noise \eqref{eq:SKT_intro_diffusion_sigma} produces, in addition, the cross term $F^{il}_1\vartheta_iXY/(8u_i\Ared_i(u))$, which has no counterpart for Dean--Kawasaki. Proposition~\ref{prop:lambda_obstruction} shows that this term cannot be bounded by the entropy and its dissipation, even for noise that is stationary on a subdomain. Exactly stationary noise, $F^{il}_2\equiv0$ on all of $\dom$, is not compatible with the no-flux condition \eqref{eq:noise_bc} on a bounded domain, unless $F^{il}_1\equiv0$, since $F^{il}_1$ must vanish where $\nu_l\neq0$. It is natural on the torus, as in \citet{fehrman2024well}. Noise that is stationary on a subdomain is admissible. For instance, on $\dom=(0,\ell)^d$ let $\chi\in C^\infty([0,\ell])$ satisfy $\chi(0)=\chi(\ell)=0$, $\chi>0$ on $(0,\ell)$ and $\chi=1$ on $[\frac\ell8,\frac{7\ell}8]$. In the $(i,l)$-component take the family $\chi(x_l)\,t_m(x)$, where $(t_m)$ is the real trigonometric system $\{w_m^{1/2}\cos(2\pi m\cdot x/\ell),\,w_m^{1/2}\sin(2\pi m\cdot x/\ell)\}$, with $m$ ranging over a half-lattice of $\Z^d$ (for $m=0$ only the constant function) and $w_m=c(1+|m|^2)^{-s'}$, $s'>s+\frac d2$. Then $F^{il}_1=\chi(x_l)^2\sum_mw_m$ vanishes on $\{x_l\in\{0,\ell\}\}$, and $F^{il}_1$ is constant and $F^{il}_2=0$ on $\{\frac\ell8\le x_l\le\frac{7\ell}8\}$. The map $c\mapsto\sum_mc_m\chi t_m$ is injective on $\ell^2$, since $\chi>0$ on $(0,\ell)$, and $\sum_m\|\chi t_m\|_{W^{s,2}(\dom)}^2\le C\sum_mw_m(1+|m|^2)^s<\infty$ for $s'>s+\frac d2$. Hence the Hilbert space $U$ in which the family is orthonormal embeds continuously into $W^{s,2}(\dom)^{n\times d}$, and Assumption~\ref{Assumption:A5_noise_onb} holds.
\end{remark}

\subsection{The obstruction at \texorpdfstring{$\lambda\le\frac12$}{lambda<=1/2}}
The proof of Proposition~\ref{prop.ent} bounds the entropy production pointwise in time by the dissipation and the entropy. Concretely, it uses an inequality of the form $\frac1N\mathfrak P_\lambda(u)\le\theta_0\,\mathfrak D(u)+C(1+\int_\dom h(u)\dx)$ with $\theta_0<1$. The terms coming from the BDG inequality are treated separately. The next proposition shows that no inequality of this form holds for $\lambda<\frac12$, nor for $\lambda=\frac12$ when $\sigma_1(u)$ depends on $u_2$, whatever the values of $N$, $\theta_0$ and $C$.

\begin{proposition}[No entropy bound for $\lambda\le\frac12$]\label{prop:lambda_obstruction}
Let $\dom=(0,1)^d$, $n\ge2$, let Assumptions~\ref{Assumption:A3_SKT_matrix} and \ref{Assumption:A5_noise_onb} hold with $a_{i0}>0$ for all $i$, and suppose that $F^{11}_2\equiv0$ and $F^{11}_1\ge\underline F_S$ on the slab $S:=\{x\in\dom:\frac14<x_1<\frac34\}$, for a constant $\underline F_S>0$. Then there is a family $(u^r)_{r\in(0,1)}\subset W^{1,\infty}(\dom;(0,\infty)^n)$ with $\inf_\dom u^r_i>0$ such that
\[
  \sup_{r\in(0,1)}\Big(\mathfrak D(u^r)+\int_\dom h(u^r)\dx\Big)<\infty
\]
and the following holds.
\begin{enumerate}[label=(\roman*)]
\item If $a_{12}>0$, then $\mathfrak P_{1/2}(u^r)\ge\frac{\pi_1a_{12}\underline F_S}{8}\log\big(1+\frac1{16\sqrt r}\big)-C$, with $C$ independent of $r$.
\item For every $\lambda<\frac12$, whether or not $a_{12}>0$, $\mathfrak P_\lambda(u^r)\ge\frac{(1-2\lambda)\pi_1a_{10}\underline F_S}{2}\big(\frac1{\sqrt r}-\frac1{\sqrt r+1/16}\big)-C$, with $C$ independent of $r$.
\end{enumerate}
In particular, $\mathfrak P_\lambda(u^r)\to\infty$ as $r\to0$, for $\lambda<\frac12$ and, if $a_{12}>0$, also for $\lambda=\frac12$.
\end{proposition}

\begin{proof}
The functions depend only on $x_1$. Let $\rho(t):=\min\{\operatorname{dist}(t,[\frac38,\frac58]),\frac18\}$ for $t\in[0,1]$, and set
\[
  u^r_1(x):=\big(\sqrt r+\rho(x_1)/2\big)^2,\qquad u^r_2(x):=1+\big(x_1-\tfrac12\big)^2,\qquad u^r_j(x):=1\quad(j\ge3).
\]
These functions are Lipschitz, $u^r_1\ge r$, $u^r_1\le(17/16)^2$ and $u^r_j\ge1$ for $j\ge2$. Since $|\rho'|\le1$, $X_{11}=(\sqrt r+\rho/2)\rho'$ satisfies $X_{11}^2\le u_1^r$. All other derivatives $X_{il}$, $Y_{il}$ are bounded uniformly in $r$, and $X_{il}=Y_{il}=0$ for $l\ge2$.

\emph{Bounded entropy and dissipation.} By \eqref{6.D}, $\mathfrak D(u)=\sum_{i,l}\pi_i\int_\dom\big(\Ared_i(u)X_{il}^2/u_i+X_{il}\,\partial_{x_l}\Ared_i(u)\big)\dx$. Here $X_{11}^2/u^r_1\le1$, the $u_j^r$ with $j\ge2$ are bounded below by $1$, and $\Ared_i(u^r)$ and $\partial_{x_l}\Ared_i(u^r)$ are bounded uniformly in $r$. Hence $\mathfrak D(u^r)$ is bounded uniformly in $r$, and so is $\int_\dom h(u^r)\dx$.

\emph{Bounded terms.} Consider the integrand of \eqref{6.identity}.
\begin{itemize}
\item For $i\ge2$ every term is bounded uniformly in $r$, since $u_i^r\ge1$, $\Ared_i\ge a_{i0}>0$, $F^{il}_1\le E_\infty$ and $|F^{il}_2|\le(F^{il}_1F^{il}_3)^{1/2}$.
\item For $i=1$ and $l\ge2$, only $\frac12\Ared_1F^{1l}_3\ge0$ remains.
\item For $i=1$, $l=1$, the terms in $F^{11}_2$ vanish on $S$ by assumption. Outside $S$ we have $X=0$, and the remaining term $\frac12F^{11}_2Y$ is bounded. Moreover, $\frac{u_1^2Y^2}{8u_1^2\Ared_1}\ge0$ and $\frac12\Ared_1F^{11}_3\ge0$.
\end{itemize}

\emph{The singular terms.} It remains to bound $\pi_1\int_\dom F^{11}_1\big[(1-2\lambda)\vartheta_1^2X^2+2(1-\lambda)\vartheta_1u_1XY\big]/(8u_1^2\Ared_1)\dx$ from below, with $X=X_{11}$ and $Y=Y_{11}=2a_{12}(x_1-\frac12)$. On $\{\frac14<x_1<\frac38\}$ we have $\rho=\frac38-x_1\in(0,\frac18)$, $X=-(\sqrt r+\rho/2)$ and $Y=-2a_{12}(\frac18+\rho)$. On $\{\frac58<x_1<\frac34\}$ both signs are reversed. Elsewhere $X=0$. Hence $X\neq0$ only in $S$, $XY\ge\frac{a_{12}}4(\sqrt r+\rho/2)\ge0$ wherever $X\ne0$, and $\vartheta_1\ge\Ared_1$.
\begin{itemize}
\item For $\lambda=\frac12$, the $X^2$-term vanishes. Using $\vartheta_1/\Ared_1\ge1$ and $F^{11}_1\ge\underline F_S$ on $S$,
\[
  \pi_1\int_\dom F^{11}_1\frac{\vartheta_1XY}{8u_1\Ared_1}\dx\ge\frac{\pi_1a_{12}\underline F_S}{32}\cdot2\int_0^{1/8}\frac{\dd\rho}{\sqrt r+\rho/2}=\frac{\pi_1a_{12}\underline F_S}{8}\log\Big(1+\frac1{16\sqrt r}\Big).
\]
\item For $\lambda<\frac12$, both coefficients $(1-2\lambda)$ and $2(1-\lambda)$ are positive. Using $\vartheta_1^2/\Ared_1\ge a_{10}$ and $X^2=u_1$ on $\{X\ne0\}$, the term is at least
\[
  \frac{(1-2\lambda)\pi_1a_{10}\underline F_S}{8}\cdot 2\int_0^{1/8}\frac{\dd\rho}{(\sqrt r+\rho/2)^2}=\frac{(1-2\lambda)\pi_1a_{10}\underline F_S}{2}\Big(\frac1{\sqrt r}-\frac1{\sqrt r+1/16}\Big).
\]
\end{itemize}
\end{proof}

Proposition~\ref{prop:lambda_obstruction} concerns the method, not the equation: it does not show that \eqref{eq:SKT_intro} is ill-posed for $\lambda\le\frac12$. It does show that the restriction $\lambda>\frac12$ cannot in general be removed from the entropy estimate of this paper. Reaching $\lambda=\frac12$ would require an estimate that is not pointwise in time, or a different Lyapunov functional. In the scalar, stationary case a substitute is provided by the kinetic approach of \citet{fehrman2024well}.

\subsection{How close to the Stratonovich equation the result reaches}
The constants in \ref{Assumption:A4_correction_factor} blow up as $\lambda\downarrow\frac12$, but only like $(2\lambda-1)^{-1}$, and the restriction is multiplied by $\frac1N$. The excess $\lambda-\frac12$ can therefore be taken of order $\frac1N$.

\begin{corollary}[Stratonovich up to a correction with coefficient $O(N^{-2})$]\label{cor:lambda_N}
Let Assumptions~\ref{Assumption:A1_domain}--\ref{Assumption:A3_SKT_matrix} and \ref{Assumption:A5_noise_onb} hold with $a_{i0},a_{ii}>0$ for all $i$, and fix $\tau\in(0,\frac12)$ and $p>2$. Let
\[
  c_\star\ge\max_{1\le j\le n}\frac{21\,\beta_jE_\infty}{8\tau\pi_ja_{jj}}
\]
with $\beta_j$ as in \ref{Assumption:A4_correction_factor}, and set $\lambda_N:=\frac12+\frac{c_\star}N$. Then there is $N_0\in\N$ such that for all $N\ge N_0$ Assumption~\ref{Assumption:A4_correction_factor} holds with $\lambda=\lambda_N$, $\kappa=\kappa_N:=\frac{2\lambda_N-1}{2(4\lambda_N+3)}$ and $\widetilde\kappa_3=\frac12$. Consequently, by Theorem~\ref{thm:existence_SKT}, \eqref{eq:SKT_intro}--\eqref{eq:SKT_intro_diffusion_sigma} with $\lambda=\lambda_N$ has a global nonnegative martingale solution. Its correction term is
\[
  \frac{\lambda_N}{N}\,\mathcal T(u)=\frac1{2N}\,\mathcal T(u)+\frac{c_\star}{N^2}\,\mathcal T(u),
\]
i.e.\ the equation is the Stratonovich equation plus the correction term with the coefficient $c_\star/N^2$.
\end{corollary}

\begin{proof}
We check the conditions in \ref{Assumption:A4_correction_factor} one by one.
\begin{itemize}
\item \emph{Range of the parameters.} Let $N\ge2c_\star$, so that $\frac12<\lambda_N\le1$. Then $0<\kappa_N<\frac{2\lambda_N-1}{4\lambda_N+3}$, and $C_{\lambda_N,\kappa_N}=\frac{2\lambda_N-1}8-\kappa_N\frac{4\lambda_N+3}8=\frac{2\lambda_N-1}{16}=\frac{c_\star}{8N}>0$.
\item \emph{The first three inequalities.} They do not involve $\lambda$ and hold for all large $N$, since their left-hand sides are of order $N^{-1/2}$.
\item \emph{The fourth inequality.} Since $2\lambda_N-1=2c_\star/N$,
\[
  \frac{2\lambda_N+1}{8\kappa_N}=\frac{(2\lambda_N+1)(4\lambda_N+3)}{4(2\lambda_N-1)}=\frac{(2\lambda_N+1)(4\lambda_N+3)\,N}{8c_\star}\le\frac{21N}{8c_\star}.
\]
With $\lambda_N\kappa_N\le\frac12$ and $\widetilde\kappa_3=\frac12$, this gives $K_j(\lambda_N,\kappa_N,\frac12)\le\frac{21N}{8c_\star}\beta_j+\frac18(\beta_j+\gamma_j+\pi_ja_{jj})$, and hence
\[
  \frac1N K_j(\lambda_N,\kappa_N,\tfrac12)E_\infty\le\frac{21\beta_jE_\infty}{8c_\star}+\frac{(\beta_j+\gamma_j+\pi_ja_{jj})E_\infty}{8N}.
\]
By the choice of $c_\star$, the first term is at most $\tau\pi_ja_{jj}$. The second is at most $\tau\pi_ja_{jj}$ for $N\ge\max_j(\beta_j+\gamma_j+\pi_ja_{jj})E_\infty/(8\tau\pi_ja_{jj})$. Together they give the fourth inequality.
\end{itemize}
Theorem~\ref{thm:existence_SKT} is stated for each fixed $N$ and $\lambda>\frac12$, so it applies with $\lambda=\lambda_N$.
\end{proof}

All constants in the estimates of Section~\ref{sec.uniform} may depend on $N$. Corollary~\ref{cor:lambda_N} therefore makes no claim about uniformity as $N\to\infty$. In particular, the weighted bound on $\sqrt{F^{il}_1}\,\partial_{x_l}\log u_i$ obtained from the good term $\mathrm{IS}_1'$ deteriorates like $N^2/c_\star$, consistent with Proposition~\ref{prop:lambda_obstruction}.

\begin{remark}[Which $\lambda$ does the particle system select?]\label{rem:lambda_modelling}
In forthcoming work of the author it will be shown, under the hypotheses stated there, that the fluctuation field $\sqrt N(\mu^N-u_\eta)$ around the mollified mean-field limit $u_\eta$, in a regime $\eta=\eta(N)\to0$, of the moderately interacting particle system of \citet{chen_holzinger_2021_rigorous_SKT} converges in law to the solution of a linear Ornstein--Uhlenbeck SPDE. Its drift is the linearization of the SKT operator at the mean-field limit $u=\lim_{\eta\to0}u_\eta$, and its noise is $\diver\big(\sqrt{2u_i\Theta_i}\,\dd W_i\big)$. There $\Theta_i=\varsigma_i+\sum_jf(a_{ij}u_j)$ with the individual diffusion constants $\varsigma_i$ (denoted $\sigma_i$ in Remark~\ref{rem:particle_motivation}) and an interaction function $f$. For the linear interaction $f(x)=x$ and $\varsigma_i=a_{i0}$ this is $\Theta_i=\Ared_i(u)$, so the noise coefficient is $\sqrt2\,\sigma(u)_{ii}$ with $\sigma$ as in \eqref{eq:SKT_intro_diffusion_sigma}. The factor $\sqrt2$ can be absorbed into $N$ or into the $e^{il}_k$. That result concerns $\R^d$ and space-time white noise, so we compare formally. Let $u^N$ solve \eqref{eq:SKT_intro} and set $V^N:=\sqrt N(u^N-u)$. Then
\[
  \dd V^N=\sqrt N\big[F(u^N)-F(u)\big]\dt+\diver\big(\sigma(u^N)\dd W\big)+\frac{\lambda}{\sqrt N}\,\mathcal T(u^N)\dt,
\]
where $F$ denotes the deterministic SKT operator. Assume that $\mathcal T(u^N)$ stays bounded uniformly in $N$, in a suitable negative Sobolev norm. Then, as $N\to\infty$, the correction term vanishes and the first two terms formally converge to the linearized operator applied to $V$ and to $\diver(\sigma(u)\dd W)$. Hence, formally, every $\lambda$ gives the same Gaussian fluctuations at order $N^{-1/2}$. We claim this $\lambda$-independence, not a rigorous identification with the particle system. The particle system does not single out a value of $\lambda$ at this order. The parameter $\lambda$ affects \eqref{eq:SKT_intro} only at order $N^{-1}$, the order of the It\^o--Stratonovich correction itself, which the Gaussian fluctuation theory does not resolve. Compare the weak-error analysis of a nonlinear Dean--Kawasaki approximation in \citet{perkowski_2024weak}. Within this freedom, the Stratonovich value $\lambda=\frac12$ is one natural choice, because it preserves the chain rule. The numerical experiments of Section~\ref{sec.numerics} indicate that the order-$N^{-1}$ correction selected by the particle system depends on the ratio of the noise scale to the interaction range, and that $\lambda>0$ is needed when the noise is correlated over a longer range than the interaction. Corollary~\ref{cor:lambda_N} shows that our existence result reaches equations that differ from the Stratonovich equation only by a correction term with the coefficient $c_\star/N^2$.
\end{remark}

\section{\texorpdfstring{A numerical illustration}{A numerical illustration}}\label{sec.numerics}

We illustrate Section~\ref{sec.lambda} numerically in the simplest setting. The computations are heuristic and are meant to address two questions: how much the solution of \eqref{eq:SKT_intro} depends on $\lambda$, and which $\lambda$ the particle system of \citet{chen_holzinger_2021_rigorous_SKT} selects at order $N^{-1}$.

For simplicity we work on the torus $[0,1)$, with $n=2$, $a_{10}=a_{20}=\frac12$, $a_{11}=a_{22}=1$, $a_{12}=\frac12$, $a_{21}=\frac14$ and $\pi=(1,2)$, which satisfies detailed balance. The initial data are $u_{1,2}(0)=1\pm0.3\cos2\pi x$, the final time is $T=0.01$, and the observable is $m:=\langle u_1(T),\cos2\pi x\rangle$. We discretize \eqref{eq:SKT_intro} by finite volumes with cells of width $h=1/16$ and use the Euler--Maruyama scheme in time. The noise is white on the dual cells, i.e.\ $e_k=\sqrt{2/h}\,\mathbf 1_{\text{dual cell }k}$, so that $F_1=2/h$ and $F_2=0$. Then the noise has the same quadratic variation, $\frac2N\int_0^tu_i\Ared_i(u)|\varphi'|^2\ds$, as the fluctuation martingale of the particle system. The correction $\frac\lambda N\mathcal T$ is implemented in the closed form \eqref{eqn:T_closed_form}. The particle system is
\[
  \dd X^k_i=\sqrt{2\Theta_i(X^k_i)}\,\dd W^k_i,\qquad \Theta_i=a_{i0}+\sum_{j=1}^na_{ij}\,B^\eta*\mu^N_j,\qquad k=1,\dots,N,
\]
with a Gaussian mollifier $B^\eta$ and i.i.d.\ initial positions. We use the binned initial particle configurations as initial data for the discretized SPDE, so that both start with the same fluctuations. All values of $\lambda$ are driven by the same noise. We report the bias $\E m-m_{\rm det}$, where $m_{\rm det}$ is the deterministic ($N=\infty$) limit of the respective model, computed with the same discretization. The expectations are estimated from $2\cdot10^4$ to $4\cdot10^5$ independent runs, using $m(0)$ as a control variate.

\emph{Dependence on $\lambda$.} Figure~\ref{fig:lambda} shows that $\lambda$ has a small and predictable effect.
\begin{itemize}
\item \emph{Mean.} The mean is shifted by $N(\E m_\lambda-\E m_0)=-0.664\,\lambda$ for all $N\in\{250,\dots,4000\}$. This agrees to three digits with the shift caused by the deterministic drift $\frac\lambda N\mathcal T$ alone.
\item \emph{Variance.} The variance does not depend on $\lambda$: $N\,\mathrm{Var}(m)$ ranges from $0.351$ to $0.335$ for $\lambda\in[0,1]$ at $N=250$, and stays within $1.5\%$ of $0.35$ for $N\ge1000$. This is consistent with Remark~\ref{rem:lambda_modelling}.
\item \emph{Pathwise.} Under common noise, the difference $\|u^\lambda-u^0\|_{L^2}$ is below $1.3\%$ of the size of the fluctuations $\|u^0-\E u^0\|_{L^2}$ for $N\ge1000$. This ratio decreases at least like $N^{-1/2}$.
\end{itemize}
At $N=250$, about $15$ particles per cell, the number of time steps with a negative cell value decreases monotonically in $\lambda$: $275$, $206$, $159$, $90$ for $\lambda=0,\frac14,\frac12,1$, out of about $10^9$ cell updates. This is in line with the dissipative effect of $\mathcal T$ near vacuum that the existence proof exploits.

\begin{figure}[t]
\centering
\includegraphics[width=0.98\textwidth]{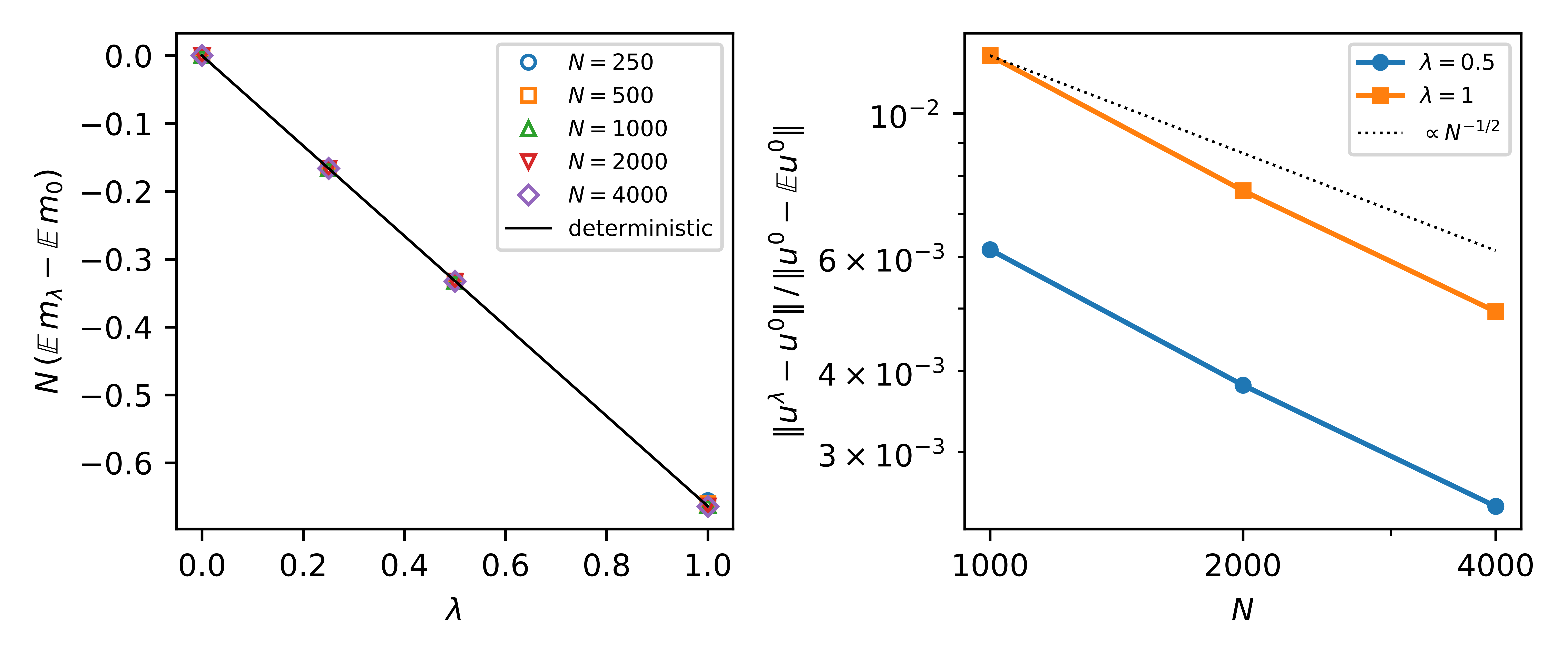}
\caption{Effect of $\lambda$ on the discretized SPDE ($h=1/16$, common noise). Left: shift of the mean, $N(\E m_\lambda-\E m_0)$, for $N=250,\dots,4000$, together with the shift caused by the deterministic drift $\frac\lambda N\mathcal T$ alone. Right: root-mean-square pathwise difference $\|u^\lambda_1-u^0_1\|_{L^2}$ at time $T$, relative to the fluctuations $\|u_1^0-\E u_1^0\|_{L^2}$.}\label{fig:lambda}
\end{figure}

\begin{figure}[t]
\centering
\includegraphics[width=0.75\textwidth]{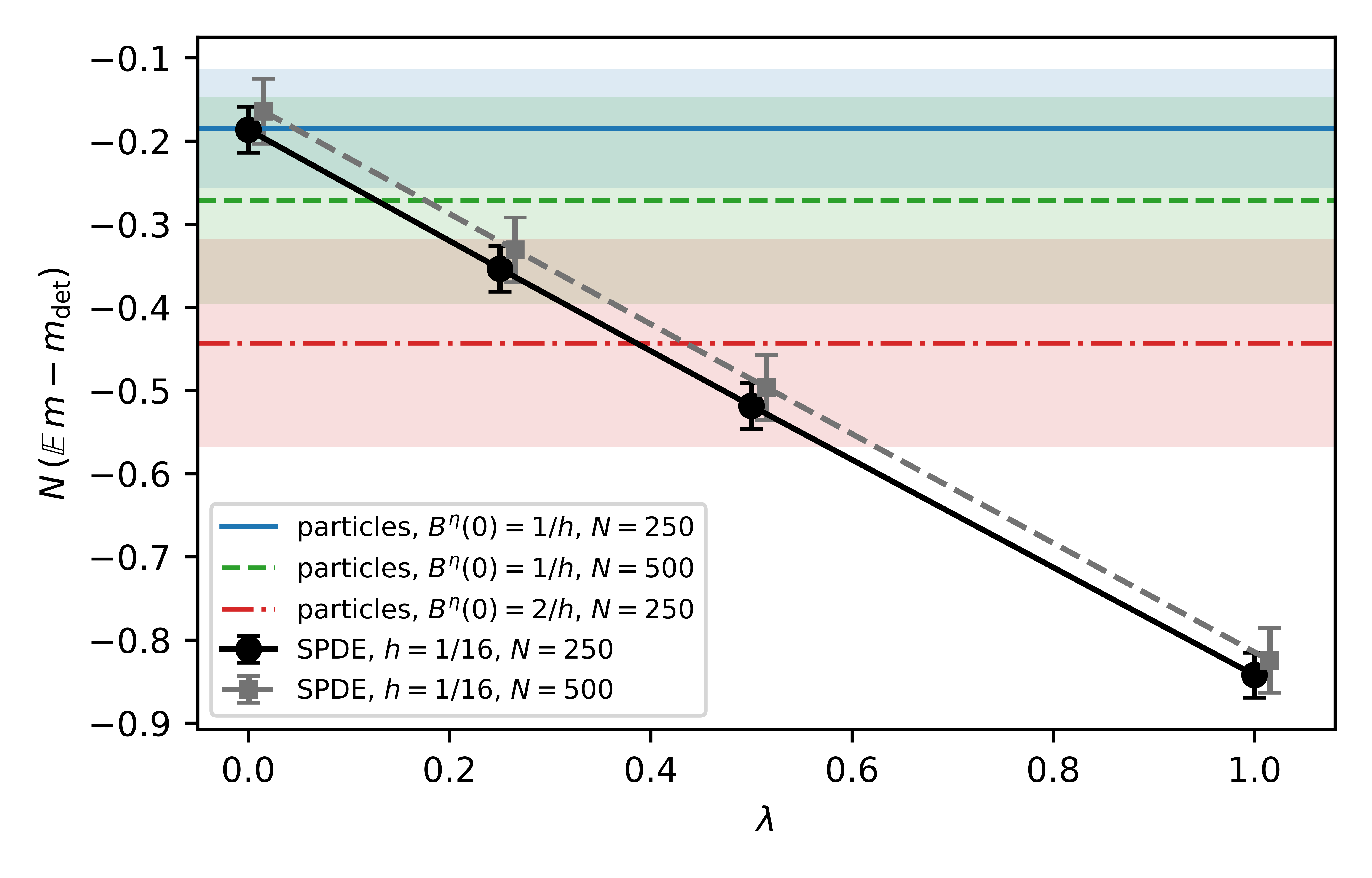}
\caption{Bias of the mean at order $N^{-1}$, $N(\E m-m_{\rm det})$. Points: discretized SPDE with $h=1/16$ as a function of $\lambda$ (error bars $\pm2$ standard errors). Horizontal lines and bands ($\pm2$ standard errors): particle system with interaction range matched to the noise scale ($B^\eta(0)=1/h$) and with half that range ($B^\eta(0)=2/h$).}\label{fig:bias}
\end{figure}

\emph{Which $\lambda$ does the particle system select?} The particle system determines the $O(N^{-1})$ bias of the mean, and Figure~\ref{fig:bias} compares it with the SPDE.
\begin{itemize}
\item \emph{Matched scales.} If the interaction range is matched to the noise scale, $B^\eta(0)=1/h$, the particle bias is $-0.185\pm0.036$ ($N=250$) and $-0.271\pm0.062$ ($N=500$). The It\^o equation, $\lambda=0$, gives $-0.186\pm0.014$ and $-0.164\pm0.02$. The Stratonovich equation, $\lambda=\frac12$, gives $-0.52$, which is excluded.
\item \emph{Unmatched scales.} If the noise is correlated over twice the interaction range, $B^\eta(0)=2/h$, the particle bias is $-0.443\pm0.063$. This corresponds to $\lambda=0.39\pm0.10$. Consistently, the It\^o equation discretized at the interaction scale ($h=1/32$, $N=1000$) gives $-0.394\pm0.04$.
\end{itemize}
The mechanism is the following. In the particle system, $\Theta_i$ contains the self-interaction $\frac{a_{ii}}NB^\eta(0)$. In the It\^o SPDE, the corresponding term is produced by the fluctuations of the cell averages, since $\E[u_i\Ared_i(u)]$ contains $a_{ii}\mathrm{Var}(u_{i,c})\approx a_{ii}u_i/(Nh)$. The two agree when $B^\eta(0)=1/h$. The correction $\frac\lambda N\mathcal T$ is of the same order $\frac1{Nh}$, since $F_1=2/h$. When the noise is smoothed over a longer range than the interaction, it compensates for the missing part of the self-interaction.

In this sense, the $\lambda$-modified correction acts as a renormalization of the coarse-grained description. Its value is fixed by the ratio of the noise scale to the interaction range, not by the particle system alone. A value $\lambda>\frac12$ corresponds to noise that is correlated over a range noticeably longer than the interaction. This is the natural situation in the moderately interacting regime $\eta=\eta(N)\to0$, in which the noise of a well-posed SPDE description has to remain regular. Since $\mathcal T$ and the self-interaction term have different structures, a single $\lambda$ cannot match all observables exactly. These experiments are one-dimensional, use one observable and short times, and depend on the discretization of the noise. They illustrate, but do not prove, the interpretation above.

\section{Conclusion and outlook}\label{sec.outlook}

We have shown that the stochastic SKT system \eqref{eq:SKT_intro}--\eqref{eq:SKT_intro_diffusion_sigma}, with a multiplicative, spatially colored noise motivated by the fluctuations of an underlying particle system, admits a global, nonnegative martingale solution under a detailed-balance condition and a smallness condition on the noise intensity $1/N$. The proof extends the entropy-based regularization technique of \citet{braukhoff2024global} to a noise structure that only partially respects the entropy (gradient-flow) structure of the deterministic system, the central new technical difficulty being the $\lambda$-modified It\^o--Stratonovich correction term $\mathcal T$, whose entropy dissipation is not sign-definite. We close by collecting the main open questions raised by this work.

\begin{itemize}
\item \emph{The Stratonovich equation.} Section~\ref{sec.lambda} shows that, in general, the entropy estimate of this paper cannot reach $\lambda\le\frac12$, and that Theorem~\ref{thm:existence_SKT} reaches $\lambda_N=\frac12+O(1/N)$. For the Dean--Kawasaki equation the Stratonovich equation with correlated noise is well posed \citep{fehrman2024well,fehrman2025nonstationary}, by a kinetic formulation and a regularization of $\log u$. For the SKT noise, the cross term in \eqref{6.strat} is the new difficulty. Whether solutions exist at $\lambda=\frac12$, and whether they are more singular than for $\lambda>\frac12$, is open. A natural first step is an estimate for $\log u_i$ that uses the cross-diffusion structure of the deterministic part.

\item \emph{Dependence on the number of species $n$.} Assumption~\ref{Assumption:A4_correction_factor} becomes more restrictive as $n$ grows: for fixed noise coefficients $(e_k^{il})$, the smallness condition on $1/N$ required for our existence theorem scales unfavorably with $n$ (Remark following Assumption~\ref{Assumption:A4_correction_factor}). It would be of interest to determine whether this $n$-dependence is sharp, or whether a more careful bookkeeping of the cross-diffusion terms in the entropy estimate could yield a smallness condition uniform in $n$, which would be more faithful to the many-species regime the model is intended to describe.

\item \emph{Rigorous derivation of the noise from the particle system.} As discussed in Remark~\ref{rem:particle_motivation}, the noise coefficient \eqref{eq:SKT_intro_diffusion_sigma} and the correction term $\mathcal T$ are motivated by, but not rigorously derived from, the fluctuations of the underlying moderately interacting particle system of \citet{chen_holzinger_2021_rigorous_SKT}. Completing this derivation , identifying the fluctuation martingale as a stochastic integral against a colored Wiener process on a suitable weighted Sobolev scale, and controlling the mean-field convergence rate, is open. A linear fluctuation limit, on $\R^d$ and with space-time white noise, will be proved in forthcoming work. Connecting it with \eqref{eq:SKT_intro}, with colored noise and on bounded domains,would place the present existence theory on a fully rigorous footing relative to the underlying particle system.

\item \emph{Uniqueness.} We do not address uniqueness, in law or pathwise. This is not specific to the stochastic setting: for cross-diffusion systems, uniqueness of weak solutions is in general not known, even without noise. The entropy method provides existence, but not the stability estimate needed for uniqueness, because the diffusion matrix is neither symmetric nor positive definite and there is no comparison principle. Uniqueness of weak solutions has been proved only under additional structure or regularity, for instance for two species with self-diffusion in dimension $d\le4$, for solutions in $L^\infty(0,T;H^1(\dom))$ \citep{pham_temam_uniqueness}, or for systems whose coefficients depend on the species only through a common total density \citep{chen_juengel_note_uniqueness}. For general SKT-type systems only weak--strong uniqueness is known: a weak or renormalized solution coincides with a strong solution emanating from the same initial data, as long as the latter exists \citep{chen_juengel_weak_strong,geltner_heitzinger_juengel_2026}. A stochastic analogue would require a strong solution of \eqref{eq:SKT_intro}, which is not available. A kinetic formulation in the spirit of \citet{fehrman2023non,fehrman2024well}, adapted to systems, is a possible route. Higher regularity of solutions is likewise not pursued here.

\item \emph{Other cross-diffusion systems.} It is natural to ask whether the entropy-based regularization approach of this paper extends to other cross-diffusion systems possessing a comparable entropy structure. We are currently investigating this question for the Maxwell--Stefan cross-diffusion system. Unlike for the SKT system, the rigorous derivation of a stochastic correction from an underlying particle system is, to our knowledge, not available for the Maxwell--Stefan system, so that at present only a heuristic noise term can be formulated; moreover, the volume-filling structure typical of Maxwell--Stefan models is expected to introduce additional technical difficulties not present here.
\end{itemize}


\appendix
\section{Proofs of technical lemmata}\label{app:proofs_of_lemmata}
\begin{proof}[Proof of Lemma \ref{lem:is_correction_entropy_estimate}]\label{proof:is_correction_entropy_estimate}
\emph{Step 1: expanding the correction term.} Write $\sigma'(x):=\frac{\dd}{\dd x}\sigma(x)$ for $\sigma(x)=\sqrt x$, $x>0$, so that $\sigma'(x)=\frac1{2\sqrt x}$ and, by the chain rule,
\[
(\partial_{u_i}\sigma)(u^\eps)_{ii}=\sigma'(u_i^\eps\Ared_i(u^\eps))\big(\Ared_i(u^\eps)+a_{ii}u_i^\eps\big).
\]
and
\[
\partial_{x_l}(\sigma(u^\eps))_{ii}=\sigma'(u_i^\eps\Ared_i(u^\eps))\big(\Ared_i(u^\eps)\partial_{x_l}u_i^\eps+u_i^\eps\partial_{x_l}\Ared_i(u^\eps)\big).
\]
Substituting into \eqref{eqn:IS_correction}, and using that pairing with $R_\eps(v_i)=\pi_i\log u_i^\eps$ contributes a factor $\partial_{x_l}R_\eps(v_i)=\pi_i\frac{\partial_{x_l}u_i^\eps}{u_i^\eps}$ to both brackets (via integration by parts in $x_l$. The boundary terms vanish by \eqref{eq:noise_bc}), together with $\sigma'(x)^2=\frac1{4x}$, $\sigma'(x)\sigma(x)=\frac12$,
\begin{align*}
\lambda\int_0^t\langle\mathcal T(u^\eps),R_\eps\rangle_{L^2(\dom)}\ds
&=-\lambda\sum_{i=1}^{n}\pi_i\int_0^t\int_\dom\sum_{k,l=1}^{\infty,d}\frac{\Ared_i(u^\eps)+a_{ii}u_i^\eps}{4u_i^\eps\Ared_i(u^\eps)}\\
&\phantom{xxxx}\times\big(\Ared_i(u^\eps)\partial_{x_l}u_i^\eps+u_i^\eps\partial_{x_l}\Ared_i(u^\eps)\big)\frac{\partial_{x_l}u_i^\eps}{u_i^\eps}(e_k^{il})^2\dx\ds\\
&\phantom{xx}-\frac\lambda2\sum_{i=1}^{n}\pi_i\int_0^t\int_\dom\sum_{k,l=1}^{\infty,d}\big(\Ared_i(u^\eps)+a_{ii}u_i^\eps\big)e_k^{il}\partial_{x_l}e_k^{il}\,\frac{\partial_{x_l}u_i^\eps}{u_i^\eps}\dx\ds.
\end{align*}
Splitting $\Ared_i(u^\eps)+a_{ii}u_i^\eps$ into its two summands in each line gives four terms,
\begin{align}
\mathrm{IS}_1 &:= -\frac\lambda4\sum_{i}\pi_{i}\int_{0}^{t}\int_{\dom}\sum_{k,l=1}^{\infty,d}\Ared_{i}(u^{\eps})\left(\frac{\partial_{x_{l}}u^{\eps}_{i}}{u^{\eps}_{i}}\right)^{2}\left(e^{il}_{k}\right)^{2}\dx \ds, \label{eqn:IS1_final}\\
\mathrm{IS}_2 &:= -\frac\lambda4\sum_{i=1}^{n}\pi_i a_{ii}\int_0^t\int_\dom\sum_{k,l=1}^{\infty,d}\frac{(\partial_{x_l}u_i^\eps)^2}{u_i^\eps}(e_k^{il})^2\dx\ds\\
&\phantom{\mathrm{IS}_2 :={}}=-\lambda\sum_{i=1}^{n}\pi_i a_{ii}\int_0^t\int_\dom\sum_{k,l=1}^{\infty,d}\big(\partial_{x_l}\sqrt{u_i^\eps}\big)^2(e_k^{il})^2\dx\ds, \label{eqn:IS2_final}\\
\mathrm{IS}_3 &:= -\frac\lambda4\sum_{i}\pi_{i}\int_{0}^{t}\int_{\dom}\sum_{k,l=1}^{\infty,d}\frac{\Ared_{i}(u^{\eps})+a_{ii}u^{\eps}_{i}}{u^{\eps}_{i}\Ared_{i}(u^{\eps})}\partial_{x_{l}}\Ared_{i}(u^{\eps})\,\partial_{x_{l}}u^{\eps}_{i}\left(e^{il}_{k}\right)^{2}\dx \ds, \label{eqn:IS3_final}\\
\mathrm{IS}_4 &:= -\frac\lambda2\sum_{i}\pi_{i}\int_{0}^{t}\int_{\dom}\sum_{k,l=1}^{\infty,d}\left(\Ared_{i}(u^{\eps})+a_{ii}u^{\eps}_{i}\right)\partial_{x_{l}}e^{il}_{k}\,e^{il}_{k}\,\frac{\partial_{x_{l}}u^{\eps}_{i}}{u^{\eps}_{i}}\dx \ds, \label{eqn:IS4_final}
\end{align}
with $\lambda\int_0^t\langle\mathcal T(u^\eps),R_\eps\rangle_{L^2(\dom)}\ds=\mathrm{IS}_1+\mathrm{IS}_2+\mathrm{IS}_3+\mathrm{IS}_4$, using $\frac14\frac{(\partial_{x_l}u_i^\eps)^2}{u_i^\eps}=\big(\partial_{x_l}\sqrt{u_i^\eps}\big)^2$ in \eqref{eqn:IS2_final}.

\emph{Step 2: expanding the It\^o correction.} Write $A=A_1+A_2$, with $A_1=g_\delta'(u_i^\eps\Ared_i(u^\eps))\Ared_i(u^\eps)\partial_{x_l}u_i^\eps\,e_k^{il}$, $A_2=g_\delta'(u_i^\eps\Ared_i(u^\eps))u_i^\eps\partial_{x_l}\Ared_i(u^\eps)\,e_k^{il}$, and $B=g_\delta(u_i^\eps\Ared_i(u^\eps))\partial_{x_l}e_k^{il}$, so that $\operatorname{IC}_{i,k,l}=\pi_i(A+B)^2/u_i^\eps=\pi_i(A_1^2+A_2^2+B^2+2A_1A_2+2A_1B+2A_2B)/u_i^\eps$. Integrating,
\[
\frac12\int_0^t\int_\dom\sum_{i,k,l=1}^{n,\infty,d}\operatorname{IC}_{i,k,l}\dx\ds=\operatorname{IC}_1+\operatorname{IC}_2+\operatorname{IC}_3+\operatorname{IC}_4+\operatorname{IC}_5+\operatorname{IC}_6,
\]
where
\begin{align}
\mathrm{IC}_1 &:= \frac{1}{2}\sum_{i}\pi_{i}\int_{0}^{t}\int_{\dom}\sum_{k,l=1}^{\infty,d}
g_{\delta}'(u^{\eps}_{i}\Ared_{i}(u^{\eps}))^{2}
\Ared_{i}(u^{\eps})^{2}u^{\eps}_{i}
\left(\frac{\partial_{x_{l}}u^{\eps}_{i}}{u^{\eps}_{i}}\right)^{2}\left(e^{il}_{k}\right)^{2} \dx\ds, \label{eqn:IC1_final}\\
\mathrm{IC}_2 &:= \sum_{i}\pi_{i}\int_{0}^{t}\int_{\dom}\sum_{k,l=1}^{\infty,d}
g_{\delta}'(u^{\eps}_{i}\Ared_{i}(u^{\eps}))^{2}
\Ared_{i}(u^{\eps})\,\partial_{x_{l}}u^{\eps}_{i}\cdot\partial_{x_{l}}\Ared_{i}(u^{\eps})\left(e^{il}_{k}\right)^{2}\dx\ds, \label{eqn:IC2_final}\\
\mathrm{IC}_3 &:= \frac{1}{2}\sum_{i}\pi_{i}\int_{0}^{t}\int_{\dom}\sum_{k,l=1}^{\infty,d}
g_{\delta}'(u^{\eps}_{i}\Ared_{i}(u^{\eps}))^{2}
u^{\eps}_{i}\left(\partial_{x_{l}}\Ared_{i}(u^{\eps})\right)^{2}\left(e^{il}_{k}\right)^{2} \dx\ds, \label{eqn:IC3_final}\\
\mathrm{IC}_4 &:= \sum_{i}\pi_{i}\int_{0}^{t}\int_{\dom}\sum_{k,l=1}^{\infty,d}
g_{\delta}'(u^{\eps}_{i}\Ared_{i}(u^{\eps}))\,g_{\delta}(u^{\eps}_{i}\Ared_{i}(u^{\eps}))
\Ared_{i}(u^{\eps})\frac{\partial_{x_{l}}u^{\eps}_{i}}{u^{\eps}_{i}}
e^{il}_{k}\partial_{x_{l}}e^{il}_{k}\dx\ds, \label{eqn:IC4_final}\\
\mathrm{IC}_5 &:= \sum_{i}\pi_{i}\int_{0}^{t}\int_{\dom}\sum_{k,l=1}^{\infty,d}
g_{\delta}'(u^{\eps}_{i}\Ared_{i}(u^{\eps}))\,g_{\delta}(u^{\eps}_{i}\Ared_{i}(u^{\eps}))
\partial_{x_{l}}\Ared_{i}(u^{\eps})\,e^{il}_{k}\partial_{x_{l}}e^{il}_{k}\dx\ds, \label{eqn:IC5_final}\\
\mathrm{IC}_6 &:= \frac{1}{2}\sum_{i}\pi_{i}\int_{0}^{t}\int_{\dom}\sum_{k,l=1}^{\infty,d}
g_{\delta}(u^{\eps}_{i}\Ared_{i}(u^{\eps}))^{2}
\left(\partial_{x_{l}}e^{il}_{k}\right)^{2}
\frac{1}{u^{\eps}_{i}}\dx\ds. \label{eqn:IC6_final}
\end{align}
Here $\mathrm{IC}_1,\mathrm{IC}_3,\mathrm{IC}_6\geq0$, being (up to the nonnegative weight $\pi_i/u_i^\eps$) the squares $A_1^2,A_2^2,B^2$. The cross terms $\mathrm{IC}_2,\mathrm{IC}_4,\mathrm{IC}_5$ carry no definite sign and are treated via Young's inequality below.

\emph{Step 3: pairing $\mathrm{IS}_1$ with $\mathrm{IC}_1$.} By \eqref{eqn:sqrt_approximation_properties2}, $g_\delta'(x)^2x<\frac14$ for every $x>0$, so
\[
\mathrm{IC}_1<\frac18\sum_{i,k,l=1}^{n,\infty,d}\pi_i\int_0^t\int_\dom\Ared_i(u^\eps)\Big(\frac{\partial_{x_l}u_i^\eps}{u_i^\eps}\Big)^2(e_k^{il})^2\dx\ds=-\frac1{2\lambda}\mathrm{IS}_1,
\]
whence
\begin{equation}\label{eqn:IS1_IC1_final2}
\mathrm{IS}_1+\mathrm{IC}_1<-\Big(\frac\lambda4-\frac18\Big)\sum_{i,k,l=1}^{n,\infty,d}\pi_i\int_0^t\int_\dom\Ared_i(u^\eps)\Big(\frac{\partial_{x_l}u_i^\eps}{u_i^\eps}\Big)^2(e_k^{il})^2\dx\ds,
\end{equation}
which is negative precisely when $\lambda>\frac12$.

\emph{Step 4: $\mathrm{IS}_2$.} By \eqref{eqn:IS2_final}, $\mathrm{IS}_2$ is already sign-definite negative and matches the corresponding term in the statement of the lemma; no further estimate is needed.

\emph{Step 5: pairing $\mathrm{IS}_3$ with $\mathrm{IC}_2$.} Since $0\leq a_{ii}u_i^\eps\leq\Ared_i(u^\eps)$, so that $\Ared_i(u^\eps)\leq\Ared_i(u^\eps)+a_{ii}u_i^\eps\leq2\Ared_i(u^\eps)$, and $g_\delta'(x)^2\leq\frac1{4x}$,
\begin{align*}
|\mathrm{IS}_3|&\leq\frac\lambda2\sum_{i=1}^{n}\pi_i\int_0^t\int_\dom\sum_{k,l=1}^{\infty,d}\Big|\frac{\partial_{x_l}u_i^\eps}{u_i^\eps}\Big|\,|\partial_{x_l}\Ared_i(u^\eps)|(e_k^{il})^2\dx\ds,\\
\mathrm{IC}_2&\leq\frac14\sum_{i=1}^{n}\pi_i\int_0^t\int_\dom\sum_{k,l=1}^{\infty,d}\Big|\frac{\partial_{x_l}u_i^\eps}{u_i^\eps}\Big|\,|\partial_{x_l}\Ared_i(u^\eps)|(e_k^{il})^2\dx\ds.
\end{align*}

Young's inequality $ab\leq\frac{\kappa}2a^2+\frac1{2\kappa}b^2$, applied with $a=\sqrt{\Ared_i(u^\eps)}\,\frac{\partial_{x_l}u_i^\eps}{u_i^\eps}$, $b=\frac{\partial_{x_l}\Ared_i(u^\eps)}{\sqrt{\Ared_i(u^\eps)}}$ (each multiplied by $|e_k^{il}|$), then gives
\begin{equation}\label{eqn:IS3_IC2_final2}
\begin{split}
\mathrm{IS}_3+\mathrm{IC}_2&\leq\Big(\frac\lambda2+\frac14\Big)\Big[\frac{\kappa}2\sum_{i,k,l=1}^{n,\infty,d}\pi_i\int_0^t\int_\dom\Ared_i(u^\eps)\Big(\frac{\partial_{x_l}u_i^\eps}{u_i^\eps}\Big)^2(e_k^{il})^2\dx\ds\\
&\phantom{xxxxxx}+\frac1{2\kappa}\sum_{i,k,l=1}^{n,\infty,d}\pi_i\int_0^t\int_\dom\frac{(\partial_{x_l}\Ared_i(u^\eps))^2}{\Ared_i(u^\eps)}(e_k^{il})^2\dx\ds\Big].
\end{split}
\end{equation}

\emph{Step 6: splitting and pairing $\mathrm{IS}_4$ with $\mathrm{IC}_4$ and $\mathrm{IC}_5$.} Write $\mathrm{IS}_4=\mathrm{IS}_{4a}+\mathrm{IS}_{4b}$, splitting $\Ared_i(u^\eps)+a_{ii}u_i^\eps$ in \eqref{eqn:IS4_final} into its two summands:
\begin{align*}
 \mathrm{IS}_{4a}&:=-\frac\lambda2\sum_{i=1}^{n}\pi_i\int_0^t\int_\dom\sum_{k,l=1}^{\infty,d}\Ared_i(u^\eps)\,\partial_{x_l}e_k^{il}\,e_k^{il}\,\frac{\partial_{x_l}u_i^\eps}{u_i^\eps}\dx\ds,\\
 \mathrm{IS}_{4b}&:=-\frac\lambda2\sum_{i=1}^{n}\pi_ia_{ii}\int_0^t\int_\dom\sum_{k,l=1}^{\infty,d}\partial_{x_l}u_i^\eps\,e_k^{il}\,\partial_{x_l}e_k^{il}\dx\ds.
\end{align*}
By \eqref{eqn:sqrt_approximation_properties2}, $g_\delta'(x)g_\delta(x)<\frac12$, so
\[
|\mathrm{IS}_{4a}|+|\mathrm{IC}_4|\leq\Big(\frac\lambda2+\frac12\Big)\sum_{i=1}^{n}\pi_i\int_0^t\int_\dom\sum_{k,l=1}^{\infty,d}\Ared_i(u^\eps)\Big|\frac{\partial_{x_l}u_i^\eps}{u_i^\eps}\Big|\,|e_k^{il}\partial_{x_l}e_k^{il}|\dx\ds,
\]
and Young's inequality with parameter $\kappa$ gives
\begin{align}\label{eqn:IS4a_IC4_final2}
    \mathrm{IS}_{4a}+\mathrm{IC}_4&\leq\frac{\lambda+1}2\Big[\frac{\kappa}2\sum_{i,k,l=1}^{n,\infty,d}\pi_i\int_0^t\int_\dom\Ared_i(u^\eps)\Big(\frac{\partial_{x_l}u_i^\eps}{u_i^\eps}\Big)^2(e_k^{il})^2\dx\ds\Big]\\
    &\phantom{xx}+\frac{\lambda+1}2\Big[\frac1{2\kappa}\sum_{i,k,l=1}^{n,\infty,d}\pi_i\int_0^t\int_\dom\Ared_i(u^\eps)(\partial_{x_l}e_k^{il})^2\dx\ds\Big].\nonumber
\end{align}
We bound the integrals of $\Ared_i(u^\eps)$ by the entropy. Let $g(x):=x\log x-x+2\ge1$. Since $x\le(x\log x-x+1)+(e-1)$ for $x\ge0$, we have $x\le(e-1)g(x)$. By the detailed-balance condition, $\sum_i\pi_i\Ared_i(u)=\sum_j\pi_j\big(a_{j0}+\sum_ia_{ji}u_j\big)\le C_a\sum_j\pi_jg(u_j)$, with $C_a$ as in the statement. Together with $\sum_k\|\partial_{x_l}e^{il}_k\|_{L^\infty}^2\le E_\infty'$, this gives
\begin{equation}\label{eqn:Ared_entropy_domination_final}
\sum_{i,k,l=1}^{n,\infty,d}\pi_i\int_0^t\int_\dom\Ared_i(u^\eps)(\partial_{x_l}e_k^{il})^2\dx\ds\leq C_a\,d\,E_\infty'\sum_{i=1}^{n}\pi_i\int_0^t\int_\dom\big(u_i^\eps\log u_i^\eps-u_i^\eps+2\big)\dx\ds.
\end{equation}
 By Young's inequality with parameter $\kappa$,
\begin{equation}\label{eqn:IS4b_final2}
\mathrm{IS}_{4b}\leq\frac\lambda2\sum_{i=1}^{n}\pi_ia_{ii}\int_0^t\int_\dom\sum_{k,l=1}^{\infty,d}\Big[\frac{\kappa}2(\partial_{x_l}u_i^\eps)^2(e_k^{il})^2+\frac1{2\kappa}(\partial_{x_l}e_k^{il})^2\Big]\dx\ds,
\end{equation}
and, again by \eqref{eqn:sqrt_approximation_properties2} and Young's inequality with parameter $\widetilde\kappa_3\in(0,\tfrac12]$,
\begin{equation}\label{eqn:IC5_final_bound2}
\mathrm{IC}_5\leq\frac12\sum_{i=1}^{n}\pi_i\int_0^t\int_\dom\sum_{k,l=1}^{\infty,d}\Big[\frac{\widetilde\kappa_3}2(\partial_{x_l}\Ared_i(u^\eps))^2(e_k^{il})^2+\frac1{2\widetilde\kappa_3}(\partial_{x_l}e_k^{il})^2\Big]\dx\ds.
\end{equation}

\emph{Step 7: $\mathrm{IC}_3$ and $\mathrm{IC}_6$.} By $g_\delta'(x)^2\leq\frac1{4x}$,
\begin{equation}\label{eqn:IC3_final_bound2}
\mathrm{IC}_3\leq\frac18\sum_{i,k,l=1}^{n,\infty,d}\pi_i\int_0^t\int_\dom\frac{(\partial_{x_l}\Ared_i(u^\eps))^2}{\Ared_i(u^\eps)}(e_k^{il})^2\dx\ds,
\end{equation}
and, by $g_\delta(x)\leq\sqrt x$ (\eqref{eqn:sqrt_approximation_properties}) together with \eqref{eqn:Ared_entropy_domination_final},
\begin{align}\label{eqn:IC6_final_bound2}
    \mathrm{IC}_6&\leq\frac12\sum_{i=1}^{n}\pi_i\int_0^t\int_\dom\sum_{k,l=1}^{\infty,d}\Ared_i(u^\eps)(\partial_{x_l}e_k^{il})^2\dx\ds\\
    &\leq\frac{C_a}2\,d\,E_\infty'\sum_{i=1}^{n}\pi_i\int_0^t\int_\dom\big(u_i^\eps\log u_i^\eps-u_i^\eps+2\big)\dx\ds.\nonumber
\end{align}

\emph{Step 8: conclusion.} Summing \eqref{eqn:IS1_IC1_final2}, $\mathrm{IS}_2$, \eqref{eqn:IS3_IC2_final2}, \eqref{eqn:IS4a_IC4_final2}--\eqref{eqn:Ared_entropy_domination_final}, \eqref{eqn:IS4b_final2}, \eqref{eqn:IC5_final_bound2}, \eqref{eqn:IC3_final_bound2}, and \eqref{eqn:IC6_final_bound2}, gives the bound stated in the lemma, with leading coefficient
\[
\frac\lambda4-\frac18-\Big(\frac\lambda2+\frac14\Big)\frac\kappa2-\frac{\lambda+1}2\cdot\frac\kappa2=\frac\lambda4-\frac18-\frac{\lambda\kappa}2-\frac{3\kappa}8,
\]
positive whenever $\lambda>\frac12$ and $\kappa<\frac{2\lambda-1}{4\lambda+3}$; coefficient $\big(\frac\lambda2+\frac14\big)\frac1{2\kappa}+\frac18=\frac{2\lambda+1}{8\kappa}+\frac18$ on $(\partial_{x_l}\Ared_i(u^\eps))^2/\Ared_i(u^\eps)$; coefficient $\frac{\widetilde\kappa_3}4$ on $(\partial_{x_l}\Ared_i(u^\eps))^2$; coefficient $\frac{\lambda\kappa}4$ on $a_{ii}(\partial_{x_l}u_i^\eps)^2$; coefficient $\frac{\lambda a_{ii}}{4\kappa}+\frac1{4\widetilde\kappa_3}$ on $\|\partial_{x_l}e_k^{il}\|^2_{L^2}$; and coefficient $C_a\big(\frac{\lambda+1}{4\kappa}+\frac12\big)$ on the entropy term (combining \eqref{eqn:Ared_entropy_domination_final} and \eqref{eqn:IC6_final_bound2}).
\end{proof}

\begin{remark}[On the $n$-dependence and the scope of the estimate]\label{rem:standalone_n_and_scope}
The term $\sum_{i,k,l=1}^{n,\infty,d}\pi_i\int_0^t\int_\dom(\partial_{x_l}\Ared_i(u^\eps))^2(e_k^{il})^2\dx\ds$ produced by $\mathrm{IC}_5$ in \eqref{eqn:IC5_final_bound2} is the only place in this proof where the number of species $n$ enters explicitly. Writing $\partial_{x_l}\Ared_i(u^\eps)=\sum_{j=1}^na_{ij}\partial_{x_l}u_j^\eps$ and applying Cauchy--Schwarz,
\[
\big(\partial_{x_l}\Ared_i(u^\eps)\big)^2\leq\Big(\sum_{j=1}^na_{ij}\Big)\Big(\sup_{1\leq m\leq n}a_{im}\Big)\sum_{j=1}^n(\partial_{x_l}u_j^\eps)^2,
\]
which becomes a genuine multiple of $n$ under the additional uniform bound $a_{ij}\leq\bar a$ ($1\leq i,j\leq n$). We leave \eqref{eqn:IC5_final_bound2} in the form above, rather than substitute this bound, so that either this or a sharper species-dependent estimate can be used as needed; using it in Proposition~\ref{prop.ent} requires Assumption~\ref{Assumption:A4_correction_factor} to control the corresponding term.
\end{remark}

\begin{proof}\label{proof:is_linear_growth_lipschitz} (of Lemma \ref{lem:is_linear_growth_lipschitz})
Here $\mathcal T(v)$ stands for $\mathcal T(u(R_\eps(v)))$, and we use $a_{i0}>0$. Let $w:=R_\eps(v)$, so that $u_i=\exp(w_i/\pi_i)$. By Lemma~\ref{lem.Reps}, $\|w\|_{D(L)}\le\|R_\eps(0)\|_{D(L)}+\frac C\eps\|v\|_{D(L)'}\le C(\eps,R)$, and $D(L)\hookrightarrow W^{1,\infty}(\dom)$ gives $\|w\|_{W^{1,\infty}(\dom)}\le C(\eps,R)$. We write the closed form \eqref{eqn:T_closed_form} in terms of $w$. Since $\partial_{x_l}\log u_i=\partial_{x_l}w_i/\pi_i$ and $\partial_{x_l}\Ared_i(u)=\sum_ja_{ij}u_j\partial_{x_l}w_j/\pi_j$,
\[
  \langle\mathcal T(v)_i,\phi\rangle=-\sum_{l=1}^d\int_\dom\Big[F^{il}_1(x)\,\Phi^{il}_1\big(w,\na w\big)+F^{il}_2(x)\,\Phi^{il}_2(w)\Big]\partial_{x_l}\phi\dx,
\]
where $\Phi^{il}_1(w,\na w):=\frac{\vartheta_i(u)}{4\pi_i}\partial_{x_l}w_i+\frac{\vartheta_i(u)}{4\Ared_i(u)}\sum_ja_{ij}\frac{u_j}{\pi_j}\partial_{x_l}w_j$ and $\Phi^{il}_2(w):=\frac12\vartheta_i(u)$, with $u=\exp(w/\pi)$. Both are smooth functions of $(w,\na w)\in\R^n\times\R^{n\times d}$, because $u\mapsto\vartheta_i(u)$ is affine and $\Ared_i(u)\ge a_{i0}>0$. On the compact set $K:=\{|w|,|\na w|\le C(\eps,R)\}$ they are bounded and Lipschitz continuous, with constants depending on $\eps$ and $R$. Since $F^{il}_1\le E_\infty$, $|F^{il}_2|\le(E_\infty E_\infty')^{1/2}$ and $\|\partial_{x_l}\phi\|_{L^\infty}\le C\|\phi\|_{D(L)}$, we obtain $\|\mathcal T(v)\|_{D(L)'}\le C_{R,\eps}$, which implies the first estimate. For $v^1,v^2$ in the ball, the Lipschitz bound on $K$ gives
\begin{align*}
  \|\mathcal T(v^1)-\mathcal T(v^2)\|_{D(L)'}&\le C_{R,\eps}\|w^1-w^2\|_{W^{1,1}(\dom)}\\
  &\le C_{R,\eps}\|R_\eps(v^1)-R_\eps(v^2)\|_{D(L)}\le C_{R,\eps}\|v^1-v^2\|_{D(L)'},
\end{align*}
by the Lipschitz continuity of $R_\eps$ (Lemma~\ref{lem.Reps}).
\end{proof}
\section{Auxiliary lemmata}

\begin{lemma}[{\citealp[Lemma 14]{braukhoff2024global}; \citealp[Lemma 2.7.3]{huber2022stochastic}}]\label{lem.ab}
Let $w\in D(L)$, $a=(a_{ij})\in L^1(\dom;\R^{n\times n})$, and 
$b=(b_{ij})\in D(L)^{n\times n}$ satisfying $\DD R_\eps[w](a) = b$.
Then
$$
	\int_\dom a:b\dx  
	\leq \int_\dom\operatorname{tr}[a^T u'(w)^{-1}a]\dx .
$$
\end{lemma}

\begin{lemma}[{\citealp[Lemma 15]{braukhoff2024global}; \citealp[Lemma 2.7.4]{huber2022stochastic}}]\label{lem.v0}
Let $v^0\in L^p(\Omega;L^1(\dom))$ for some $p\ge 1$ satisfy
$\E\int_\dom h(v^0)\dx \leq C$. Then 
$$
  \int_\dom h(u(R_\eps(v^0)))\dx  + \frac{\eps}{2}\|LR_\eps(v^0)\|_{L^2(\dom)}^2
	\leq \int_\dom h(v^0)\dx .
$$
\end{lemma}

 \begin{lemma}\label{lem.lpmult}
Let $1 \leq k_{1},k_{2},k_{3},l_{1},l_{2},l_{3} <\infty$ as well as $p,q,m \geq 1$. If $u \in L^{k_{3}}(\Omega;L^{k_{2}}(0,T;L^{k_{1}}(\dom)))$ and $v \in L^{l_{3}}(\Omega,L^{l_{2}}(0,T;L^{l_{1}}(\dom)))$, then $uv \in L^{m}(\Omega,L^{p}(0,T;L^{q}(\dom)))$, with $q=\frac{k_{1}l_{1}}{k_{1}+l_{1}}$, $p=\frac{k_{2}l_{2}}{k_{2}+l_{2}}$, $m=\frac{k_{3}l_{3}}{k_{3}+l_{3}}$.
\end{lemma}
\begin{proof}
The proof is a repeated application of H\"older's inequality: first in $x$ with the exponents $k_1/q$ and $l_1/q$, then in $t$ with $k_2/p$ and $l_2/p$, and finally in $\omega$ with $k_3/m$ and $l_3/m$. These exponents are conjugate, since $\frac q{k_1}+\frac q{l_1}=\frac p{k_2}+\frac p{l_2}=\frac m{k_3}+\frac m{l_3}=1$ by the definition of $q$, $p$ and $m$. This gives
\[
  \|uv\|_{L^m(\Omega;L^p(0,T;L^q(\dom)))}\le\|u\|_{L^{k_3}(\Omega;L^{k_2}(0,T;L^{k_1}(\dom)))}\,\|v\|_{L^{l_3}(\Omega;L^{l_2}(0,T;L^{l_1}(\dom)))}.
\]
\end{proof}

\subsection*{Acknowledgements}
Parts of this paper were completed during the author's affiliation with the \'Ecole Polytechnique F\'ed\'erale de Lausanne (EPFL), Switzerland.

\subsection*{Code availability}
The code for the numerical experiments is available at \url{https://github.com/Fluctuation-correction-stochastic-Shigesada-Kawasaki-Teramoto}.

\bibliographystyle{abbrvnat}
\bibliography{references}

\end{document}